\documentclass[envcountsame]{svmult}

\usepackage{amssymb}
\usepackage{amsmath}
\usepackage[matrix,arrow,curve,cmtip]{xy}
\usepackage{mathrsfs}

\usepackage{mathptmx}       
\usepackage{helvet}         
\usepackage{courier}        
\usepackage{type1cm}        

\usepackage{makeidx}         
\usepackage{graphicx}        
\usepackage{multicol}        
\usepackage[bottom]{footmisc}

\smartqed

\begin{document}

\title*{On the Whitehead spectrum of the circle}
\titlerunning{The Whitehead spectrum of the circle}
\author{Lars Hesselholt\thanks{Supported in part by a grant from the
    National Science Foundation}}
\authorrunning{Lars Hesselholt}
\institute{
Massachusetts Institute of Technology, Cambridge,
  Massachusetts \email{larsh@math.mit.edu} \and \at
Nagoya University, Nagoya, Japan
\email{larsh@math.nagoya-u.ac.jp}}


\maketitle

\abstract{The seminal work of Waldhausen, Farrell and Jones, Igusa,
and Weiss and Williams shows that the homotopy groups in low degrees
of the space of homeomorphisms of a closed Riemannian manifold of
negative sectional curvature can be expressed as a functor of the
fundamental group of the manifold. To determine this functor, however,
it remains to determine the homotopy groups of the topological
Whitehead spectrum of the circle. The cyclotomic trace of
B\"{o}kstedt, Hsiang, and Madsen and a theorem of Dundas, in turn, lead
to an expression for these homotopy groups in terms of the equivariant
homotopy groups of the homotopy fiber of the map from the topological
Hochschild $\mathbb{T}$-spectrum of the sphere spectrum to that of the
ring of integers induced by the Hurewicz map. We evaluate the latter
homotopy groups, and hence, the homotopy groups of the topological
Whitehead spectrum of the circle in low degrees. The result extends
earlier work by Anderson and Hsiang and by Igusa and complements recent
work by Grunewald, Klein, and Macko.}

\section*{Introduction}\label{intro}

Let $M$ be a closed smooth manifold of dimension $m \geqslant 5$.
Then, the stability theorem of Igusa~\cite{igusa} and a theorem of
Weiss and Williams~\cite[Thm.~A]{weisswilliams} show that, for all
integers $q$ less both $(m-4)/3$ and $(m-7)/2$, there is a long-exact
sequence
$$\dots \to \mathbb{H}_{q+2}(C_2,\tau_{\geqslant
  2}\operatorname{Wh}^{\text{Top}}(M)) \to
\pi_q(\operatorname{Homeo}(M)) \to
\pi_q(\widetilde{\operatorname{Homeo}}(M)) \to \dots$$
where the middle group is the $q$th homotopy group of the space of
homeomorphisms of $M$. In particular, the group
$\pi_0(\operatorname{Homeo}(M))$ is the mapping class group of
$M$. The right-hand term is the $q$th homotopy group of the space of
block homeomorphisms of $M$ and is the subject of surgery theory.
The left-hand term is the $(q+2)$th homotopy group of the Borel quotient
of the $2$-connective cover of the topological Whitehead spectrum of
$M$ by the canonical involution. It is one of the great past
achievements that the left-hand term can be expressed by Waldhausen's
algebraic $K$-theory of spaces~\cite{waldhausen, waldhausenjahrenrognes,
  vogell}. 

Suppose, in addition, that $M$ carries a Riemannian metric of
negative, but not necessarily constant, sectional curvature. Another
great achievement is the topological rigidity
theorems~\cite[Rem.~1.10,~Thm.~2.6]{farrelljones} of Farrell and Jones
which, in this case, give considerable simplifications of the left and
right-hand terms in the above sequence. For the right-hand term, there
are canonical isomorphisms
$$\pi_q(\widetilde{\operatorname{Homeo}}(M)) \xrightarrow{\sim} 
\pi_q(\widetilde{\operatorname{HoAut}}(M)) \xleftarrow{\sim}
\pi_q(\operatorname{HoAut}(M)),$$
where $\operatorname{HoAut}(M)$ and
$\widetilde{\operatorname{HoAut}}(M)$ are the spaces of self-homotopy
equivalences and block self-homotopy equivalences of $M$,
respectively. We note that, as $M$ is aspherical with $\pi_1(M)$
centerless~\cite[Thms.~22,~24]{petersen}, it follows
from~\cite[Thm.~III.2]{gottlieb} that the canonical map from
$\operatorname{HoAut}(M)$ to the discrete group
$\operatorname{Out}(\pi_1(M))$ is a weak equivalence. For the
left-hand term, there is a  canonical isomorphism
$$\bigoplus_{(C)} \operatorname{Wh}_q^{\text{Top}}(S^1)
\xrightarrow{\sim} \operatorname{Wh}_q^{\text{Top}}(M),$$
where the sum ranges over the set of conjugacy classes of maximal
cyclic subgroups of the torsion-free group $\pi_1(M)$; see
also~\cite[Thm.~139]{lueckreich}. Hence, in order to evaluate the
groups $\pi_q(\operatorname{Homeo}(M))$, it remains to evaluate
$$\operatorname{Wh}_q^{\text{Top}}(S^1) =
\pi_q(\operatorname{Wh}^{\text{Top}}(S^1))$$
and the canonical involution on these groups. We prove the
following result.

\begin{theorem}\label{main}The groups
  $\operatorname{Wh}_0^{\operatorname{Top}}(S^1)$ and 
$\operatorname{Wh}_1^{\operatorname{Top}}(S^1)$ are zero. Moreover,
there are canonical isomorphisms
$$\begin{aligned}
\operatorname{Wh}_2^{\operatorname{Top}}(S^1) & \xrightarrow{\sim}
\bigoplus_{r \geqslant 1} \bigoplus_{j \in \mathbb{Z} \smallsetminus
  2\mathbb{Z}} \mathbb{Z}/2\mathbb{Z} \cr
\operatorname{Wh}_3^{\operatorname{Top}}(S^1) & \xrightarrow{\sim}
\bigoplus_{r \geqslant 0} \bigoplus_{j \in \mathbb{Z} \smallsetminus
  2\mathbb{Z}} \mathbb{Z}/2\mathbb{Z} \; \oplus \;
\bigoplus_{r \geqslant 1} \bigoplus_{j \in \mathbb{Z} \smallsetminus
  2\mathbb{Z}} \mathbb{Z}/2\mathbb{Z}. \cr
\end{aligned}$$
\end{theorem}

The statement for $q = 0$ and $q = 1$ was proved earlier by Anderson
and Hsiang~\cite{andersonhsiang} by different methods. It was also
known by work of Igusa~\cite{igusa1} that the two sides of the statement
for $q = 2$ are abstractly isomorphic. The statement for $q = 3$ is
new. We also note that in recent work, Grunewald, Klein, and
Macko~\cite{grunewaldkleinmacko} have proved that for $p$ an odd prime
and $q \leqslant 4p-7$, the $p$-primary torsion subgroup of
$\operatorname{Wh}_q^{\operatorname{Top}}(S^1)$ is a countably
dimensional $\mathbb{F}_p$-vector space, if $q = 2p-2$ or $2p-1$, and
zero, otherwise. Hence, we will here focus the attention on the
$2$-primary torsion subgroup.

We briefly outline the proof of Thm.~\ref{main}. The seminal work of
Waldhausen establishes a cofibration sequence of spectra
$$S^1_+ \wedge K(\mathbb{S}) \xrightarrow{\alpha}
K(\mathbb{S}[x^{\pm 1}]) \to 
\operatorname{Wh}^{\text{Top}}(S^1) \xrightarrow{\partial}
\Sigma S^1_+ \wedge K(\mathbb{S}),$$
which identifies the topological Whitehead spectrum of the circle as
the mapping cone of the assembly map in algebraic
$K$-theory~\cite[Thm.~3.3.3]{waldhausen},~\cite[Thm.~0.1]{waldhausenjahrenrognes}. Here $\mathbb{S}$ is the  
sphere spectrum and $\mathbb{S}[x^{\pm 1}]$ is the Laurent polynomial
extension. If we replace the sphere spectrum by the ring of integers,
the assembly map
$$\alpha \colon S^1_+ \wedge K(\mathbb{Z}) \to
K(\mathbb{Z}[x^{\pm 1}])$$
becomes a weak equivalence by the fundamental theorem of algebraic
$K$-theory~\cite[Thm.~8, Cor.]{quillen}. Hence, we obtain a
cofibration sequence of spectra
$$S^1_+ \wedge K(\mathbb{S},I) \xrightarrow{\alpha}
K(\mathbb{S}[x^{\pm 1}],I[x^{\pm 1}]) \to 
\operatorname{Wh}^{\text{Top}}(S^1) \xrightarrow{\partial}
\Sigma S^1_+ \wedge K(\mathbb{S},I),$$
where the spectra $K(\mathbb{S},I)$ and
$K(\mathbb{S}[x^{\pm1}],I[x^{\pm 1}])$ are defined to be the homotopy
fibers of the maps of $K$-theory spectra induced by the Hurewicz maps
$\ell \colon \mathbb{S} \to \mathbb{Z}$ and $\ell \colon
\mathbb{S}[x^{\pm1}] \to \mathbb{Z}[x^{\pm1}]$, respectively. The
Hurewicz  maps are rational equivalences, as was proved by Serre. This
implies that $K(\mathbb{S},I)$ and
$K(\mathbb{S}[x^{\pm1}],I[x^{\pm1}])$ are rationally trivial
spectra. It follows that, for all integers $q$, 
$$\operatorname{Wh}_q^{\text{Top}}(S^1) \otimes \mathbb{Q} = 0.$$
Therefore, it suffices to evaluate, for every prime number $p$, the
homotopy groups with $p$-adic coefficients,
$$\operatorname{Wh}_q^{\text{Top}}(S^1;\mathbb{Z}_p) =
\pi_q(\operatorname{Wh}^{\text{Top}}(S^1)_p),$$
that are defined to be the homotopy groups of the
$p$-completion~\cite{bousfield}.

The cyclotomic trace map of
B\"{o}kstedt, Hsiang, and Madsen~\cite{bokstedthsiangmadsen} induces a
map
$$\operatorname{tr} \colon K(\mathbb{S},I) \to
\operatorname{TC}(\mathbb{S},I;p)$$
from the relative $K$-theory spectrum to the relative topological
cyclic homology spectrum. It was proved by Dundas~\cite{dundas} that
this map becomes a weak equivalence after $p$-completion. The same is true
for the Laurent polynomial extension. Hence, we have a cofibration
sequence of implicitly $p$-completed spectra
$$S^1_+ \wedge \operatorname{TC}(\mathbb{S},I;p)
\xrightarrow{\alpha}
\operatorname{TC}(\mathbb{S}[x^{\pm 1}],I[x^{\pm 1}];p) \to
\operatorname{Wh}^{\text{Top}}(S^1) \xrightarrow{\partial}
\Sigma S^1_+ \wedge \operatorname{TC}(\mathbb{S},I;p).$$
There is also a `fundamental theorem' for topological cyclic homology
which was proved by Madsen and the author in~\cite[Thm.~C]{hm3}. If
$A$ is a symmetric ring spectrum whose homotopy groups are
$\mathbb{Z}_{(p)}$-modules, this theorem expresses, up to an
extension, the topological cyclic homology groups 
$\smash{ \operatorname{TC}_q(A[x^{\pm1}];p) }$ of the Laurent
polynomial extension in terms of the equivariant homotopy groups
$$\operatorname{TR}_q^n(A;p) = [S^q \wedge
(\mathbb{T}/C_{p^{n-1}})_+, T(A)]_{\mathbb{T}}$$
of the topological Hochschild $\mathbb{T}$-spectrum $T(A)$ and the
maps
$$\begin{aligned}
R \colon & \operatorname{TR}_q^n(A;p) \to
\operatorname{TR}_q^{n-1}(A;p)
\hskip7mm \text{(restriction)} \hfill\space \cr
F \colon & \operatorname{TR}_q^n(A;p) \to
\operatorname{TR}_q^{n-1}(A;p)
\hskip7mm \text{(Frobenius)} \hfill\space \cr
V \colon & \operatorname{TR}_q^{n-1}(A;p) \to
\operatorname{TR}_q^n(A;p) 
\hskip7mm \text{(Verschiebung)} \hfill\space \cr
d \colon & \operatorname{TR}_q^n(A;p) \to
\operatorname{TR}_{q+1}^n(A;p) 
\hskip7mm \text{(Connes' operator)} \hfill\space \cr
\end{aligned}$$
which relate these groups. Here $\mathbb{T}$ is the multiplicative
group of complex numbers of modulus $1$, and $C_{p^{n-1}} \subset
\mathbb{T}$ is the subgroup of the indicated order. We recall the
groups $\operatorname{TR}_q^n(A;p)$ in Sect.~\ref{TRintro} and give a
detailed discussion of the fundamental theorem in
Sect.~\ref{fundamentaltheoremsection}. In the following
Sects.~\ref{TCintro} and~\ref{skeletonspectralsequence}, we briefly
recall the cyclotomic trace map and the skeleton spectral sequence
which we use extensively in later sections. A minor novelty here is
Prop.~\ref{fundamentalcofibrationsequence} which generalizes of the
fundamental long-exact sequence~\cite[Thm.~2.2]{hm} to a long-exact
sequence
$$\cdots \to \mathbb{H}_q(C_{p^m}, \mathit{TR}^{\,n}(A;p)) \to
\operatorname{TR}_q^{m+n}(A;p) \xrightarrow{R^n}
\operatorname{TR}_q^m(A;p) \to \cdots$$
valid for all positive integers $m$ and $n$.

The problem to evaluate
$\operatorname{Wh}_q^{\operatorname{Top}}(S^1)$ is thus reduced to the
homotopy theoretical problem of evaluating the equivariant homotopy
groups $\operatorname{TR}_q^n(\mathbb{S},I;p)$ along with the maps
listed above. In the paper~\cite{grunewaldkleinmacko} mentioned
earlier, the authors approximate the Hurewicz map $\ell \colon
\mathbb{S} \to \mathbb{Z}$ by a map of suspension spectra $\theta
\colon \mathbb{S}[SG] \to \mathbb{S}$ and use the Segal-tom Dieck
splitting to essentially evaluate the groups
$\operatorname{TR}_q^n(\mathbb{S},I;p)$, for $p$ odd and
$q \leqslant 4p-7$. However, this approach is not available, for
$q > 4p-7$, where a genuine understanding of the domain and target of
the map
$$\operatorname{TR}_q^n(\mathbb{S};p) \to
\operatorname{TR}_q^n(\mathbb{Z};p)$$
appears necessary. We evaluate
$\operatorname{TR}_q^n(\mathbb{S},I;2)$, for $q \leqslant 3$, and we
partly evaluate the four maps listed above. The result, which is
Thm.~\ref{TRrelative} below, is the main result of the paper, and the
proof occupies Sects.~\ref{spheresection}--\ref{relativesection}. The
homotopy  theoretical methods we employ here are perhaps somewhat
simple-minded and more sophisticated methods will certainly make it
possible to evaluate the groups
$\operatorname{TR}_q^n(\mathbb{S},I;p)$ in a wider range of
degrees. In particular, it would be very interesting to understand the
corresponding homology groups. However, to evaluate the groups
$\operatorname{TR}_q^n(\mathbb{S},I;p)$ is at least as difficult as to
evaluate the stable homotopy groups of spheres. In the final 
Sect.~\ref{algebrasection}, we apply the fundamental theorem to 
the result of Thm.~\ref{TRrelative} and prove Thm.~\ref{main}.

This paper was written in part during a visit to Aarhus University. It
is a pleasure to thank the university and Ib Madsen, in particular,
for their hospitality and support. Finally, the author would like to
express his gratitude to Marcel B\"{o}kstedt for the proof of
Lemma~\ref{bokstedtlemma} below. 

\section{The groups $\operatorname{TR}_q^n(A;p)$}\label{TRintro}

Let $A$ be a symmetric ring
spectrum~\cite[Sect.~5.3]{hoveyshipleysmith}. The topological
Hochschild $\mathbb{T}$-spectrum $T(A)$ is a cyclotomic spectrum in
the sense of~\cite[Def.~2.2]{hm}. In particular, it is an object of the
$\mathbb{T}$-stable homotopy category. Let $C_r \subset \mathbb{T}$ be
the subgroup of order $r$, and let $(\mathbb{T}/C_r)_+$ be the
suspension $\mathbb{T}$-spectrum of the union of $\mathbb{T}/C_r$ and
a disjoint basepoint. One defines the equivariant homotopy group
$$\operatorname{TR}_q^n(A;p) = [ S^q \wedge
(\mathbb{T}/C_{p^{n-1}})_+, T(A) ]_{\mathbb{T}}.$$
to be the abelian group of maps in the $\mathbb{T}$-stable homotopy
category between the indicated $\mathbb{T}$-spectra. The Frobenius
map, Verschiebung map, and Connes' operator, which we mentioned in the
Introduction, are induced by maps
$$\begin{aligned}
f \colon & (\mathbb{T}/C_{p^{n-2}})_+ \to (\mathbb{T}/C_{p^{n-1}})_+ \cr
v \colon & (\mathbb{T}/C_{p^{n-1}})_+ \to (\mathbb{T}/C_{p^{n-2}})_+ \cr
\delta \colon & \Sigma (\mathbb{T}/C_{p^{n-1}})_+ \to
(\mathbb{T}/C_{p^{n-1}})_+ \cr
\end{aligned}$$
in the $\mathbb{T}$-stable homotopy category defined as
follows. The map $f$ is the map of suspension $\mathbb{T}$-spectra
induced by the canonical projection  $\operatorname{pr} \colon
\mathbb{T}/C_{p^{n-2}} \to \mathbb{T}/C_{p^{n-1}}$, and the map $v$ is
the corresponding transfer map. To define the latter, we choose an
embedding $\iota \colon \mathbb{T}/C_{p^{n-2}} \hookrightarrow
\lambda$ into a finite dimensional orthogonal
$\mathbb{T}$-presentation. The product embedding
$(\iota,\operatorname{pr}) \colon \mathbb{T}/C_{p^{n-2}} \to \lambda
\times \mathbb{T}/C_{p^{n-1}}$ has trivial normal bundle, and the
linear structure of $\lambda$ determines a preferred
trivialization. Hence, the Pontryagin-Thom construction gives a map of
pointed $\mathbb{T}$-spaces
$$S^{\lambda} \wedge (\mathbb{T}/C_{p^{n-1}})_+ \to S^{\lambda} \wedge
(\mathbb{T}/C_{p^{n-2}})_+$$
and $v$ is the induced map of suspension
$\mathbb{T}$-spectra. Finally, there is a unique homotopy class of
maps of pointed spaces $\delta'' \colon S^1 \to
(\mathbb{T}/C_{p^{n-1}})_+$ such that image by the Hurewicz map is the
fundamental class $[\mathbb{T}/C_{p^{n-1}}]$ corresponding to the
counter-clockwise orientation of the circle $\mathbb{T} \subset
\mathbb{C}$ and such that the composite of $\delta''$ and the map
$(\mathbb{T}/C_{p^{n-1}})_+ \to S^0$ that collapses
$\mathbb{T}/C_{p^{n-1}}$ to the non-base point of $S^0$ is the
null-map. The map $\delta''$ induces the map of suspension
$\mathbb{T}/C_{p^{n-1}}$-spectra
$$\delta' \colon \Sigma (\mathbb{T}/C_{p^{n-1}})_+ \to
(\mathbb{T}/C_{p^{n-1}})_+$$
which, in turn, induces the map $\delta$.

The definition of the restriction map is more delicate. We let $E$ be
the unit sphere in $\mathbb{C}^{\infty}$ and consider the cofibration
sequence of pointed $\mathbb{T}$-spaces
$$E_+ \to S^0 \to \tilde{E} \to \Sigma E_+$$
where the left-hand map collapses $E$ onto the non-base point of
$S^0$; the $\mathbb{T}$-space $\tilde{E}$ is canonically homeomorphic
to the one-point compactification of $\mathbb{C}^{\infty}$. It induces
a cofibration sequence of $\mathbb{T}$-spectra
$$E_+ \wedge T(A) \to T(A) \to \tilde{E} \wedge T(A) \to
\Sigma E_+ \wedge T(A),$$
and hence, a long-exact sequence of equivariant homotopy
groups. By~\cite[Thm.~2.2]{hm}, the latter sequence is canonically
isomorphic to the sequence
$$\cdots \to \mathbb{H}_q(C_{p^{n-1}}, T(A)) \xrightarrow{N}
\operatorname{TR}_q^n(A;p) \xrightarrow{R}
\operatorname{TR}_q^{n-1}(A;p) \xrightarrow{\partial}
\mathbb{H}_{q-1}(C_{p^{n-1}}, T(A)) \to \cdots$$
which is called the fundamental long-exact sequence. The left-hand
term is the group homology of $C_{p^{n-1}}$ with coefficients in the
underlying $C_{p^{n-1}}$-spectrum of $T(A)$ and is defined to be the
equivariant homotopy group
$$\mathbb{H}_q(C_{p^{n-1}},T(A)) = [ S^q, E_+ \wedge T(A)
]_{C_{p^{n-1}}}.$$
The isomorphism of the left-hand terms in the two sequences is given
by the canonical change-of-groups isomorphism
$$[S^q, E_+ \wedge T(A)]_{C_{p^{n-1}}} \! \xrightarrow{\sim} \;
[S^q \wedge (\mathbb{T}/C_{p^{n-1}})_+, E_+ \wedge
T(A)]_{\mathbb{T}}$$
and the resulting map $N$ in the fundamental long-exact sequence is
called the norm map. The isomorphism of the right-hand terms in the
two sequences involves the cyclotomic structure of the spectrum $T(A)$
as we now explain. The $C_p$-fixed points of the $\mathbb{T}$-spectrum
$T(A)$ is a $\mathbb{T}/C_p$-spectrum $T(A)^{C_p}$. Moreover, the
isomorphism
$$\rho_p \colon \mathbb{T} \to \mathbb{T}/C_p$$
given by the $p$th root induces an equivalence of categories that to
the $\mathbb{T}/C_p$-spectrum $D$ associates the $\mathbb{T}$-spectrum
$\rho_p^*D$. Then the additional cyclotomic structure of the
topological Hochschild $\mathbb{T}$-spectrum $T(A)$ consists of a map
of $\mathbb{T}$-spectra
$$r \colon \rho_p^*((\tilde{E} \wedge T(A))^{C_p}) \to T(A)$$
with the property that the induced map of equivariant homotopy groups
$$[S^q \wedge (\mathbb{T}/C_{p^{n-1}})_+, \rho_p^*((\tilde{E} \wedge
T(A))^{C_p})]_{\mathbb{T}} \to 
[S^q \wedge (\mathbb{T}/C_{p^{n-1}})_+, T(A)]_{\mathbb{T}}$$
is an isomorphism for all positive integers $n$. The right-hand sides
of the two sequences above are now identified by the composition
$$\begin{aligned}
{} & [ S^q \wedge (\mathbb{T}/C_{p^{n-1}})_+, \tilde{E} \wedge T(A) ]_{\mathbb{T}} 
\xleftarrow{\sim}
[ S^q \wedge (\mathbb{T}/C_{p^{n-1}})_+, (\tilde{E} \wedge T(A))^{C_p} ]_{\mathbb{T}/C_p} \cr
{} & \xrightarrow{\sim} [ S^q \wedge (\mathbb{T}/C_{p^{n-2}})_+,
\rho_p^*((\tilde{E} \wedge T(A))^{C_p}) ]_{\mathbb{T}} \xrightarrow{\sim} 
[ S^q \wedge (\mathbb{T}/C_{p^{n-2}})_+, T(A) ]_{\mathbb{T}} \cr 
\end{aligned}$$
of the canonical isomorphism, the isomorphism $\rho_p^*$, and the
isomorphism induced by the map $r$. By definition, the restriction map
is the resulting map $R$ in the fundamental long-exact sequence. Since
$r$ is a map of $\mathbb{T}$-spectra, the restriction map commutes
with the Frobenius map, the Verschiebung map, and Connes' operator.

We mention that, if the symmetric ring spectrum $A$ is commutative,
then $T(A)$ has the structure of a commutative ring
$\mathbb{T}$-spectrum which, in turn, gives the graded abelian group
$\operatorname{TR}_*^n(A;p)$ the structure of an anti-symmetric graded
ring, for all $n \geqslant 1$. The restriction and Frobenius maps are
both ring homomorphisms, the Frobenius and Verschiebung maps satisfy
the projection formula
$$x V(y) = V(F(x)y),$$
and Connes' operator is a derivation with respect to the product. 

In general, the restriction map does not admit a section. However, if
$A = \mathbb{S}$ is the sphere spectrum, there exists a map
$$s \colon T(\mathbb{S}) \to \rho_p^*(T(\mathbb{S})^{C_p})$$
in the  $\mathbb{T}$-stable homotopy category such that the
composition
$$T(\mathbb{S}) \xrightarrow{s}
\rho_p^*(T(\mathbb{S})^{C_p}) \to
\rho_p^*((\tilde{E} \wedge T(\mathbb{S}))^{C_p}) \xrightarrow{r}
T(\mathbb{S})$$
is the identity map~\cite[Cor.~4.4.8]{madsen}. The map $s$ induces a
section
$$\hskip5mm S = S_n \colon \operatorname{TR}_q^{n-1}(\mathbb{S};p) \to
\operatorname{TR}_q^n(\mathbb{S};p) \hskip9.6mm
\text{(Segal-tom~Dieck splitting)}$$ 
of the restriction map. The section $S$ is a ring homomorphism and
commutes with the Verschiebung map and Connes' operator. The
composition $FS_n$ is equal to $S_{n-1}F$, for $n \geqslant 3$, and to
the identity map, for $n = 2$. It follows that, for every symmetric
ring spectrum $A$, the graded abelian group
$\operatorname{TR}_*^n(A;p)$ is a graded module over the graded ring
$\operatorname{TR}_*^1(\mathbb{S};p)$ which is canonically isomorphic
to the graded ring given by the stable homotopy groups of spheres. It
is proved in~\cite[Sect.~1]{h} that Connes' operator satisfies the
following additional relations 
$$\begin{aligned}
FdV & = d + (p-1)\eta, \cr
dd &  = d\eta = \eta d, \cr
\end{aligned}$$
where $\eta$ indicates multiplication by the Hopf class $\eta \in
\operatorname{TR}_1^1(\mathbb{S};p)$. It follows from the above that
$FV = p$, $dF = pFd$, and $Vd = pdV$.

The zeroth space $A_0$ of the symmetric spectrum $A$ is a pointed
monoid which is commutative if $A$ is commutative. There is a
canonical map 
$$[-]_n \colon \pi_0(A_0) \to \operatorname{TR}_0^n(A;p) 
\hskip7mm \text{(Teichm\"{u}ller map)}$$
which satisfies $R([a]_n) = [a]_{n-1}$ and $F([a]_n) =
[a^p]_{n-1}$; see~\cite[Sect.~2.5]{hm3}. If $A$ is commutative, the
Teichm\"{u}ller map is multiplicative and satisfies
$$Fd([a]_n) = [a]_{n-1}^{p-1}d([a]_{n-1}).$$

\section{The fundamental theorem}\label{fundamentaltheoremsection}

Let $A$ be a symmetric ring spectrum, and let $\Gamma$ be the free
group on a generator $x$. We define the symmetric ring spectrum
$A[x^{\pm 1}]$ to be the symmetric spectrum
$$A[x^{\pm 1}] = A \wedge \Gamma_+$$
with the multiplication map given by the composition of the canonical
isomorphism from $A \wedge \Gamma_+ \wedge A \wedge \Gamma_+$ to $A \wedge A 
\wedge \Gamma_+ \wedge \Gamma_+$ that permutes the second and third smash
factors and the smash product $\mu_A \wedge \mu_{\Gamma}$ of the
multiplication maps of $A$ and $\Gamma$ and with the unit map given by the
composition of the canonical isomorphism from $\mathbb{S}$ to
$\mathbb{S} \wedge S^0$ and the smash product $e_A \wedge e_{\Gamma}$ of the
unit maps of $A$ and $\Gamma$. There is a natural map of symmetric ring
spectra $f \colon A \to A[x^{\pm 1}]$ defined to be the composition of
the canonical isomorphism from $A$ to $A \wedge S^0$ and the smash
product $\operatorname{id}_A \wedge e_{\Gamma}$ of the identity map of $A$
and the unit map of $\Gamma$. It induces a natural map
$$f_* \colon \operatorname{TR}_q^n(A;p) \to
\operatorname{TR}_q^n(A[x^{\pm 1}];p).$$
Moreover, there is a map of symmetric ring spectra $g \colon
\mathbb{S}[x^{\pm 1}] \to A[x^{\pm 1}]$ defined to be the smash
product $e_A \wedge \operatorname{id}_{\Gamma}$ of the unit map of $A$ and
the identity map of $\Gamma$. The map $g$ makes $A[x^{\pm 1}]$ into an
algebra spectrum over the commutative symmetric ring spectrum
$\mathbb{S}[x^{\pm 1}]$. It follows that there is a natural pairing
$$\nu \colon \operatorname{TR}_q^n(A[x^{\pm 1}];p) \otimes
\operatorname{TR}_{q'}^n(\mathbb{S}[x^{\pm 1}];p)
\to \operatorname{TR}_{q+q'}^n(A[x^{\pm 1}];p)$$
which makes the graded abelian group $\operatorname{TR}_*^n(A[x^{\pm
1}];p)$ a graded module over the anti-symmetric graded ring
$\operatorname{TR}_*^n(\mathbb{S}[x^{\pm 1}];p)$. The element $[x]_n
\in \operatorname{TR}_0^n(\mathbb{S}[x^{\pm 1}];p)$ is a unit with
inverse $[x]_n^{-1} = [x^{-1}]_n$ and we define 
$$d \log [x]_n = [x]_n^{-1}d[x]_n \in
\operatorname{TR}_1^n(\mathbb{S}[x^{\pm 1}];p).$$
It follows from the general relations that
$$F(d\log [x]_n) = R(d\log [x]_n) = d\log [x]_{n-1}.$$
Now, given an integer $j$ and and element $a \in
\operatorname{TR}_q^n(A;p)$, we define
$$\begin{aligned}
a[x]_n^j & = \nu(f_*(a) \otimes [x]_n^j) \in
\operatorname{TR}_q^n(A[x^{\pm 1}];p) \cr
a[x]_n^jd\log[x]_n & = \nu(f_*(a) \otimes [x]_n^jd\log[x]_n) \in
\operatorname{TR}_{q+1}^n(A[x^{\pm 1}];p). \cr
\end{aligned}$$
The following theorem, which is similar to the fundamental theorem of
algebraic $K$-theory, was proved by Ib Madsen and the author
in~\cite[Thm.~C]{hm3}. The assumption in loc.~cit.~that the prime $p$
be odd is unnecessary; the same proof works for $p = 2$. However, the
formulas for $F$, $V$, and $d$ given in~loc.~cit.~are valid for odd
primes only. Below, we give a formula for the Frobenius which holds
for all primes $p$.

\begin{theorem}\label{fundamentaltheorem}Let $p$ be a prime number, and
let $A$ be a symmetric ring spectrum whose homotopy groups are
$\mathbb{Z}_{(p)}$-modules. Then every element $\omega \in
\operatorname{TR}_q^n(A[x^{\pm 1}];p)$ can be written uniquely as a
(finite) sum
$$\sum_{j \in \mathbb{Z}} \big( a_{0,j} [x]_n^j + b_{0,j} [x]_n^j
d\log [x]_n \big) + 
\sum_{ \substack{1 \leqslant s < n \\ j \in \mathbb{Z} \smallsetminus
  p\mathbb{Z}}} \big( V^s(a_{s,j} [x]_{n-s}^j) +
dV^s(b_{s,j} [x]_{n-s}^j) \big)$$
with $a_{s,j} = a_{s,j}(\omega) \in \operatorname{TR}_q^{n-s}(A;p)$ and
$b_{s,j} = b_{s,j}(\omega) \in
\operatorname{TR}_{q-1}^{n-s}(A;p)$. The corresponding statement for
the equivariant homotopy groups with $\mathbb{Z}_p$-coefficients is
valid for every symmetric ring spectrum $A$.
\end{theorem}

It is perhaps helpful to point out that the formula in the statement
of Thm.~\ref{fundamentaltheorem} defines a canonical map from the
direct sum
$$\begin{aligned}
{} & \bigoplus_{j \in \mathbb{Z}} \big( \operatorname{TR}_q^n(A;p)
\oplus \operatorname{TR}_{q-1}^n(A;p) \big) \oplus
\bigoplus_{ \substack{ 1 \leqslant s < n \\ j \in \mathbb{Z}
\smallsetminus p\mathbb{Z}}} \big( \operatorname{TR}_q^{n-s}(A;p) 
\oplus  \operatorname{TR}_{q-1}^{n-s}(A;p) \big) \cr
\end{aligned}$$
to the group $\operatorname{TR}_q^n(A[x^{\pm 1}];p)$ and that the
theorem states that this map is an isomorphism. We also remark that
the assembly map
$$\alpha \colon \operatorname{TR}_q^n(A;p) \oplus
\operatorname{TR}_{q-1}^n(A;p) \to 
\operatorname{TR}_q^n(A[x^{\pm 1}];p)$$
is given by the formula
$$\alpha(a,b) = a[x]_n^0 + b[x]_n^0d\log[x]_n,$$
where $[x]_n^0 = [1]_n \in \operatorname{TR}_0^n(\mathbb{S}[x^{\pm
  1}];p)$ is the multiplicative unit element.

The value of the restriction and Frobenius maps on
$\smash{ \operatorname{TR}_q^n(A[x^{\pm 1}];p) }$ are readily derived
from the general relations. Indeed, if $\smash{ \omega \in
\operatorname{TR}_q^n(A[x^{\pm 1}];p) }$ is equal to the sum in the 
statement of Thm.~\ref{fundamentaltheorem}, then
$$\begin{aligned}
R(\omega) & = 
\sum_{j \in \mathbb{Z}} \big( R(a_{0,j}) [x]_{n-1}^j +
R(b_{0,j}) [x]_{n-1}^j d\log [x]_{n-1} \big) \cr
{} & + 
\sum_{ \substack{1 \leqslant s < n-1 \\ j \in \mathbb{Z} \smallsetminus
  p\mathbb{Z}}} \big( V^s(R(a_{s,j}) [x]_{n-1-s}^j) +
dV^s(R(b_{s,j}) [x]_{n-1-s}^j) \big) \cr
\end{aligned}$$
$$\begin{aligned}
F(\omega) & = 
\sum_{j \in p\mathbb{Z}} \big( F(a_{0,j/p}) [x]_{n-1}^j +
F(b_{0,j/p}) [x]_{n-1}^j d\log [x]_{n-1} \big) \cr
{} & + \sum_{j \in \mathbb{Z} \smallsetminus p\mathbb{Z}}\big(
(pa_{1,j} + db_{1,j} + (p-1) \eta b_{1,j})
[x]_{n-1}^j \cr
{} & \hskip30mm
+ (-1)^{q-1}jb_{1,j}[x]_{n-1}^j d\log [x]_{n-1} \big) \cr
{} & + 
\sum_{ \substack{1 \leqslant s < n-1 \\ j \in \mathbb{Z} \smallsetminus
  p\mathbb{Z}}} \big( V^s((pa_{s+1,j} + (p-1) \eta b_{s+1,j})
[x]_{n-1-s}^j) \cr
{} & \hskip30mm
+ dV^s(b_{s+1,j} [x]_{n-1-s}^j) \big). \cr
\end{aligned}$$
We leave it to the reader to derive the corresponding formulas for the
Verschiebung map and Connes' operator. The following result is an
immediate consequence.

We recall that the limit system $\{ M_n \}$ satisfies the
Mittag-Leffler condition if, for every $n$, there exists
$m \geqslant n$ such that, for all $k \geqslant m$, the image of
$M_k \to M_n$ is equal to the image of $M_m \to M_n$. This implies
that the derived limit $R^1\operatornamewithlimits{lim}_n M_n$
vanishes.

\begin{corollary}\label{kernelR-F}Let $p$ be a prime number, let
$A$ be a symmetric ring spectrum whose homotopy groups are
$\mathbb{Z}_{(p)}$-modules, and let $q$ be an integer. If both of the
limit systems $\smash{ \{ \operatorname{TR}_q^n(A;p) \} }$ and
$\smash{ \{\operatorname{TR}_{q-1}^n(A;p) \} }$ satisfy the 
Mittag-Leffler condition, then so does the limit system
$\smash{ \{ \operatorname{TR}_q^n(A[x^{\pm1}];p)\} }$. Moreover, the
element
$$\omega = (\omega^{(n)}) \in 
\lim_R \operatorname{TR}_q^n(A[x^{\pm1}];p)$$
lies in the kernel of the map $1-F$ if and only if the coefficients
$$\begin{aligned}
a_{s,j}^{(n)} & = a_{s,j}(\omega^{(n)}) \in
\operatorname{TR}_q^{n-s}(A;p) \cr
b_{s,j}^{(n)} & = b_{s,j}(\omega^{(n)}) \in
\operatorname{TR}_{q-1}^{n-s}(A;p) \cr
\end{aligned}$$
satisfy the equations
$$\begin{aligned}
a_{s,j}^{(n-1)} & = \begin{cases}
F(a_{0,j/p}^{(n)}) & \hskip1mm
(\text{$s = 0$ and $j \in p\mathbb{Z}$}) \cr
pa_{1,j}^{(n)} + db_{1,j}^{(n)} + (p-1)\eta b_{1,j}^{(n)} & \hskip1mm
(\text{$s = 0$ and $j \in \mathbb{Z} \smallsetminus p\mathbb{Z}$}) \cr
pa_{s+1,j}^{(n)} + (p-1)\eta b_{s+1,j}^{(n)} & \hskip1mm
(\text{$1 \leqslant s < n-1$ and
$j \in \mathbb{Z}\smallsetminus p\mathbb{Z}$}) \cr
\end{cases} \cr
b_{s,j}^{(n-1)} & = \begin{cases}
F(b_{0,j/p}^{(n)}) & \hskip22mm
(\text{$s = 0$ and $j \in p\mathbb{Z}$}) \cr
(-1)^{q-1}jb_{1,j}^{(n)} & \hskip22mm
(\text{$s = 0$ and $j \in \mathbb{Z}\smallsetminus p\mathbb{Z}$}) \cr
b_{s+1,j}^{(n)} & \hskip22mm
(\text{$1 \leqslant s < n-1$ and
$j \in \mathbb{Z}\smallsetminus p\mathbb{Z}$}) \cr
\end{cases} \cr
\end{aligned}$$
for all $n \geqslant 1$. The corresponding statements for the
equivariant homotopy groups with $\mathbb{Z}_p$-coefficients is valid
for every symmetric ring spectrum $A$.
\end{corollary}

We do not have a good description of the cokernel of $1-F$. In
particular, it is generally not easy to decide whether or not this map
is surjective.

\section{Topological cyclic homology}\label{TCintro}

Let $A$ be a symmetric ring spectrum. We recall the definition of the
topological cyclic homology groups
$\operatorname{TC}_q(A;p)$ and refer to~\cite{hm,hm4} for a full
discussion.

We consider the $\mathbb{T}$-fixed point spectrum
$$\operatorname{TR}^n(A;p) = F((\mathbb{T}/C_{p^{n-1}})_+,
T(A))^{\mathbb{T}}$$
of the function $\mathbb{T}$-spectrum $F((\mathbb{T}/C_{p^{n-1}})_+,
T(A))$. There is a canonical isomorphism
$$\iota \colon \pi_q(\operatorname{TR}^n(A;p)) \xrightarrow{\sim}
\operatorname{TR}_q^n(A;p)$$
and maps of spectra
$$R^{\iota}, F^{\iota} \colon \operatorname{TR}^n(A;p) \to
\operatorname{TR}^{n-1}(A;p)$$
such that the following diagrams commute
$$\xymatrix{
{ \pi_q(\operatorname{TR}^n(A;p)) } \ar[r]^{\iota} \ar[d]^{R_*^{\iota}} &
{ \operatorname{TR}_q^n(A;p) } \ar[d]^{R} &
{ \pi_q(\operatorname{TR}^n(A;p)) } \ar[r]^{\iota} \ar[d]^{F_*^{\iota}} &
{ \operatorname{TR}_q^n(A;p) } \ar[d]^{F} \cr
{ \pi_q(\operatorname{TR}^{n-1}(A;p)) } \ar[r]^{\iota} &
{ \operatorname{TR}_q^{n-1}(A;p) } &
{ \pi_q(\operatorname{TR}^{n-1}(A;p)) } \ar[r]^{\iota} &
{ \operatorname{TR}_q^{n-1}(A;p). } \cr
}$$
The map $F^{\iota}$ is induced by the map of $\mathbb{T}$-spectra $f \colon
(\mathbb{T}/C_{p^{n-2}})_+ \to (\mathbb{T}/C_{p^{n-1}})_+$ and the map
$R^{\iota}$ is defined to be the composition of the map
$$F((\mathbb{T}/C_{p^{n-1}})_+,T(A))^{\mathbb{T}} \to
F((\mathbb{T}/C_{p^{n-1}})_+, \tilde{E} \wedge T(A))^{\mathbb{T}}$$
induced by the canonical inclusion of $S^0$ in $\tilde{E}$ and the
weak equivalence
$$\begin{aligned}
{} & F((\mathbb{T}/C_{p^{n-1}})_+, \tilde{E} \wedge T(A))^{\mathbb{T}}
\xleftarrow{\sim}
F((\mathbb{T}/C_{p^{n-1}})_+,(\tilde{E} \wedge T(A))^{C_p})^{\mathbb{T}/C_p} \cr
{} & \xrightarrow{\sim} 
F((\mathbb{T}/C_{p^{n-2}})_+, \rho_p^*((\tilde{E} \wedge
T(A))^{C_p}))^{\mathbb{T}} 
\xrightarrow{\sim} F((\mathbb{T}/C_{p^{n-2}})_+,
T(A))^{\mathbb{T}} \cr
\end{aligned}$$
defined by the composition of of the canonical isomorphism, the
isomorphism $\rho_p^*$, and the map induced by the map $r$ which we
recalled in Sect.~\ref{TRintro} above. We then define
$\operatorname{TC}^n(A;p)$ to be the homotopy equalizer of 
the maps $R^{\iota}$ and $F^{\iota}$ and
$$\operatorname{TC}(A;p) =
\operatornamewithlimits{holim}_{n} \operatorname{TC}^n(A;p)$$
to be the homotopy limit with respect to the maps $R^{\iota}$. We also
define
$$\operatorname{TR}(A;p) = \operatornamewithlimits{holim}_n
\operatorname{TR}^n(A;p)$$
to be the homotopy limit with respect to the maps $R^{\iota}$ such
that we have a long-exact sequence of homotopy groups
$$\cdots \to \operatorname{TC}_q(A;p) \to
\operatorname{TR}_q(A;p) \xrightarrow{1-F}
\operatorname{TR}_q(A;p) \xrightarrow{\partial}
\operatorname{TC}_{q-1}(A;p) \to \cdots.$$
We recall Milnor's short-exact sequence
$$0 \to 
R^1\operatornamewithlimits{lim}_n \operatorname{TR}_{q+1}^n(A;p) \to
\operatorname{TR}_q(A;p) \to
\operatornamewithlimits{lim}_n
\operatorname{TR}_q^n(A;p) \to 0.$$
In the cases we consider below, the derived limit on the left-hand
side vanishes.

The cyclotomic trace map of
B\"{o}kstedt-Hsiang-Madsen~\cite{bokstedthsiangmadsen} is a map of
spectra
$$\operatorname{tr} \colon K(A) \to \operatorname{TC}(A;p).$$
A technically better definition of the cyclotomic trace map was given
by Dundas-McCarthy~\cite[Sect.~2.0]{dundasmccarthy}
and~\cite{dundas1}. From the latter definition it is clear that every
class $x$ in the image of the composite map
$$K_q(A) \xrightarrow{\operatorname{tr}} \operatorname{TC}_q(A;p) \to
\operatorname{TC}_q^n(A;p) \to \operatorname{TR}_q^n(A;p)$$
is annihilated by Connes' operator. It is also not difficult to show 
that, for $A$ commutative, the cyclotomic trace is
multiplicative;~see~\cite[Appendix]{gh}. 

The spectrum $\operatorname{TR}^n(A;p)$ considered here is
canonically isomorphic to the underlying non-equivariant spectrum
associated with the $\mathbb{T}$-spectrum
$$\mathit{TR}^{\,n}(A;p) = \rho_{p^{n-1}}^*(T(A)^{C_{p^{n-1}}}).$$
Moreover, the fundamental long-exact sequence of~\cite[Thm.~2.2]{hm}
has the following generalization which is used in the proof of 
Lemma~\ref{connesoperatorsurjective} below.

\begin{proposition}\label{fundamentalcofibrationsequence}Let $A$ be a
symmetric ring spectrum, and let $m$ and $n$ be positive
integers. Then there is a natural long-exact sequence
$$\cdots \to \mathbb{H}_q(C_{p^m},\mathit{TR}^{\,n}(A;p))
\xrightarrow{N_n} 
\operatorname{TR}_q^{m+n}(A;p) \xrightarrow{R^n} 
\operatorname{TR}_q^m(A;p) \to \cdots$$
where the left-hand term is the group homology of $C_{p^m}$ with
coefficients in the underlying $C_{p^m}$-spectrum of
$\mathit{TR}^{\,n}(A;p)$.
\end{proposition}

\begin{proof}A map of $\mathbb{T}$-spectra $f \colon T \to T'$ is
defined to be an $\mathcal{F}_p$-equivalence if it induces an
isomorphism of equivariant homotopy groups
$$f_* \colon [S^q \wedge (\mathbb{T}/C_{p^v})_+, T]_{\mathbb{T}} \to
[S^q \wedge (\mathbb{T}/C_{p^v})_+, T']_{\mathbb{T}}$$
for all integers $q$ and $v \geqslant 0$. The cofibration sequence of
pointed $\mathbb{T}$-spaces
$$E_+ \xrightarrow{\pi} S^0 \xrightarrow{\iota} \tilde{E}
\xrightarrow{\partial} \Sigma E_+,$$
which we considered in Sect.~\ref{TRintro}, induces a cofibration
sequence of $\mathbb{T}$-spectra
$$E_+ \wedge \rho_{p^s}^*(T(A)^{C_{p^s}}) \to
\rho_{p^s}^*(T(A)^{C_{p^s}}) \to
\tilde{E} \wedge \rho_{p^s}^*(T(A)^{C_{p^s}}) \to
\Sigma E_+ \wedge \rho_{p^s}^*(T(A)^{C_{p^s}}).$$
We show that with $s = n - 1$, the induced long-exact sequence of
equivariant homotopy groups is isomorphic to the sequence of the
statement. The isomorphism of the left-hand terms in the two sequences
is defined as in Sect.~\ref{TRintro}. To define the isomorphism of the
right-hand terms in the two sequences, we first show that the
cyclotomic structure map $r$ gives rise to an
$\mathcal{F}_p$-equivalence
$$r' \colon \tilde{E} \wedge \rho_{p^{n-1}}^*(T(A)^{C_{p^{n-1}}})
\xrightarrow{\sim}
\tilde{E} \wedge T(A).$$
Since the map $\pi \colon E_+ \to S^0$ induces a weak equivalence
$$E_+ \wedge \rho_{p^s}^*((E_+ \wedge T(A))^{C_{p^s}})
\xrightarrow{\sim} \rho_{p^s}^*((E_+ \wedge T(A))^{C_{p^s}}),$$
a diagram chase shows that the map $\iota \colon S^0 \to \tilde{E}$
induces a weak equivalence
$$\tilde{E} \wedge \rho_{p^s}^*(T(A)^{C_{p^s}}) \xrightarrow{\sim}
\tilde{E} \wedge \rho_{p^s}^*(\tilde{E} \wedge T(A)^{C_{p^s}}).$$
The cyclotomic structure map $r$ induces an $\mathcal{F}_p$-equivalence
$$\tilde{E} \wedge \rho_{p^s}^*(\tilde{E} \wedge T(A)^{C_{p^s}})
\xrightarrow{\sim} \tilde{E} \wedge \rho_{p^{s-1}}^*(T(A)^{C_{p^{s-1}}})$$
which, composed with the former equivalence, defines an
$\mathcal{F}_p$-equivalence
$$\tilde{E} \wedge \rho_{p^s}^*(T(A)^{C_{p^s}}) \xrightarrow{\sim}
\tilde{E} \wedge \rho_{p^{s-1}}^*(T(A)^{C_{p^{s-1}}}).$$
The composition of these $\mathcal{F}_p$-equivalence as $s$ varies
from $n-1$ to $1$ gives the desired $\mathcal{F}_p$-equivalence
$r'$. The isomorphism of the right-hand terms in the two sequences is
now given by the composition of the isomorphism
$$[S^q \wedge (\mathbb{T}/C_{p^m})_+, \tilde{E} \wedge
\rho_{p^{n-1}}^*(T(A)^{C_{p^{n-1}}})]_{\mathbb{T}} \xrightarrow{\sim}
[S^q \wedge (\mathbb{T}/C_{p^m})_+, \tilde{E} \wedge
T(A)]_{\mathbb{T}}$$
induced by the map $r'$ and the isomorphism
$$[S^q \wedge (\mathbb{T}/C_{p^m})_+, \tilde{E} \wedge
T(A)]_{\mathbb{T}} \xrightarrow{\sim}
[S^q \wedge (T/C_{p^{m-1}})_+, T(A)]_{\mathbb{T}}$$
defined in Sect.~\ref{TRintro}.
\qed
\end{proof}

\section{The skeleton spectral
  sequence}\label{skeletonspectralsequence}

The left-hand groups in Prop.~\ref{fundamentalcofibrationsequence} are
the abutment of the strongly convergent skeleton spectral sequence which  
we now discuss in some detail. Let $G$ be a finite group, and let $T$
be a $G$-spectrum. Then we define
$$\mathbb{H}_q(G,T) = [S^q, E_+ \wedge T]_G,$$
where $E$ is a free contractible $G$-CW-complex. The group
$\mathbb{H}_q(G,T)$ is well-defined up to canonical
isomorphism. Indeed, if also $E'$ is a free contractible
$G$-CW-complex, then there is a unique homotopy class of $G$-maps $u
\colon E \to E'$, and the induced map $u_* \colon [S^q, E_+ \wedge
T]_G \to [S^q, E_+' \wedge T]_G$ is the canonical isomorphism. The
skeleton filtration of the $G$-CW-complex $E$ gives rise to a spectral
sequence
$$E_{s,t}^2 = H_s(G; \pi_t(T)) \Rightarrow \mathbb{H}_{s+t}(G,T)$$
from the homology of the group $G$ with coefficients in the $G$-module
$\pi_t(T)$. We will need the precise identification of the $E^2$-term
below. The augmented cellular complex of $E$ is the augmented chain
complex $\epsilon \colon P \to \mathbb{Z}$ defined by
$$P_s = \tilde{H}_s(E_s/E_{s-1}; \mathbb{Z})$$
with the differential $d$ induced by the map $\partial$ in the
cofibration sequence
$$E_{s-1}/E_{s-2} \to E_s/E_{s-2} \to E_s/E_{s-1}
\xrightarrow{\partial} \Sigma E_{s-1}/E_{s-2}.$$
and with the augmentation given by $\epsilon(x) = 1$, for all $x \in
E_0$. It is a resolution of the trivial $G$-module $\mathbb{Z}$ by
free $\mathbb{Z}[G]$-modules. We define
$$H_s(G,\pi_t(T)) = H_s((P \otimes \pi_t(T))^G, d \otimes
\operatorname{id}).$$
The $E^1$-term of the spectral sequence is defined by
$$E_{s,t}^1 = [S^{s+t}, (E_s/E_{s-1}) \wedge T]_G$$
with the $d^1$-differential induced by the boundary map $\partial$ in
the cofibration sequence above. The quotient $E_s/E_{s-1}$ is
homeomorphic to a wedge of $s$-spheres indexed by a set on which the
groups $G$ acts freely. Therefore, the Hurewicz homomorphism
$$\pi_s(E_s/E_{s-1}) \to
\tilde{H}(E_s/E_{s-1};\mathbb{Z}),$$
the exterior product map
$$\pi_s(E_s/E_{s-1}) \otimes \pi_t(T) \to
\pi_{s+t}((E_s/E_{s-1}) \wedge T),$$
and the canonical map
$$[S^{s+t}, (E_s/E_{s-1}) \wedge T]_G \to
(\pi_{s+t}((E_s/E_{s-1}) \wedge T))^G$$
are all isomorphisms. These isomorphisms gives rise to a canonical
isomophism
$$h \colon (P_s \otimes \pi_t(T))^G \xrightarrow{\sim} E_{s,t}^1$$
which satisfies $h \circ (d \otimes \operatorname{id}) = d^1 \circ h$.
The induced isomorphism of homology groups is then the desired
identification of the $E^2$-term.

We consider the skeleton spectral sequence with $G = C_{p^{n-1}}$ and
$T = \mathit{TR}^{\,v}(A;p)$. Since the action by
$\smash{C_{p^{n-1}}}$ on $\mathit{TR}^{\,v}(A;p)$ is the 
restriction of an action by the circle group $\mathbb{T}$, the induced
action on the homotopy groups $\operatorname{TR}_t^v(A;p)$ is
trivial. Moreover, it follows from~\cite[Lemma~1.4.2]{h} that the
$d^2$-differential of the spectral sequence is related to Connes'
operator $d$ in the following way.

\begin{lemma}\label{d^2-differential}Let $A$ be a symmetric ring
spectrum. Then, in the spectral sequence
$$E_{s,t}^2 = H_s(C_{p^{n-1}},\operatorname{TR}_t^v(A;p)) \Rightarrow
\mathbb{H}_{s+t}(C_{p^{n-1}}, \mathit{TR}^{\,v}(A;p)),$$
the $d^2$-differential $d^2 \colon E_{s,t}^2 \to E_{s-2,t+1}^2$
is equal to the map of group homology groups induced by $d + \eta$, if
$s$ is congruent to $0$ or $1$ modulo $4$, and the map induced by $d$,
if $s$ is congruent to $2$ or $3$ modulo $4$.
\end{lemma}

The Frobenius and Verschiebung maps
$$\begin{aligned}
F & \colon \mathbb{H}_q(C_{p^{n-1}},\mathit{TR}^{\,v}(A;p)) \to
\mathbb{H}_q(C_{p^{n-2}},\mathit{TR}^{\,v}(A;p)) \cr
V & \colon \mathbb{H}_q(C_{p^{n-2}},\mathit{TR}^{\,v}(A;p)) \to
\mathbb{H}_q(C_{p^{n-1}},\mathit{TR}^{\,v}(A;p)) \cr
\end{aligned}$$
induce maps of spectral sequences which on the $E^2$-terms of the
corresponding skeleton spectral sequence are given by the transfer and
corestriction maps in group homology corresponding to the inclusion of
$C_{p^{n-2}}$ in $C_{p^{n-1}}$.

Let $g \in C_{p^{n-1}}$ be the generator $g =
\exp(2 \pi i/p^{n-1})$, and let $\epsilon \colon W \to
\mathbb{Z}$ be the standard resolution which in degree $s$ is a free
$\mathbb{Z}[C_{p^{n-1}}]$-module of rank one on a generator $x_s$ with
differential $dx_s = Nx_{s-1}$, for $s$ even, and $dx_s =
(g-1)x_{s-1}$, for $s$ odd, and with augmentation $\epsilon(x_0) = 1$. 

\begin{lemma}\label{homologyfinite}Let $r$ and $n$ be positive
integers, and let $p$ be a prime number.

(i) If $r \leqslant n-1$, then
$$H_s(C_{p^{n-1}}, \mathbb{Z}/p^r\mathbb{Z}) = 
\mathbb{Z}/p^r\mathbb{Z} \cdot z_s,$$
where $z_s = z_s(p,n,r)$ is the class of $Nx_s \otimes 1$.

(ii) If $r \geqslant n-1$, then
$$H_s(C_{p^{n-1}}, \mathbb{Z}/p^r\mathbb{Z}) = \begin{cases}
{} \mathbb{Z}/p^r\mathbb{Z} \cdot z_0 &
\text{($s = 0$)} \cr
{} \mathbb{Z}/p^{n-1}\mathbb{Z} \cdot z_s &
\text{($s$ odd)} \cr
{} \mathbb{Z}/p^{n-1}\mathbb{Z} \cdot p^{r-(n-1)}z_s &
\text{($s > 0$ and even)} \cr
\end{cases}$$
where $z_s = z_s(p,n,r)$ and $p^{r-(n-1)}z_s = p^{r-(n-1)}z_s(p,n,r)$ are
the classes of $Nx_s \otimes 1$ and $p^{r-(n-1)}Nx_s \otimes 1$,
respectively. 

(iii) The transfer map
$$F \colon H_s(C_{p^{n-1}},\mathbb{Z}/p^r\mathbb{Z}) \to
H_s(C_{p^{n-2}},\mathbb{Z}/p^r\mathbb{Z})$$
maps $z_s$ to $z_s$, if $s$ is odd, maps $z_s$ to $pz_s$, if $s = 0$
or if $s > 0$ is even and $r \leqslant n-1$, and maps $p^{r-(n-1)}z_s$
to $p^{r-(n-2)}z_s$, if $s > 0$ is even and $r \geqslant n-1$.

(iv) The corestriction map
$$V \colon H_s(C_{p^{n-2}},\mathbb{Z}/p^r\mathbb{Z}) \to
H_s(C_{p^{n-1}},\mathbb{Z}/p^r\mathbb{Z})$$
maps $z_s$ to $pz_s$, if $s$ is odd, maps $z_s$ to $z_s$, if $s = 0$
or if $s > 0$ is even and $r \leqslant n-1$, and maps $p^{r-(n-1)}z_s$
to $p^{r-(n-1)}z_s$, if $s > 0$ is even and $r \geqslant n-1$.
\end{lemma}

\begin{proof}The statements (i) and (ii) are readily verified. To
prove (iii) and (iv), we write $\epsilon \colon W \to \mathbb{Z}$ and
$\epsilon' \colon W' \to \mathbb{Z}$ for the standard resolutions for
the groups $C_{p^{n-1}}$ and $C_{p^{n-2}}$, respectively. Then
$\epsilon \colon W \to \mathbb{Z}$ is a resolution of $\mathbb{Z}$ by
free $C_{p^{n-2}}$-modules. The map $h \colon W \to W'$ defined by
$$h(g^{dp+r}x_s) = \begin{cases}
g'{}^r x_s' & \text{($s$ even)} \cr
\delta_{r,p-1} g'{}^d x_s' & \text{($s$ odd),} \cr
\end{cases}$$
where $0 \leqslant r < p$ and $0 \leqslant d < n-2$, is a
$C_{p^{n-2}}$-linear chain map that lifts the identity map of
$\mathbb{Z}$, and the $C_{p^{n-2}}$-linear map $k \colon W' \to W$
defined by
$$k(x_s') = \begin{cases}
x_s & \text{($s$ even)} \cr
(1 + g + \dots g^{p-1})x_s & \text{($s$ odd)} \cr
\end{cases}$$
is a chain map and lifts the identity of $\mathbb{Z}$. Now the transfer
map $F$ is the map of homology groups induced by the composite chain
map
$$(W \otimes \mathbb{Z}/p^r\mathbb{Z})^{C_{p^{n-1}}} \hookrightarrow
(W \otimes \mathbb{Z}/p^r\mathbb{Z})^{C_{p^{n-2}}} \xrightarrow{h
  \otimes 1}
(W' \otimes \mathbb{Z}/p^r\mathbb{Z})^{C_{p^{n-2}}}$$
where the left-hand map is the canonical inclusion. One verifies
readily that this map takes $Nx_{2i} \otimes 1$ to $pN'x_{2i}' \otimes
1$ and $Nx_{2i-1} \otimes 1$ to $N'x_{2i-1}' \otimes 1$. Similarly,
the corestriction map $V$ is the map of homology groups induced by the
composite chain map
$$(W' \otimes \mathbb{Z}/p^r\mathbb{Z})^{C_{p^{n-2}}} \xrightarrow{k
  \otimes 1}
(W \otimes \mathbb{Z}/p^r\mathbb{Z})^{C_{p^{n-2}}} \xrightarrow{N/N'}
(W \otimes \mathbb{Z}/p^r\mathbb{Z})^{C_{p^{n-1}}}$$
where the right-hand map is multiplication by $1 + g + \dots +
g^{p-1}$. This map takes $N'x_{2i}' \otimes 1$ to $Nx_{2i} \otimes 1$
and $N'x_{2i-1}' \otimes 1$ to $pNx_{2i-1} \otimes 1$. \qed
\end{proof}

\begin{lemma}\label{homologyinfinite}Let $n$ be a positive integer and
let $p$ be a prime number. Then
$$H_s(C_{p^{n-1}},\mathbb{Z}) = \begin{cases}
\mathbb{Z} \cdot z_0 & \text{($s = 0$)} \cr
\mathbb{Z}/p^{n-1}\mathbb{Z} \cdot z_s & \text{($s$ odd)} \cr
0 & \text{($s>0$ even)} \cr
\end{cases}$$
where $z_s = z_s(p,n)$ is the class of $Nx_s \otimes 1$. The transfer
map
$$F \colon H_s(C_{p^{n-1}},\mathbb{Z}) \to
H_s(C_{p^{n-2}},\mathbb{Z})$$
maps $z_0$ to $pz_0$ and $z_s$ to $z_s$, for $s > 0$, and the
corestriction map
$$V \colon H_s(C_{p^{n-2}},\mathbb{Z}) \to
H_s(C_{p^{n-1}},\mathbb{Z})$$
maps $z_0$ to $z_0$ and $z_s$ to $pz_s$, for $s > 0$.
\end{lemma}

\begin{proof}The proof is similar to the proof of
Lemma~\ref{homologyfinite}. 
\qed
\end{proof}

\section{The groups
 $\operatorname{TR}_q^n(\mathbb{S};2)$}\label{spheresection}

In this section, we implicitly consider homotopy groups with
$\mathbb{Z}_2$-coefficients. The groups
$\operatorname{TR}_*^1(\mathbb{S};2)$ are the stable homotopy 
groups of spheres. The group $\operatorname{TR}_0^1(\mathbb{S};2)$ is
isomorphic to $\mathbb{Z}_2$ generated by the multiplicative unit
element $\iota = [1]_1$; the group
$\operatorname{TR}_1^1(\mathbb{S};2)$ is isomorphic to
$\mathbb{Z}/2\mathbb{Z}$ generated by the Hopf class $\eta$; the group
$\operatorname{TR}_2^1(\mathbb{S};2)$ is isomorphic to
$\mathbb{Z}/2\mathbb{Z}$ generated by $\eta^2$; the group 
$\operatorname{TR}_3^1(\mathbb{S};2)$ is isomorphic to
$\mathbb{Z}/8\mathbb{Z}$ generated by the Hopf class $\nu$ and
$\eta^3 = 4\nu$; the groups $\operatorname{TR}_4^1(\mathbb{S};2)$ and
$\operatorname{TR}_5^1(\mathbb{S};2)$ are zero; the group
$\operatorname{TR}_6^1(\mathbb{S};2)$ is isomorphic to
$\mathbb{Z}/2\mathbb{Z}$ generated by $\nu^2$, and the group
$\operatorname{TR}_7^1(\mathbb{S};2)$ is isomorphic to
$\mathbb{Z}/16\mathbb{Z}$ generated the Hopf class $\sigma$. 

We consider the skeleton spectral sequence
$$E_{s,t}^2 = H_s(C_{2^{n-1}}, \operatorname{TR}_t^1(\mathbb{S};2))
\Rightarrow \mathbb{H}_{s+t}(C_{2^{n-1}}, T(\mathbb{S})).$$
This sequence may be identified with the Atiyah-Hirzebruch spectral
sequence that converges to the homotopy groups of the suspension
spectrum of the pointed space
$(BC_{2^{n-1}})_+$~\cite[Prop.~2.4]{greenleesmay}. Therefore, the
edge-homomorphism onto the line $s = 0$ has a retraction, and
hence, the differentials $d^r \colon E_{r,t}^r \to  E_{0,t+r-1}^r$ are
all zero.

Suppose first that $n = 2$. Then the $E^2$-term for $s+t \leqslant 7$
is takes the form
\begin{center}
\begin{tabular*}{0.95\textwidth}{@{\extracolsep{\fill}}cccccccc}
$\mathbb{Z}/16\mathbb{Z}$ &  &  &  &  &  &  &  \cr
$\mathbb{Z}/2\mathbb{Z}$ & $\mathbb{Z}/2\mathbb{Z}$ &  &  &  &  &  \cr
0 & 0 & 0 &  &  &  &  &  \cr
0 & 0 & 0 & 0 &  &  &  &  \cr
$\mathbb{Z}/8\mathbb{Z}$ & $\mathbb{Z}/2\mathbb{Z}$ &
$4\mathbb{Z}/8\mathbb{Z}$ & $\mathbb{Z}/2\mathbb{Z}$ &
$4\mathbb{Z}/8\mathbb{Z}$ &  &  & \cr
$\mathbb{Z}/2\mathbb{Z}$ & $\mathbb{Z}/2\mathbb{Z}$ &
$\mathbb{Z}/2\mathbb{Z}$ & $\mathbb{Z}/2\mathbb{Z}$ &
$\mathbb{Z}/2\mathbb{Z}$ & $\mathbb{Z}/2\mathbb{Z}$ &  &  \cr
$\mathbb{Z}/2\mathbb{Z}$ & $\mathbb{Z}/2\mathbb{Z}$ &
$\mathbb{Z}/2\mathbb{Z}$ & $\mathbb{Z}/2\mathbb{Z}$ &
$\mathbb{Z}/2\mathbb{Z}$ & $\mathbb{Z}/2\mathbb{Z}$ &
$\mathbb{Z}/2\mathbb{Z}$ &  \cr
$\mathbb{Z}_2$ & $\mathbb{Z}/2\mathbb{Z}$ &
$0$ & $\mathbb{Z}/2\mathbb{Z}$ &
$0$ & $\mathbb{Z}/2\mathbb{Z}$ &
$0$ & $\mathbb{Z}/2\mathbb{Z}$ \cr
\end{tabular*}
\end{center}
where $s$ is horizontal coordinate and $t$ the vertical
coordinate. The group $\smash{E_{s,0}^2}$ is generated by the class
$\iota z_s$, the group $\smash{E_{s,1}^2}$ by the class $\eta z_s$,
the group $\smash{E_{s,2}^2}$ by the class $\eta^2 z_s$, the group
$\smash{E_{s,3}^2}$ with $s = 0$ or $s$ an odd positive integer by the
class $\nu z_s$, the group $\smash{E_{s,3}^2}$ with $s$ an even
positive integer by the class $4\nu z_s$, the group
$\smash{E_{s,6}^2}$ by the class $\nu^2 z_s$, the group
$\smash{E_{s,7}^2}$ with $s = 0$ or $s$ an odd positive integer by the
class $\sigma z_s$, and the group $\smash{E_{s,7}^2}$ with $s$ an even
positive integer by the class $8\sigma z_s$, where $z_s$ are the
classes defined in Lemmas~\ref{homologyfinite}
and~\ref{homologyinfinite}. We recall from
Lemma~\ref{d^2-differential} that the $d^2$-differential is given by
Connes' operator and by multiplication by $\eta$. Since
Connes' operator on $\operatorname{TR}_*^1(\mathbb{S};2)$ is
zero, we find that the $\smash{E^3}$-term begins
\begin{center}
\begin{tabular*}{0.95\textwidth}{@{\extracolsep{\fill}}cccccccc}
$\mathbb{Z}/16\mathbb{Z}$ &  &  &  &  &  &  &  \cr
$\mathbb{Z}/2\mathbb{Z}$ & $\mathbb{Z}/2\mathbb{Z}$ &  &  &  &  &  \cr
0 & 0 & 0 &  &  &  &  &  \cr
0 & 0 & 0 & 0 &  &  &  &  \cr
$\mathbb{Z}/8\mathbb{Z}$ & $\mathbb{Z}/2\mathbb{Z}$ &
0 & $\mathbb{Z}/2\mathbb{Z}$ &
$4\mathbb{Z}/8\mathbb{Z}$ &  &  & \cr
$\mathbb{Z}/2\mathbb{Z}$ & $\mathbb{Z}/2\mathbb{Z}$ &
0 & 0 &
0 & $\mathbb{Z}/2\mathbb{Z}$ &  &  \cr
$\mathbb{Z}/2\mathbb{Z}$ & $\mathbb{Z}/2\mathbb{Z}$ &
$\mathbb{Z}/2\mathbb{Z}$ & 0 &
0 & 0 &
$\mathbb{Z}/2\mathbb{Z}$ &  \cr
$\mathbb{Z}_2$ & $\mathbb{Z}/2\mathbb{Z}$ &
0 & $\mathbb{Z}/2\mathbb{Z}$ &
0 & 0 &
0 & $\mathbb{Z}/2\mathbb{Z}$ \cr
\end{tabular*}
\end{center}
Since the differential $d^3 \colon E_{3,0}^3 \to E_{0,2}^3$ is zero,
the $E^3$-term is also the $E^4$-term. The following result is a
consequence of Mosher~\cite[Prop.~5.2]{mosher}.

\begin{lemma}\label{d^4-differential}Let $n$ be a positive
integer. Then, in the spectral sequence
$$E_{s,t}^2 = H_s(C_{2^{n-1}}, \operatorname{TR}_t^1(\mathbb{S};2))
\Rightarrow \mathbb{H}_{s+t}(C_{2^{n-1}}, T(\mathbb{S})),$$
the $d^4$-differential $\smash{d^4 \colon E_{s,t}^4 \to
E_{s-4,t+3}^4}$ is equal to the map of sub-quotients induced from the
map of group homology groups induced from multiplication by $\nu$, if
$s$ is congruent to $0$, $1$, $2$, $3$, $8$, $9$, $10$, or $11$ modulo
$16$, by $2\nu$, if $s$ is congruent to $6$, $7$, $12$, or $13$ modulo
$16$, and by $0$, if $s$ is congruent to $4$, $5$, $14$, or $15$
modulo $16$.
\qed
\end{lemma}

In the case at hand, we find that the $d^4$-differential is zero, for 
$s+t \leqslant 7$. For degree reasons, the only possible higher
non-zero differential all have target on the fiber line $s = 0$.
However, we argued above that these differentials are zero. Therefore,
for $s+t \leqslant 7$, the $\smash{E^3}$-term is also the
$\smash{E^{\infty}}$-term.

The $E^2$-term of the skeleton spectral sequence for
$\mathbb{H}_q(C_4,T(\mathbb{S}))$ for $s+t \leqslant 7$ is
\begin{center}
\begin{tabular*}{0.95\textwidth}{@{\extracolsep{\fill}}cccccccc}
$\mathbb{Z}/16\mathbb{Z}$ &  &  &  &  &  &  &  \cr
$\mathbb{Z}/2\mathbb{Z}$ & $\mathbb{Z}/2\mathbb{Z}$ &  &  &  &  &  \cr
0 & 0 & 0 &  &  &  &  &  \cr
0 & 0 & 0 & 0 &  &  &  &  \cr
$\mathbb{Z}/8\mathbb{Z}$ & $\mathbb{Z}/4\mathbb{Z}$ &
$2\mathbb{Z}/8\mathbb{Z}$ & $\mathbb{Z}/4\mathbb{Z}$ &
$2\mathbb{Z}/8\mathbb{Z}$ &  &  & \cr
$\mathbb{Z}/2\mathbb{Z}$ & $\mathbb{Z}/2\mathbb{Z}$ &
$\mathbb{Z}/2\mathbb{Z}$ & $\mathbb{Z}/2\mathbb{Z}$ &
$\mathbb{Z}/2\mathbb{Z}$ & $\mathbb{Z}/2\mathbb{Z}$ &  &  \cr
$\mathbb{Z}/2\mathbb{Z}$ & $\mathbb{Z}/2\mathbb{Z}$ &
$\mathbb{Z}/2\mathbb{Z}$ & $\mathbb{Z}/2\mathbb{Z}$ &
$\mathbb{Z}/2\mathbb{Z}$ & $\mathbb{Z}/2\mathbb{Z}$ &
$\mathbb{Z}/2\mathbb{Z}$ &  \cr
$\mathbb{Z}_2$ & $\mathbb{Z}/4\mathbb{Z}$ &
$0$ & $\mathbb{Z}/4\mathbb{Z}$ &
$0$ & $\mathbb{Z}/4\mathbb{Z}$ &
$0$ & $\mathbb{Z}/4\mathbb{Z}$ \cr
\end{tabular*}
\end{center}
The generators of the groups $\smash{E_{s,t}^2}$ are as before with
exception that the groups $\smash{E_{s,3}^3}$ and $\smash{E_{s,7}^2}$
with $s$ an even positive integer are generated by $2\nu z_s$ and
$4\sigma z_s$, respectively. We find as before that the
$\smash{E^3}$-term for $s + t \leqslant 7$ takes
the form
\begin{center}
\begin{tabular*}{0.95\textwidth}{@{\extracolsep{\fill}}cccccccc}
$\mathbb{Z}/16\mathbb{Z}$ &  &  &  &  &  &  &  \cr
$\mathbb{Z}/2\mathbb{Z}$ & $\mathbb{Z}/2\mathbb{Z}$ &  &  &  &  &  \cr
0 & 0 & 0 &  &  &  &  &  \cr
0 & 0 & 0 & 0 &  &  &  &  \cr
$\mathbb{Z}/8\mathbb{Z}$ & $\mathbb{Z}/4\mathbb{Z}$ &
$2\mathbb{Z}/4\mathbb{Z}$ & $\mathbb{Z}/4\mathbb{Z}$ &
$2\mathbb{Z}/8\mathbb{Z}$ &  &  & \cr
$\mathbb{Z}/2\mathbb{Z}$ & $\mathbb{Z}/2\mathbb{Z}$ &
0 & 0 &
0 & $\mathbb{Z}/2\mathbb{Z}$ &  &  \cr
$\mathbb{Z}/2\mathbb{Z}$ & $\mathbb{Z}/2\mathbb{Z}$ &
$\mathbb{Z}/2\mathbb{Z}$ & 0 &
0 & 0 &
$\mathbb{Z}/2\mathbb{Z}$ &  \cr
$\mathbb{Z}_2$ & $\mathbb{Z}/4\mathbb{Z}$ &
$0$ & $\mathbb{Z}/4\mathbb{Z}$ &
$0$ & $2\mathbb{Z}/4\mathbb{Z}$ &
$0$ & $\mathbb{Z}/4\mathbb{Z}$ \cr
\end{tabular*}
\end{center}
The only possible non-zero $d^3$-differential for $s+t \leqslant 7$ is 
$\smash{d^3 \colon E_{6,1}^3 \to E_{3,3}^3}$. Since the corresponding
differential in the previous spectral sequence is zero, a comparison
by using the Verschiebung map shows that also this differential is
zero. The $\smash{d^4}$-differentials are given by
Lemma~\ref{d^4-differential}. Hence, the $\smash{E^5}$-term begins
\begin{center}
\begin{tabular*}{0.95\textwidth}{@{\extracolsep{\fill}}cccccccc}
$\mathbb{Z}/16\mathbb{Z}$ &  &  &  &  &  &  &  \cr
$\mathbb{Z}/2\mathbb{Z}$ & $\mathbb{Z}/2\mathbb{Z}$ &  &  &  &  &  \cr
0 & 0 & 0 &  &  &  &  &  \cr
0 & 0 & 0 & 0 &  &  &  &  \cr
$\mathbb{Z}/8\mathbb{Z}$ & $\mathbb{Z}/4\mathbb{Z}$ &
$2\mathbb{Z}/4\mathbb{Z}$ & $\mathbb{Z}/2\mathbb{Z}$ &
$2\mathbb{Z}/8\mathbb{Z}$ &  &  & \cr
$\mathbb{Z}/2\mathbb{Z}$ & $\mathbb{Z}/2\mathbb{Z}$ &
0 & 0 &
0 & $\mathbb{Z}/2\mathbb{Z}$ &  &  \cr
$\mathbb{Z}/2\mathbb{Z}$ & $\mathbb{Z}/2\mathbb{Z}$ &
$\mathbb{Z}/2\mathbb{Z}$ & 0 &
0 & 0 &
$\mathbb{Z}/2\mathbb{Z}$ &  \cr
$\mathbb{Z}_2$ & $\mathbb{Z}/4\mathbb{Z}$ &
$0$ & $\mathbb{Z}/4\mathbb{Z}$ &
$0$ & $2\mathbb{Z}/4\mathbb{Z}$ &
$0$ & $2\mathbb{Z}/4\mathbb{Z}$ \cr
\end{tabular*}
\end{center}
We see as before that the $E^5$-term is also the $E^{\infty}$-term.

Finally, we consider the skeleton spectral sequence for
$\mathbb{H}_q(C_{2^{n-1}},T(\mathbb{S}))$, where $n \geqslant 4$. The
$\smash{E^2}$-term for $s+t \leqslant 7$, takes the form 
\begin{center}
\begin{tabular*}{0.95\textwidth}{@{\extracolsep{\fill}}cccccccc}
$\mathbb{Z}/16\mathbb{Z}$ &  &  &  &  &  &  &  \cr
$\mathbb{Z}/2\mathbb{Z}$ & $\mathbb{Z}/2\mathbb{Z}$ &  &  &  &  &  \cr
0 & 0 & 0 &  &  &  &  &  \cr
0 & 0 & 0 & 0 &  &  &  &  \cr
$\mathbb{Z}/8\mathbb{Z}$ & $\mathbb{Z}/8\mathbb{Z}$ &
$\mathbb{Z}/8\mathbb{Z}$ & $\mathbb{Z}/8\mathbb{Z}$ &
$\mathbb{Z}/8\mathbb{Z}$ &  &  & \cr
$\mathbb{Z}/2\mathbb{Z}$ & $\mathbb{Z}/2\mathbb{Z}$ &
$\mathbb{Z}/2\mathbb{Z}$ & $\mathbb{Z}/2\mathbb{Z}$ &
$\mathbb{Z}/2\mathbb{Z}$ & $\mathbb{Z}/2\mathbb{Z}$ &  &  \cr
$\mathbb{Z}/2\mathbb{Z}$ & $\mathbb{Z}/2\mathbb{Z}$ &
$\mathbb{Z}/2\mathbb{Z}$ & $\mathbb{Z}/2\mathbb{Z}$ &
$\mathbb{Z}/2\mathbb{Z}$ & $\mathbb{Z}/2\mathbb{Z}$ &
$\mathbb{Z}/2\mathbb{Z}$ &  \cr
$\mathbb{Z}_2$ & $\mathbb{Z}/2^{n-1}\mathbb{Z}$ &
$0$ & $\mathbb{Z}/2^{n-1}\mathbb{Z}$ &
$0$ & $\mathbb{Z}/2^{n-1}\mathbb{Z}$ &
$0$ & $\mathbb{Z}/2^{n-1}\mathbb{Z}$ \cr
\end{tabular*}
\end{center}
The generators of the groups $\smash{E_{s,t}^2}$ are the same as in the
skeleton spectral sequence for $\mathbb{H}_q(C_2,T(\mathbb{S}))$ with
the exception that the groups $\smash{E_{s,3}^2}$ and $\smash{E_{s,7}^2}$
are generated by the classes $\nu z_s$ and $\sigma z_s$, respectively,
for all $s \geqslant 0$. The $\smash{d^2}$-differential is given by
Lemma~\ref{d^2-differential}. We find that the $\smash{E^3}$-term for
$s+t \leqslant 7$ becomes
\begin{center}
\begin{tabular*}{0.95\textwidth}{@{\extracolsep{\fill}}cccccccc}
$\mathbb{Z}/16\mathbb{Z}$ &  &  &  &  &  &  &  \cr
$\mathbb{Z}/2\mathbb{Z}$ & $\mathbb{Z}/2\mathbb{Z}$ &  &  &  &  &  \cr
0 & 0 & 0 &  &  &  &  &  \cr
0 & 0 & 0 & 0 &  &  &  &  \cr
$\mathbb{Z}/8\mathbb{Z}$ & $\mathbb{Z}/8\mathbb{Z}$ &
$\mathbb{Z}/4\mathbb{Z}$ & $\mathbb{Z}/4\mathbb{Z}$ &
$\mathbb{Z}/8\mathbb{Z}$ &  &  & \cr
$\mathbb{Z}/2\mathbb{Z}$ & $\mathbb{Z}/2\mathbb{Z}$ &
0 & 0 &
0 & 0 &  &  \cr
$\mathbb{Z}/2\mathbb{Z}$ & $\mathbb{Z}/2\mathbb{Z}$ &
$\mathbb{Z}/2\mathbb{Z}$ & 0 &
0 & 0 &
$\mathbb{Z}/2\mathbb{Z}$ &  \cr
$\mathbb{Z}_2$ & $\mathbb{Z}/2^{n-1}\mathbb{Z}$ &
$0$ & $\mathbb{Z}/2^{n-1}\mathbb{Z}$ &
$0$ & $2\mathbb{Z}/2^{n-1}\mathbb{Z}$ &
$0$ & $\mathbb{Z}/2^{n-1}\mathbb{Z}$ \cr
\end{tabular*}
\end{center}
A comparison with the previous spectral sequence by using the
Verschiebung map shows that the $d^3$-differential is zero. The
$\smash{d^4}$-differential is given by
Lemma~\ref{d^4-differential}. Hence, the $\smash{E^5}$-term for $s+t
\leqslant 7$ becomes
\begin{center}
\begin{tabular*}{0.95\textwidth}{@{\extracolsep{\fill}}cccccccc}
$\mathbb{Z}/16\mathbb{Z}$ &  &  &  &  &  &  &  \cr
$\mathbb{Z}/2\mathbb{Z}$ & $\mathbb{Z}/2\mathbb{Z}$ &  &  &  &  &  \cr
0 & 0 & 0 &  &  &  &  &  \cr
0 & 0 & 0 & 0 &  &  &  &  \cr
$\mathbb{Z}/8\mathbb{Z}$ & $\mathbb{Z}/8\mathbb{Z}$ &
$\mathbb{Z}/4\mathbb{Z}$ & $\mathbb{Z}/2\mathbb{Z}$ &
$\mathbb{Z}/8\mathbb{Z}$ &  &  & \cr
$\mathbb{Z}/2\mathbb{Z}$ & $\mathbb{Z}/2\mathbb{Z}$ &
0 & 0 &
0 & 0 &  &  \cr
$\mathbb{Z}/2\mathbb{Z}$ & $\mathbb{Z}/2\mathbb{Z}$ &
$\mathbb{Z}/2\mathbb{Z}$ & 0 &
0 & 0 &
$\mathbb{Z}/2\mathbb{Z}$ &  \cr
$\mathbb{Z}_2$ & $\mathbb{Z}/2^{n-1}\mathbb{Z}$ &
$0$ & $\mathbb{Z}/2^{n-1}\mathbb{Z}$ &
$0$ & $2\mathbb{Z}/2^{n-1}\mathbb{Z}$ &
$0$ & $2\mathbb{Z}/2^{n-1}\mathbb{Z}$ \cr
\end{tabular*}
\end{center}
and, for $s+t \leqslant 7$, this is also the $E^{\infty}$-term.

\begin{lemma}\label{generatorsxi} (i) There exists unique homotopy
classes 
$$\xi_{1,n-1} \in \mathbb{H}_1(C_{2^{n-1}},T(\mathbb{S})) \hskip5mm
(n \geqslant 1)$$
such that $\xi_{1,n-1}$ represents $\smash{\iota z_1 \in
E_{1,0}^{\infty}}$, $F(\xi_{1,n-1}) = \xi_{1,n-2}$, and
$\xi_{1,0} = \eta$.

(ii) There exists unique homotopy classes
$$\xi_{3,n-1} \in \mathbb{H}_3(C_{2^{n-1}},T(\mathbb{S})) \hskip5mm
(n \geqslant 1)$$
such that $\xi_{3,n-1}$ represents $\smash{\iota z_3 \in
E_{3,0}^{\infty}}$, $F(\xi_{3,n-1}) = \xi_{3,n-2}$, and $\xi_{3,0} = \nu$.

(iii) There exists unique homotopy classes
$$\xi_{5,n-1} \in \mathbb{H}_5(C_{2^{n-1}},T(\mathbb{S})) \hskip5mm
(n \geqslant 1)$$
such that $\xi_{5,n-1}$ represents $\smash{2\iota z_5 \in
  E_{5,0}^{\infty}}$ and $F(\xi_{5,n-1}) = \xi_{5,n-2}$.
\end{lemma}

\begin{proof}We consider the inverse limit with respect to the
Frobenius maps of the skeleton spectral sequences for
$\mathbb{H}_q(C_{2^{n-1}},T(\mathbb{S}))$. By
Lemmas~\ref{homologyfinite} and~\ref{homologyinfinite}, the map of
spectral sequences induced by the Frobenius map is given, formally, by
$F(z_s) = 2z_s$, if $s$ is even, and $F(z_s) = z_s$, if $s$ is
odd. Hence, the $\smash{E^{\infty}}$-term of the inverse limit
spectral sequence for $s + t \leqslant 7$ takes the form
\begin{center}
\begin{tabular*}{0.95\textwidth}{@{\extracolsep{\fill}}cccccccc}
0 &  &  &  &  &  &  &  \cr
0 & $\mathbb{Z}/2\mathbb{Z}$ &  &  &  &  &  \cr
0 & 0 & 0 &  &  &  &  &  \cr
0 & 0 & 0 & 0 &  &  &  &  \cr
0 & $\mathbb{Z}/8\mathbb{Z}$ &
0 & $\mathbb{Z}/2\mathbb{Z}$ &
0 &  &  & \cr
0 & $\mathbb{Z}/2\mathbb{Z}$ &
0 & 0 &
0 & 0 &  &  \cr
0 & $\mathbb{Z}/2\mathbb{Z}$ &
0 & 0 &
0 & 0 &
0 &  \cr
0 & $\mathbb{Z}_2$ &
$0$ & $\mathbb{Z}_2$ &
$0$ & $2\mathbb{Z}_2$ &
$0$ & $2\mathbb{Z}_2$ \cr
\end{tabular*}
\end{center}
We now prove the statement~(i). There is a unique class
$$\xi_1 = \{ \xi_{1,n-1} \} \in \operatornamewithlimits{lim}_F
\mathbb{H}_1(C_{2^{n-1}},T(\mathbb{S}))$$
such that $\xi_{1,n-1}$ represents the generator
$\smash{\iota z_1 \in E_{1,0}^{\infty}}$, for all $n \geqslant 2$.
We can write
$$\xi_{1,n-1} = a_{n-1} dV^{n-1}(1) + b_{n-1} V^{n-1}(\eta),$$
where $a_{n-1} \in \mathbb{Z}/2^{n-1}\mathbb{Z}$ and $b_{n-1} \in
\mathbb{Z}/2\mathbb{Z}$. Since the class $\xi_{1,n-1}$ represents
$\iota z_1$, the proof of~\cite[Prop.~4.4.1]{hm4} shows that
$a_{n-1} = 1$. The calculation
$$\xi_{1,n-1} = F(\xi_{1,n}) = F(dV^n(1)+b_{n-1}V^n(\eta)) =
dV^{n-1}(1) + V^{n-1}(\eta)$$
shows that also $b_{n-1} = 1$. Finally,
$$\xi_{1,0} = F(\xi_{1,1}) = F(dV(1) + V(\eta)) = \eta$$
which proves~(i). To prove~(ii), we must show that there is a unique
class
$$\xi_3 = \{\xi_{3,n-1}\} \in \operatornamewithlimits{lim}_F
\mathbb{H}_3(C_{2^{n-1}},T(\mathbb{S}))$$
such that $\xi_{3,n-1}$ represents $\iota z_3$ and such that
$\xi_{3,0} = \nu$. There are two classes $\xi_3$ and $\xi_3'$ that
satisfy the first requirement and
$$\xi_{3,n-1} - \xi_{3,n-1}' = dV^{n-1}(\eta^2).$$
Moreover, if $n \geqslant 3$, then $F^{n-1} \colon
\mathbb{H}_3(C_{2^{n-1}},T(\mathbb{S})) \to
\operatorname{TR}_3^1(\mathbb{S};2)$ induces a map
$$\overline{F^{n-1}} \colon H_3(C_{2^{n-1}},\operatorname{TR}_0^1(\mathbb{S};2))
\to \operatorname{TR}_3^1(\mathbb{S};2) / 4
\operatorname{TR}_3^1(\mathbb{S};2).$$
Indeed, $F^{n-1}V^{n-1}(\nu) = 2^{n-1}\nu$ and
$F^{n-1}dV^{n-1}(\eta^2) = \eta^3 = 4\nu$. The map
$\overline{F^{n-1}}$ is surjective by~\cite[Table IV]{whitehead}. One
readily verifies that it maps the generator $\iota z_3$ to the modulo
$4$ reduction $\bar{\nu}$ of the Hopf class $\nu$. Hence, $F^{n-1}$
maps one of the classes $\xi_{3,n-1}$ and $\xi_{3,n-1}'$ to $\nu$ and
the other class to $5\nu$. The statement~(ii) follows.

Finally, the statement~(iii) follows immediately from the inverse
limit of the spectral sequences displayed above.
\qed
\end{proof}

The group $\mathbb{H}_5(C_8, T(\mathbb{S}))$ is equal to the direct
sum of the subgroup generated by the class $\xi_{5,3}$ and a cyclic
group. We choose a generator $\rho$ this cyclic group.

\begin{proposition}\label{orbitssphere}The groups
$\mathbb{H}_q(C_{2^{n-1}}, T(\mathbb{S}))$ with $q \leqslant 5$ are
given by
$$\begin{aligned}
\mathbb{H}_0(C_{2^{n-1}},T(\mathbb{S})) & = 
\mathbb{Z}_2 \cdot V^{n-1}(\iota) \cr
\mathbb{H}_1(C_{2^{n-1}}, T(\mathbb{S})) & = \begin{cases}
\mathbb{Z}/2\mathbb{Z} \cdot \eta & \hskip1.8mm (n = 1) \cr
\mathbb{Z}/2^{n-1}\mathbb{Z} \cdot \xi_{1,n-1} \oplus
\mathbb{Z}/2\mathbb{Z} \cdot V^{n-1}(\eta) & \hskip1.8mm (n \geqslant 2) \cr
\end{cases} \cr
\mathbb{H}_2(C_{2^{n-1}}, T(\mathbb{S})) & = \begin{cases}
\mathbb{Z}/2\mathbb{Z} \cdot \eta^2 & \hskip2.7mm (n = 1) \cr
\mathbb{Z}/2\mathbb{Z} \cdot \eta \xi_{1,n-1} \oplus
\mathbb{Z}/2\mathbb{Z} \cdot V^{n-1}(\eta^2) & \hskip2.7mm (n \geqslant 2) \cr
\end{cases} \cr
\mathbb{H}_3(C_{2^{n-1}}, T(\mathbb{S})) & = \begin{cases}
\mathbb{Z}/8\mathbb{Z} \cdot \nu & \hskip4.8mm (n = 1) \cr
\mathbb{Z}/8\mathbb{Z} \cdot \xi_{3,1} \oplus
\mathbb{Z}/8\mathbb{Z} \cdot V(\nu) & \hskip4.8mm (n = 2) \cr
\mathbb{Z}/2^n \mathbb{Z} \cdot \xi_{3,n-1} \oplus
\mathbb{Z}/2\mathbb{Z} \cdot \eta^2\xi_{1,n-1} & \hskip4.8mm (n \geqslant 3) \cr
\hskip20.8mm \oplus \hskip.4mm \mathbb{Z}/8\mathbb{Z} \cdot V^{n-1}(\nu) & \cr
\end{cases} \cr
\mathbb{H}_4(C_{2^{n-1}}, T(\mathbb{S})) & = \begin{cases}
\mathbb{Z}/2^{n-1}\mathbb{Z} \cdot \nu \xi_{1,n-1} & \hskip26mm (n
\leqslant 3) \cr
\mathbb{Z}/8\mathbb{Z} \cdot \nu \xi_{1,n-1} & \hskip26mm (n \geqslant
4) \cr \end{cases} \cr
\mathbb{H}_5(C_{2^{n-1}}, T(\mathbb{S})) & = \begin{cases}
0 & \hskip1.8mm (n \leqslant 2) \cr
\mathbb{Z}/4\mathbb{Z} \cdot \xi_{5,2} & \hskip1.8mm (n = 3) \cr
\mathbb{Z}/2^{n-1}\mathbb{Z} \cdot \xi_{5,n-1} \oplus
\mathbb{Z}/2\mathbb{Z} \cdot V^{n-4}(\rho) & \hskip1.8mm
(n \geqslant 4) \cr \end{cases} \cr
\end{aligned}$$
In addition, $F(\xi_{q,n-1}) = \xi_{q,n-2}$, where $\xi_{1,0} = \eta$
and $\xi_{3,0} = \nu$, and $F(\rho) = 0$.
\end{proposition}

\begin{proof}We have already evaluated the $E^{\infty}$-term of the
spectral sequence 
$$E_{s,t}^2 = H_s(C_{2^{n-1}}, \operatorname{TR}_t^1(\mathbb{S};2))
\Rightarrow \mathbb{H}_{s+t}(C_{2^{n-1}}, T(\mathbb{S})),$$
for $s+t \leqslant 7$. We have also defined all the homotopy classes
that appear in the statement. Hence, it remains only to prove that
these homotopy classes have the indicated order. First, the edge
homomorphism of the spectral sequence is the map
$$V^{n-1} \colon \operatorname{TR}_t^1(\mathbb{S};2) \to
\mathbb{H}_t(C_{2^{n-1}}, T(\mathbb{S})).$$
Since this map has a retraction, the classes $V^{n-1}(\eta)$
and $V^{n-1}(\eta^2)$ both generate a direct summand
$\mathbb{Z}/2\mathbb{Z}$ and the class $V^{n-1}(\nu)$ generates a
direct summand $\mathbb{Z}/8\mathbb{Z}$ as stated. This completes the
proof for $q \leqslant 2$. Next, the Frobenius map
$$F \colon \mathbb{H}_3(C_2, T(\mathbb{S})) \to
\operatorname{TR}_3^1(\mathbb{S};2)$$
is surjective by~\cite[Table~IV]{whitehead}. This implies that the
class $\xi_{3,1}$ has order $8$ and that the group
$\mathbb{H}_3(C_2,T(\mathbb{S}))$ is as stated. We note that
$4\xi_{3,1}$ is congruent to $dV(\eta^2)$ modulo the image of the edge
homomorphism.

Next, we show by induction on $n \geqslant 3$ that the class
$\xi_{3,n-1}$ has order $2^n$. The class $\xi_{3,2}$ has order either
$8$ or $16$, because $F(\xi_{3,2}) = \xi_{3,1}$ has order $8$. If
$\xi_{3,2}$  has order $16$, then the quotient of $\mathbb{H}_3(C_4,
T(\mathbb{S}))$ by the image of the edge homomorphism is equal to
$\mathbb{Z}/16\mathbb{Z}$ generated by the image of $\xi_{3,2}$. But
then $V(\xi_{3,1})$ has order $8$ which contracticts that, modulo the
image of the edge homomorphism,
$$4V(\xi_{3,1}) = V(4\xi_{3,1}) \equiv VdV(\eta^2) = 2dV^2(\eta^2) = 0.$$
Hence, $\xi_{3,2}$ has order $8$, and the group $\mathbb{H}_3(C_4,
T(\mathbb{S}))$ is as stated. So we let $n \geqslant 4$ and assume,
inductively, that $\xi_{3,n-2}$ has order $2^{n-1}$. The class
$2^{n-2}\xi_{3,n-2}$ is represented in the spectral sequence by $\eta
z_2$. Now, by Lemma~\ref{homologyfinite}~(iv), we have $V(\eta z_2) =
\eta z_2$, which shows that the class $2^{n-2}V(\xi_{3,n-2}) =
V(2^{n-2}\xi_{3,n-2})$ is non-zero and represented by $\eta z_2$. This
implies that $2^{n-1}\xi_{3,n-1}$ is non-zero, and hence, $\xi_{3,n-1}$
has order $2^n$ as stated.

Next, we show that, for $n \geqslant 3$, the class $\xi_{5,n-1}$
has order $2^{n-1}$. If $n \geqslant 4$, the spectral sequence shows
that there is an extension
$$0 \to \mathbb{Z}/4\mathbb{Z} \to
\mathbb{H}_5(C_{2^{n-1}}, T(\mathbb{S})) \to
\mathbb{Z}/2^{n-2}\mathbb{Z} \to 0.$$
The Verschiebung map induces a map of extensions from the extension
for $n$ to the extension for $n + 1$, and Lemma~\ref{homologyfinite}
shows that the resulting extension of colimits with respect to the
Verschiebung maps is an extension 
$$0 \to \mathbb{Z}/4\mathbb{Z} \to
\operatornamewithlimits{colim}_{V}
\mathbb{H}_5(C_{2^{n-1}}, T(\mathbb{S})) \to 
\mathbb{Q}_2/\mathbb{Z}_2 \to 0.$$
It follows from~\cite[Lemma~4.4.9]{madsen} that there is a canonical
isomorphism
$$\operatorname{Ext}(\mathbb{Q}_2/\mathbb{Z}_2,
\operatornamewithlimits{colim}_{V} 
\mathbb{H}_5(C_{2^{n-1}}, T(\mathbb{S}))) \xrightarrow{\sim}
\operatornamewithlimits{lim}_{F}
\mathbb{H}_6(C_{2^{n-1}}, T(\mathbb{S}))$$
and, by the proof of Lemma~\ref{generatorsxi}, the
right-hand group is cyclic of order $2$. This implies that the
extension for $n \geqslant 4$ is equivalent to the extension
$$0 \to
\mathbb{Z}/4\mathbb{Z} \xrightarrow{(1,-2^{n-3})} 
\mathbb{Z}/2\mathbb{Z} \oplus \mathbb{Z}/2^{n-1}\mathbb{Z}
\xrightarrow{2^{n-3}+1} 
\mathbb{Z}/2^{n-2}\mathbb{Z} \to 0.$$
It follows that, for $n \geqslant 4$, the class $\xi_{5,n-1}$ has
order $2^{n-1}$ as stated. It remains to prove that $\xi_{5,2}$ has
order $4$. If this is not the case, the map of extensions induced by
the Verschiebung map $V \colon \mathbb{H}_5(C_4, T(\mathbb{S})) \to
\mathbb{H}_5(C_8, T(\mathbb{S}))$ takes the form
$$\xymatrix{
{ 0 } \ar[r] &
{ \mathbb{Z}/2\mathbb{Z} } \ar[r]^(.38){(1,0)} \ar[d]^{2} &
{ \mathbb{Z}/2\mathbb{Z} \oplus \mathbb{Z}/2\mathbb{Z} }
\ar[r]^(.62){0+1} \ar[d]^{V} &
{ \mathbb{Z}/2\mathbb{Z} } \ar[r] \ar[d]^{2} &
{ 0 } \cr
{ 0 } \ar[r] &
{ \mathbb{Z}/4\mathbb{Z} } \ar[r]^(.36){(1,-2)} &
{ \mathbb{Z}/2\mathbb{Z} \oplus \mathbb{Z}/8\mathbb{Z} } 
\ar[r]^(.62){2+1} &
{ \mathbb{Z}/4\mathbb{Z} } \ar[r] &
{ 0 } \cr
}$$
where the middle map $V$ takes $(1,0)$ to $(0,4)$ and $(0,1)$ to
either $(1,0)$ or $(1,4)$. The class $\xi_{5,2}$ corresponds to either
$(0,1)$ or $(1,1)$ in the top middle group. In either case, we find
that the class $V(\xi_{5,2})$ has order $2$ and reduces to a generator
of the quotient of $\mathbb{H}_5(C_8, T(\mathbb{S}))$ by the subgroup
$\mathbb{Z}/8\mathbb{Z} \cdot \xi_{5,3}$. It follows that the class
$$V(\xi_{5,2}) - 2\xi_{5,3} \in \mathbb{H}_5(C_8, T(\mathbb{S}))$$
generates the kernel of the edge homomorphism onto
$\mathbb{Z}/4\mathbb{Z} \cdot 2\iota z_5$. Then,
Lemma~\ref{homologyfinite} shows that the class
$F(V(\xi_{5,2}) - 2\xi_{5,3})$ generates the kernel of the edge
homomorphism from $\mathbb{H}_5(C_4, T(\mathbb{S}))$ onto
$\mathbb{Z}/2\mathbb{Z} \cdot 2\iota z_5$. But $F(V(\xi_{5,2}) -
2\xi_{5,3}) = 0$ which is a contradiction. We conclude that the group
$\mathbb{H}_5(C_4, T(\mathbb{S}))$ is cyclic as stated. 

Finally, the Frobenius map $F \colon \mathbb{H}_5(C_8,
T(\mathbb{S})) \to \mathbb{H}_5(C_4, T(\mathbb{S}))$ induces a map of
extensions which takes the form
$$\xymatrix{
{ 0 } \ar[r] &
{ \mathbb{Z}/4\mathbb{Z} } \ar[r]^(.36){(1,-2)} \ar[d]^{1} &
{ \mathbb{Z}/2\mathbb{Z} \oplus \mathbb{Z}/8\mathbb{Z} } 
\ar[r]^(.62){2+1} \ar[d]^{0 + 1} &
{ \mathbb{Z}/4\mathbb{Z} } \ar[r] \ar[d]^{1} &
{ 0 } \cr
{ 0 } \ar[r] &
{ \mathbb{Z}/2\mathbb{Z} } \ar[r]^{2} &
{ \mathbb{Z}/4\mathbb{Z} } \ar[r]^{1} &
{ \mathbb{Z}/2\mathbb{Z} } \ar[r] &
{ 0 } \cr
}$$
The class $\rho$ corresponds to one of the elements $(1,0)$ or $(1,4)$
of the top middle group both of which map to zero by the middle
vertical map. It follows that $F(\rho)$ is zero as stated.
\qed
\end{proof}

We define $\xi_{q,s} \in
\operatorname{TR}_q^n(\mathbb{S};2)$ to be the image of $\xi_{q,s} \in
\mathbb{H}_q(C_{2^s}, T(\mathbb{S}))$ by the composition of the norm
map and the iterated Segal-tom~Dieck splitting
$$\mathbb{H}_q(C_{2^s}, T(\mathbb{S})) \to
\operatorname{TR}_q^s(\mathbb{S};2) \to
\operatorname{TR}_q^n(\mathbb{S};2).$$
Similarly, we define $\rho \in \operatorname{TR}_5^n(\mathbb{S};2)$ to
be the image of $\rho \in \mathbb{H}_5(C_8, T(\mathbb{S}))$ by the
composition of the norm map and the itereated Segal-tom~Dieck
splitting
$$\mathbb{H}_5(C_8, T(\mathbb{S})) \to
\operatorname{TR}_5^4(\mathbb{S};2) \to
\operatorname{TR}_5^n(\mathbb{S};2).$$
Then we have the following result.

\begin{theorem}\label{TRsphere}The groups
$\operatorname{TR}_q^n(\mathbb{S};2)$ with $q \leqslant 5$ are given
by
$$\begin{aligned}
\operatorname{TR}_0^n(\mathbb{S};2) & = \bigoplus_{0 \leqslant s < n}
\mathbb{Z}_2 \cdot V^s(1) \cr
\operatorname{TR}_1^n(\mathbb{S};2) & = \bigoplus_{0 \leqslant s < n}
\mathbb{Z}/2\mathbb{Z} \cdot V^s(\eta) \oplus
\bigoplus_{1 \leqslant s < n}
\mathbb{Z}/2^s\mathbb{Z} \cdot \xi_{1,s} \cr
\operatorname{TR}_2^n(\mathbb{S};2) & = \bigoplus_{0 \leqslant s < n}
\mathbb{Z}/2\mathbb{Z} \cdot V^s(\eta^2) \oplus
\bigoplus_{1 \leqslant s < n}
\mathbb{Z}/2\mathbb{Z} \cdot \eta\xi_{1,s} \cr
\operatorname{TR}_3^n(\mathbb{S};2) & = 
\bigoplus_{0 \leqslant s < n}
\mathbb{Z}/8\mathbb{Z} \cdot V^s(\nu) \oplus
\bigoplus_{1 \leqslant s < n} 
\mathbb{Z}/2^u\mathbb{Z} \cdot \xi_{3,s} \oplus
\bigoplus_{2 \leqslant s < n} 
\mathbb{Z}/2\mathbb{Z} \cdot \eta^2\xi_{1,s} \cr
\operatorname{TR}_4^n(\mathbb{S};2) & = \bigoplus_{1 \leqslant s < n}
\mathbb{Z}/2^v\mathbb{Z} \cdot \nu\xi_{1,s} \cr
\operatorname{TR}_5^n(\mathbb{S};2) & = \bigoplus_{2 \leqslant s < n}
\mathbb{Z}/2^s\mathbb{Z} \cdot \xi_{5,s} \oplus
\bigoplus_{3 \leqslant s < n}
\mathbb{Z}/2\mathbb{Z} \cdot V^{s-3}(\rho) \cr
\end{aligned}$$
where $u = u(s)$ is the larger of $3$ and $s+1$, and where $v = v(s)$
is the smaller of $3$ and $s$. The restriction map takes $\xi_{q,s}$
to $\xi_{s,q}$, if $s < n-1$, and to zero, if $s = n-1$, and takes
$\rho$ to $\rho$, if $n \geqslant 5$, and to zero, if $n = 4$. The
Frobenius map takes $\xi_{q,s}$ to $\xi_{q,s-1}$, where $\xi_{1,0} =
\eta$ and $\xi_{3,0} = \nu$, and takes $\rho$ to zero. Connes'
operator takes $V^s(1)$ to $\xi_{1,s} + V^s(\eta)$, and takes
$\xi_{1,s}$ to zero.
\end{theorem}

\begin{proof}The Segal-tom~Dieck splitting gives a section of the
restriction map. Hence, the fundamental long-exact sequence
$$\cdots \to \mathbb{H}_q(C_{p^{n-1}},T(\mathbb{S})) \xrightarrow{N}
\operatorname{TR}_q^n(\mathbb{S};p) \xrightarrow{R}
\operatorname{TR}_q^{n-1}(\mathbb{S};p) \xrightarrow{\partial}
 \mathbb{H}_{q-1}(C_{p^{n-1}},T(\mathbb{S})) \to \cdots$$
of Prop.~\ref{fundamentalcofibrationsequence}
breaks into split short-exact sequences and Prop.~\ref{orbitssphere}
then shows that the groups $\operatorname{TR}_q^n(\mathbb{S};2)$ are
as stated. Since the Frobenius map and the Segal-tom~Dieck splitting
commute, the formula for the Frobenius also follows form
Prop.~\ref{orbitssphere}. Finally, from the proof of 
Prop.~\ref{orbitssphere}, we have $\xi_{1,s} = dV^s(1) +
V^s(\eta)$. This implies that
$$d\xi_{1,s} = ddV^s(1) + dV^s(\eta) = dV^s(\eta) + dV^s(\eta) = 0$$
as stated.
\qed
\end{proof}

\begin{remark}We have not determined $\eta \xi_{3,s}$ and
$d\xi_{3,s}$.
\end{remark}

\section{The groups $\operatorname{TR}_q^n(\mathbb{Z};2)$}\label{integerssection}

In this section, we again implicitly consider homotopy groups with
$\mathbb{Z}_2$-coefficients. The groups
$\smash{ \operatorname{TR}_q^1(\mathbb{Z};2) }$ were evaluated by 
B\"{o}kstedt~\cite{bokstedt1}; see
also~\cite[Thm.~1.1]{lindenstraussmadsen}. The group 
$\smash{ \operatorname{TR}_0^1(\mathbb{Z};2) }$ is equal to
$\mathbb{Z}_2$ generated by the multiplicative unit element
$\iota = [1]_1$, and for positive integers $q$, the group
$\smash{ \operatorname{TR}_q^1(\mathbb{Z};2) }$ is finite cyclic of
order
$$|\operatorname{TR}_q^1(\mathbb{Z};2)| = \begin{cases}
2^{v_2(i)} & (\text{$q = 2i-1$ odd}) \cr
1 & (\text{$q$ even}). \cr
\end{cases}$$
We choose a generator $\lambda$ of $K_3(\mathbb{Z})$ such that $2\lambda
= \nu$. Then, by~\cite[Thm.~10.4]{bokstedtmadsen}, the image 
of $\lambda$ by the cyclotomic trace map generates the group
$\operatorname{TR}_3^1(\mathbb{Z};2)$. We also choose a generator
$\gamma$ of the group $\operatorname{TR}_7^1(\mathbb{Z};2)$. We first
derive the following result from Rognes' paper~\cite{rognes}.

\begin{proposition}\label{evenzero}The group
$\operatorname{TR}_q^n(\mathbb{Z};2)$ is zero, for every positive even
integer $q$ and every positive integer $n$.
\end{proposition}

\begin{proof}The group $\operatorname{TR}_q^n(\mathbb{Z};2)$ is
finite, for all positive integers $q$ and $n$. Indeed, this is true,
for $n=1$ by B\"{o}kstedt's result that we recalled above and follows,
inductively, for $n \geqslant 1$, from the fundamental long-exact
sequence of Prop.~\ref{fundamentalcofibrationsequence}, the skeleton
spectral sequence, and the fact that the boundary map
$$\partial \colon \operatorname{TR}_1^{n-1}(\mathbb{Z};2) \to
\mathbb{H}_0(C_{2^{n-1}},T(\mathbb{Z}))$$
in the fundamental long-exact sequence is
zero~\cite[Prop.~3.3]{hm}. Moreover, the group
$\operatorname{TR}_0^n(\mathbb{Z};2)$ is a free
$\mathbb{Z}_2$-module. It follows that, in the strongly convergent
whole plane Bockstein spectral sequence
$$E_{s,t}^2 = \operatorname{TR}_{s+t}^n(\mathbb{Z};2,
2^{-s}\mathbb{Z}/2^{-(s-1)}\mathbb{Z}) \Rightarrow
\operatorname{TR}_{s+t}^n(\mathbb{Z};2,\mathbb{Q}_2)$$
induced from the $2$-adic filtration of $\mathbb{Q}_2$, all elements
of total degree $0$ survive to the $E^{\infty}$-term and all elements
of positive total degree are annihilated by differentials. The
differentials are periodic in the sense that the isomorphism $2 \colon
\mathbb{Q}_2 \to \mathbb{Q}_2$ induces an isomorphism of spectral
sequences
$$\underline{2} \colon E_{s,t}^r \xrightarrow{\sim} E_{s-1,t+1}^r.$$
We recall from~\cite[Lemma~9.4]{rognes} that, for all positive
integers $n$ and $i$,
$$\dim_{\mathbb{F}_2}
\operatorname{TR}_{2i-1}^n(\mathbb{Z};2,\mathbb{F}_2) =
\dim_{\mathbb{F}_2} 
\operatorname{TR}_{2i}^n(\mathbb{Z};2,\mathbb{F}_2).$$
Using this result, we show, by induction on $i \geqslant 1$, that
every element of total degree $2i-1$ is an infinite cycle and that
every non-zero element of total degree $2i$ supports a non-zero
differential. The proof of the case $i = 1$ and of the induction step
are similar. The statement that every element in total degree $2i-1$
is an infinite cycle follows, for $i = 1$, from the fact that every
element of total degree $0$ survives to the $E^{\infty}$-term, and for
$i > 1$, from the inductive hypothesis that every non-zero element of
total degree $2i-2$ supports a non-zero differential. Since no element
of total degree $2i-1$ survives to the $E^{\infty}$-term, it is hit by
a differential supported on an element of total degree $2i$. Since
the differentials are periodic and $E_{s,2i-1-s}^2$ and $E_{s,2i-s}^2$
have the same dimension, we find that non-zero every element of total
degree $2i$ supports a non-zero differential as stated.

Finally, we consider the strongly convergent left half-plane Bockstein
spectral sequence induced from the $2$-adic filtration of
$\mathbb{Z}_2$,
$$E_{s,t}^2 = \operatorname{TR}_{s+t}^n(\mathbb{Z};2,
2^{-s}\mathbb{Z}/2^{-(s-1)}\mathbb{Z}) \Rightarrow
\operatorname{TR}_{s+t}^n(\mathbb{Z};2,\mathbb{Z}_2).$$
The differentials in this spectral sequence are obtained by
restricting the differentials in the whole plan Bockstein spectral
sequence above. It follows that in this spectral sequence, too, every
non-zero element of positive even total degree supports a non-zero
differential. This completes the proof.
\qed
\end{proof}

\begin{remark}The same argument based on B\"{o}kstedt and Madsen's
paper~\cite{bokstedtmadsen}, shows that, for an odd prime $p$, the
groups $\operatorname{TR}_q^n(\mathbb{Z};p)$ are zero, for every
positive even integer $q$ and every positive integer $n$.
\end{remark}

We next consider the skeleton spectral sequence
$$E_{s,t}^2 = H_s(C_{2^{n-1}}, \operatorname{TR}_t^1(\mathbb{Z};2))
\Rightarrow \mathbb{H}_{s+t}(C_{2^{n-1}},T(\mathbb{Z})).$$
The $E^2$-term, for $s + t \leqslant 7$, takes the form
\begin{center}
\begin{tabular*}{0.95\textwidth}{@{\extracolsep{\fill}}ccccccccc}
$\mathbb{Z}/4\mathbb{Z}$ &  &  &  &  &  &  &  \cr
0 & 0 &  &  &  &  &  & \cr
0 & 0 & 0 &  &  &  &  & \cr
0 & 0 & 0 & 0 &  &  &  & \cr
$\mathbb{Z}/2\mathbb{Z}$ & $\mathbb{Z}/2\mathbb{Z}$ &
$\mathbb{Z}/2\mathbb{Z}$ & $\mathbb{Z}/2\mathbb{Z}$ &  
$\mathbb{Z}/2\mathbb{Z}$ & &  &  \cr
0 & 0 & 0 & 0 & 0 & 0 &  &  \cr
0 & 0 & 0 & 0 & 0 & 0 & 0 &  \cr
$\mathbb{Z}_2$ & $\mathbb{Z}/2^{n-1}\mathbb{Z}$ &
0 & $\mathbb{Z}/2^{n-1}\mathbb{Z}$ &
0 & $\mathbb{Z}/2^{n-1}\mathbb{Z}$ &
0 & $\mathbb{Z}/2^{n-1}\mathbb{Z}$ \cr
\end{tabular*}
\end{center}
The group $\smash{ E_{s,0}^2 }$ is generated by $\iota z_s$ and the
group $\smash{ E_{s,3}^2 }$ is generated by $\lambda z_s$. The group
$\smash{ E_{s,7}^2 }$ is generated by $\gamma z_s$, if $s = 0$ or if
$s$ is odd of if $n > 1$, and is generated by $2\gamma z_s$, if $n =
1$ and $s$ is positive and even. It follows
from~\cite[Thm.~8.14]{rognes} that the group
$\mathbb{H}_4(C_{2^{n-1}},T(\mathbb{Z});\mathbb{F}_2)$ is an 
$\mathbb{F}_2$-vector space of dimension $1$. This implies that
$\smash{ d^4(\iota z_5) = \lambda z_1 }$. On the other hand,
$\smash{ d^4(\iota z_7) = 0 }$, since $\iota z_7$ survives to the
$E^4$-term of the skeleton spectral sequence for
$\smash{ \mathbb{H}_q(C_{2^{n-1}}, T(\mathbb{S})) }$ and is a
$\smash{ d^4 }$-cycle. This shows that the $\smash{ E^5 }$-term for $s + t
\leqslant 7$ is given by
\begin{center}
\begin{tabular*}{0.95\textwidth}{@{\extracolsep{\fill}}ccccccccc}
$\mathbb{Z}/4\mathbb{Z}$ &  &  &  &  &  &  &  \cr
0 & 0 &  &  &  &  &  & \cr
0 & 0 & 0 &  &  &  &  & \cr
0 & 0 & 0 & 0 &  &  &  & \cr
$\mathbb{Z}/2\mathbb{Z}$ & 0 &
$\mathbb{Z}/2\mathbb{Z}$ & $\mathbb{Z}/2\mathbb{Z}$ &  
$\mathbb{Z}/2\mathbb{Z}$ & &  &  \cr
0 & 0 & 0 & 0 & 0 & 0 &  &  \cr
0 & 0 & 0 & 0 & 0 & 0 & 0 &  \cr
$\mathbb{Z}_2$ & $\mathbb{Z}/2^{n-1}\mathbb{Z}$ &
0 & $\mathbb{Z}/2^{n-1}\mathbb{Z}$ &
0 & $2\mathbb{Z}/2^{n-1}\mathbb{Z}$ &
0 & $\mathbb{Z}/2^{n-1}\mathbb{Z}$ \cr
\end{tabular*}
\end{center}
We claim that the differential $d^5 \colon E_{5,3}^5 \to E_{0,7}^5$ is
zero. Indeed, let
$$\smash{ F^{m-n} \colon {}'E_{s,t}^r \to E_{s,t}^r }$$
be the map of spectral sequences induced by the iterated Frobenius map
$$F^{m-n} \colon \mathbb{H}_q(C_{2^{m-1}}, T(\mathbb{Z})) \to
\mathbb{H}_q(C_{2^{n-1}}, T(\mathbb{Z})).$$
It follows from Lemma~\ref{homologyfinite} that the map $\smash{ F^{m-n}
\colon {}'E_{5,3}^5 \to E_{5,3}^5 }$ is an isomorphism and that, for
$m \geqslant n+2$, the map $\smash{ F^{m-n} \colon {}'E_{5,3}^5 \to
E_{5,3}^5 }$ is zero. Hence, the differential in question is zero as
claimed. It follows that the $\smash{ E^5 }$-term of the spectral
sequence is also the $\smash{ E^{\infty} }$-term. 

We choose a generator $\kappa$ of the infinite cyclic group
$K_5(\mathbb{Z})$ and recall the generator $\lambda$ of the group
$K_3(\mathbb{Z})$. We continue to write $\lambda$ and $\kappa$ for the
images of $\lambda$ and $\kappa$ in
$\operatorname{TR}_3^n(\mathbb{Z};2)$ and
$\operatorname{TR}_5^n(\mathbb{Z};2)$ by the cyclotomic trace
map. The norm map
$$\mathbb{H}_5(C_2, T(\mathbb{Z})) \to 
\operatorname{TR}_5^2(\mathbb{Z};2)$$
is an isomorphism, and we will also write $\kappa$ for the unique
class on the left-hand side whose image by the norm map is the class
$\kappa$ on the right-hand side. Finally, we continue to write
$\xi_{q,n} \in \mathbb{H}_q(C_{2^{n-1}}, T(\mathbb{Z}))$ for the
image by the map induced from the Hurewicz map $\ell \colon \mathbb{S}
\to \mathbb{Z}$ of the class $\xi_{q,n} \in \mathbb{H}_q(C_{2^{n-1}},
T(\mathbb{S}))$.

\begin{proposition}\label{orbitsintegers}The groups
$\mathbb{H}_q(C_{2^{n-1}}, T(\mathbb{Z}))$ with $q \leqslant 6$ are
given by
$$\begin{aligned}
\mathbb{H}_0(C_{2^{n-1}},T(\mathbb{Z})) & = 
\mathbb{Z}_2 \cdot V^{n-1}(\iota) \cr
\mathbb{H}_1(C_{2^{n-1}}, T(\mathbb{Z})) & = \begin{cases}
0 & \hskip26.2mm (n = 1) \cr
\mathbb{Z}/2^{n-1}\mathbb{Z} \cdot \xi_{1,n-1} & \hskip26.2mm
(n \geqslant 2) \cr 
\end{cases} \cr
\mathbb{H}_2(C_{2^{n-1}}, T(\mathbb{Z})) & = 0 \cr
\mathbb{H}_3(C_{2^{n-1}}, T(\mathbb{Z})) & = \begin{cases}
\mathbb{Z}/2\mathbb{Z} \cdot \lambda & \hskip29.6mm (n = 1) \cr
\mathbb{Z}/2^n \mathbb{Z} \cdot \xi_{3,n-1} & \hskip29.6mm (n \geqslant 2) \cr
\end{cases} \cr
\mathbb{H}_4(C_{2^{n-1}}, T(\mathbb{Z})) & = 0 \cr
\mathbb{H}_5(C_{2^{n-1}}, T(\mathbb{Z})) & = \begin{cases}
0 & (n = 1) \cr
\mathbb{Z}/2\mathbb{Z} \cdot \kappa & (n = 2) \cr
\mathbb{Z}/2^{n-2}\mathbb{Z} \cdot \xi_{5,n-1} \oplus
\mathbb{Z}/2\mathbb{Z} \cdot V^{n-2}(\kappa) & (n \geqslant 3) \cr
\end{cases} \cr
\mathbb{H}_6(C_{2^{n-1}}, T(\mathbb{Z})) & = \begin{cases}
0 & \hskip25.5mm (n = 1) \cr
\mathbb{Z}/2\mathbb{Z} \cdot dV^{n-2}(\kappa) & \hskip25.5mm
(n \geqslant 2) \cr
\end{cases} \cr
\end{aligned}$$
In addition, $F(\xi_{q,n-1}) = \xi_{q,n-2}$, where $\xi_{1,0}$,
$\xi_{3,0}$, and $\xi_{5,0}$ are zero.
\end{proposition}

\begin{proof}The cases $q = 0$ and $q = 1$ follow immediately
from the spectral sequence above and from the fact that the
map $\operatorname{TR}_0^1(\mathbb{S};2) \to
\operatorname{TR}_0^1(\mathbb{Z};2)$ induced  by the Hurewicz map is
an isomorphism. The cases $q = 2$ and $q = 4$ follow directly from the
spectral sequence above. It follows from~\cite[Thm.~8.14]{rognes} that
$\mathbb{H}_3(C_{2^{n-1}}, T(\mathbb{Z}); \mathbb{F}_2)$ is an
$\mathbb{F}_2$-vector space of dimension $1$, for all $n \geqslant
1$. The statement for $q = 3$ follows. It also follows from
loc.~cit.~that $\mathbb{H}_5(C_{2^{n-1}}, T(\mathbb{Z});
\mathbb{F}_2)$ is an $\mathbb{F}_2$-vector space of dimension $0$, if
$n = 1$, dimension $1$, if $n = 2$, and dimension $2$, if $n \geqslant
3$. Hence, to prove the statement for $q = 5$, it will suffice to show
that the group $\mathbb{H}_5(C_2, T(\mathbb{Z}))$ is generated by the
class $\kappa$, or equivalently, that the composition
$$K_5(\mathbb{Z}) \to \operatorname{TC}_5^2\mathbb{Z};2) \to
\operatorname{TR}_5^2(\mathbb{Z};2) \xrightarrow{\sim}
\operatorname{TR}_5^2(\mathbb{Z};2,\mathbb{F}_2)$$
of the cyclotomic trace map and the modulo $2$ reduction map is
surjective. But this is the statement that $i_1(\kappa) = \xi_5(0)$ 
in~\cite[Prop.~4.2]{rognes1}. (Here $\xi_5(0)$ is name given in
loc.~cit.~to the generator of the right-hand group; it is unrelated to
the class $\xi_{5,0}$.) Finally, the statement for $q = 6$ follows
from~\cite[Prop.~4.4.1]{hm4}.
\qed
\end{proof}

\begin{corollary}\label{cokernellambda}The cokernel of the map induced
by the Hurewicz map 
$$\ell \colon \operatorname{TR}_3^n(\mathbb{S};2) \to 
\operatorname{TR}_3^n(\mathbb{Z};2)$$
is equal to $\mathbb{Z}/2\mathbb{Z} \cdot \lambda$.
\end{corollary}

\begin{proof}The proof is by induction on $n \geqslant 1$. In the case
$n = 1$, the Hurewicz map induces the zero map
$\smash{ \operatorname{TR}_q^1(\mathbb{S};2) \to
\operatorname{TR}_q^1(\mathbb{Z};2) }$, for all positive 
integers $q$. Indeed, the spectrum $\smash{
\operatorname{TR}^1(\mathbb{Z};2) }$ is a module spectrum over the
Eilenberg-MacLane spectrum for $\mathbb{Z}$ and therefore is weakly
equivalent to a product of Eilenberg-MacLane spectra. As we recalled
above, $\operatorname{TR}_3^1(\mathbb{Z};2) = \mathbb{Z}/2\mathbb{Z}
\cdot \lambda$, which proves the case $n = 1$. To prove the induction
step, we use that the Hurewicz map induces a map of fundamental
long-exact sequences which takes the form
$$\xymatrix{
{ 0 } \ar[r] &
{ \mathbb{H}_3(C_{2^{n-1}}, T(\mathbb{S})) } \ar[r] \ar[d] &
{ \operatorname{TR}_3^n(\mathbb{S};2) } \ar[r] \ar[d] &
{ \operatorname{TR}_3^{n-1}(\mathbb{S};2) } \ar[r] \ar[d] &
{ 0 } \cr
{ 0 } \ar[r] &
{ \mathbb{H}_3(C_{2^{n-1}}, T(\mathbb{Z})) } \ar[r] &
{ \operatorname{TR}_3^n(\mathbb{Z};2) } \ar[r] &
{ \operatorname{TR}_3^{n-1}(\mathbb{Z};2) } \ar[r] &
{ 0 } \cr
}$$
The zero on the lower right-hand side follows from
Prop.~\ref{orbitsintegers}, and the zero on the lower left-hand side 
from Prop.~\ref{evenzero}. Since Props.~\ref{orbitssphere}
and~\ref{orbitsintegers} show that the left-hand vertical map is
surjective, the induction step follows.
\qed
\end{proof}

We owe the proof of the following result to Marcel B\"{o}kstedt.

\begin{lemma}\label{bokstedtlemma}The square of homotopy groups with
$\mathbb{Z}_2$-coefficients
$$\xymatrix{
{ K_5(\mathbb{S};\mathbb{Z}_2) } \ar[r] \ar[d] &
{ K_5(\mathbb{S}_2;\mathbb{Z}_2) } \ar[d] \cr
{ K_5(\mathbb{Z};\mathbb{Z}_2) } \ar[r] &
{ K_5(\mathbb{Z}_2;\mathbb{Z}_2), } \cr
}$$
where the vertical maps are induced by the Hurewicz maps and the
horizontal maps are induced by the completion maps, takes the from
$$\xymatrix{
{ \mathbb{Z}_2 \cdot 8\kappa } \ar[r] \ar[d] &
{ \mathbb{Z}_2 \cdot (4\kappa + \tau) } \ar[d] \cr
{ \mathbb{Z}_2 \cdot \kappa } \ar[r] &
{ \mathbb{Z}_2 \cdot \kappa \oplus \mathbb{Z}/2\mathbb{Z} \cdot \tau } \cr
}$$
\end{lemma}

\begin{proof}It was proved in~\cite[Prop.~4.2]{rognes1} that the group 
$K_5(\mathbb{Z}_2;\mathbb{Z}_2)$ is the direct sum of a free
$\mathbb{Z}_2$-module of rank one generated by $\kappa$ and a torsion
subgroup of order $2$; the class $\tau$ is the unique generator of the
torsion subgroup. Moreover,~\cite[Thm.~5.8]{rognes2}
shows that the group $K_5(\mathbb{S};\mathbb{Z}_2)$ is a free
$\mathbb{Z}_2$-module of rank one, and~\cite[Thm.~2.11]{rognes2}
and~\cite[Thm.~5.17]{bokstedthsiangmadsen} show that the group 
$K_5(\mathbb{S}_2;\mathbb{Z}_2)$ is a free $\mathbb{Z}_2$-module of rank 
one. To complete the proof of the lemma, it remains to show that the
left-hand vertical map in the diagram in the statement is equal to the
inclusion of a subgroup of index $8$. This is essentially proved
in~\cite{bokstedt2} as we now explain. In op.~cit., B\"{o}kstedt
constructs a homotopy commutative diagram of pointed spaces
$$\xymatrix{
{ G/O } \ar[r]^{e} \ar[d] &
{ \operatorname{Fib}(s) } \ar[d] &
{ \operatorname{Fib}(t) } \ar[l]_{g} \ar[d] \cr
{ BSO } \ar@{=}[r] \ar[d] &
{ BSO } \ar@{=}[r] \ar[d]^{s} &
{ BSO } \ar[d]^{t} \cr
{ BSG } \ar[r]^{\eta} &
{ SG } &
{ SJ } \ar[l] \cr
}$$
in which the columns are fibration sequences. The induced diagram of
$4$th homotopy groups with $\mathbb{Z}_2$-coefficients is isomorphic to the
diagram
$$\xymatrix{
{ \mathbb{Z}_2 } \ar[r]^{8} \ar[d]^{8} &
{ \mathbb{Z}_2 } \ar[d]^{\operatorname{id}} &
{ \mathbb{Z}_2 } \ar[l]_{\operatorname{id}} \ar[d]^{\operatorname{id}} \cr
{ \mathbb{Z}_2 } \ar@{=}[r] \ar[d] &
{ \mathbb{Z}_2 } \ar@{=}[r] \ar[d] &
{ \mathbb{Z}_2 } \ar[d] \cr
{ \mathbb{Z}/8\mathbb{Z} } \ar[r] &
{ 0 } &
{ 0. } \ar[l] \cr
}$$
We compare this diagram to the following diagram constructed by
Waldhausen.
$$\xymatrix{
{} &
{ G/O } \ar[r]^(.43){f} \ar[d]^{e} &
{ \Omega \operatorname{Wh}^{\operatorname{Diff}}(*) } \ar[d] & 
{} \cr
{ \operatorname{Fib}(t) } \ar[r]^{g} &
{ \operatorname{Fib}(s) } \ar[r] &
{ \Omega K(\mathbb{Z}) } \ar[r] &
{ \Omega JK(\mathbb{Z}). } \cr
}$$
It is proved in~\cite[p.~30]{bokstedt2} that the composition of the
lower horizontal maps in this diagram becomes a weak equivalence after
$2$-completion. Moreover, it is proved in~\cite[Thm.~7.5]{rognes2}
that the upper horizontal map induces an isomorphism of homotopy
groups with $\mathbb{Z}_2$-coefficients in degrees less than or equal to
$8$. Hence, the induced diagram of $4$th homotopy groups with
$\mathbb{Z}_2$-coefficients is isomorphic to the diagram
$$\xymatrix{
{} &
{ \mathbb{Z}_2 } \ar[r]^{\operatorname{id}} \ar[d]^{8} &
{ \mathbb{Z}_2 } \ar[d]^{8} & 
{} \cr
{ \mathbb{Z}_2 } \ar[r]^{\operatorname{id}} &
{ \mathbb{Z}_2 } \ar[r]^{\operatorname{id}} &
{ \mathbb{Z}_2 } \ar[r]^{\operatorname{id}} &
{ \mathbb{Z}_2. } \cr
}$$
The right-hand vertical map in this diagram is induced by the
composition
$$\operatorname{Wh}^{\operatorname{Diff}}(*) \to
K(\mathbb{S}) \to K(\mathbb{Z})$$
of the canonical section of the canonical map $K(\mathbb{S}) \to
\operatorname{Wh}^{\operatorname{Diff}}(*)$ and the map induced by the
Hurewicz map. The left-hand map induces an isomorphism of fifth
homotopy groups with $\mathbb{Z}_2$-coefficients because
$\pi_5(\mathbb{S};\mathbb{Z}_2)$ is zero. This completes the proof
that the map induced by the Hurewicz map $K_5(\mathbb{S};\mathbb{Z}_2)
\to K_5(\mathbb{Z};\mathbb{Z}_2)$ is the inclusion of an index eigth
subgroup. The lemma follows.
\qed
\end{proof}

We define the class $\xi_{q,s} \in 
\operatorname{TR}_q^n(\mathbb{Z};2)$ to be the image of the class
$\xi_{q,s} \in \operatorname{TR}_q^n(\mathbb{S};2)$ by the map induced
by the Hurewicz map.

\begin{theorem}\label{TRintegers}The groups
$\operatorname{TR}_q^n(\mathbb{Z};2)$ with $q \leqslant 6$ are given
by
$$\begin{aligned}
\operatorname{TR}_0^n(\mathbb{Z};2) & = \bigoplus_{0 \leqslant s < n}
\mathbb{Z}_2 \cdot V^s(1) \cr
\operatorname{TR}_1^n(\mathbb{Z};2) & = \bigoplus_{1 \leqslant s < n}
\mathbb{Z}/2^s\mathbb{Z} \cdot \xi_{1,s} \cr
\operatorname{TR}_2^n(\mathbb{Z};2) & = 0 \cr
\operatorname{TR}_3^n(\mathbb{Z};2) & = \begin{cases}
\mathbb{Z}/2\mathbb{Z} \cdot \lambda & (n = 1) \cr
\mathbb{Z}/8\mathbb{Z} \cdot \lambda \oplus \displaystyle{
\bigoplus_{2 \leqslant s < n}
\mathbb{Z}/2^{s+1}\mathbb{Z} \cdot \xi_{3,s} } & (n \geqslant 2) \cr
\end{cases} \cr
\operatorname{TR}_4^n(\mathbb{Z};2) & = 0 \cr
\operatorname{TR}_5^n(\mathbb{Z};2) & = 
\mathbb{Z}/2^{n-1}\mathbb{Z} \cdot \kappa \oplus \bigoplus_{2
  \leqslant s < n} \mathbb{Z}/2^{s-1}\mathbb{Z} \cdot (\xi_{5,s} +
\dots + \xi_{5,n-1} + 4u \kappa) \cr
\operatorname{TR}_6^n(\mathbb{Z};2) & = 0 \cr
\end{aligned}$$
where $u \in \mathbb{Z}_2^*$ is a unit.
\end{theorem}

\begin{proof}The map induced by the Hurewicz map is an isomorphism,
for $q = 0$, so the statement for the group
$\operatorname{TR}_0^n(\mathbb{Z};2)$ follows from
Thm.~\ref{TRsphere}. The statement for $q = 1$ follows 
from Prop.\ref{orbitsintegers} and from the fact that the generator
$\xi_{1,s}$ is annihilated by $2^s$. For $q = 3$, the case $n = 1$ was
recalled at the beginning of the section, so suppose that $n \geqslant
2$. We know from Prop.~\ref{orbitsintegers} that the two sides of the
statement are groups of the same order. We also know that both groups
are the direct sum of $n-1$ cyclic groups. Indeed, this is trivial,
for the right-hand side, and is proved in~\cite[Lemma~9.4]{rognes},
for the left-hand side. Now, it follows from Thm.~\ref{TRsphere} that
$\xi_{3,s}$ is annihilated by $2^{s+1}$, so it suffices to show that
$\lambda$ is annihilated by $8$. We have a commutative diagram
$$\xymatrix{
{ K_3(\mathbb{Z};\mathbb{Z}_2) } \ar[r] \ar[d] &
{ \operatorname{TR}_3^n(\mathbb{Z};2,\mathbb{Z}_2) } \ar[d] \cr
{ K_3(\mathbb{Z}_2;\mathbb{Z}_2) } \ar[r] &
{ \operatorname{TR}_3^n(\mathbb{Z}_2;2,\mathbb{Z}_2) } \cr
}$$
where the horizontal maps are the cyclotomic trace maps, where the
vertical maps are induced by the completion maps, and where we have
explicitly indicated that we are considering the homotopy groups with
$\mathbb{Z}_2$-coefficients. The right-hand vertical map is an
isomorphism by~\cite[Addendum~6.2]{hm}. Therefore, it suffices to show
that the image of $\lambda$ in $K_3(\mathbb{Z}_2;\mathbb{Z}_2)$ has
order $8$. But this is proved in~\cite[Prop.~4.2]{rognes1}. 

It remains to prove the statement for $q = 5$. We first show that
$\operatorname{TR}_5^n(\mathbb{Z};2)$ is generated by the classes 
$\kappa$, $\xi_{5,2}$, \dots, $\xi_{5,n-1}$, or equivalently, that the
group
$$\operatorname{TR}_5^n(\mathbb{Z};2) /
2\operatorname{TR}_5^n(\mathbb{Z};2)
\xrightarrow{\sim}
\operatorname{TR}_5^n(\mathbb{Z};2,\mathbb{F}_2)$$
is generated by the images of the classes $\kappa$, $\xi_{5,2}$,
\dots, $\xi_{5,n-1}$. We prove this by induction on $n \geqslant
2$. The case $n = 2$ is true, so we assume the statement for $n-1$ and
prove it for $n$. The fundamental long-exact sequence takes the form
$$\mathbb{H}_5(C_{2^{n-1}}, T(\mathbb{Z}); \mathbb{F}_2) \xrightarrow{N} 
\operatorname{TR}_5^n(\mathbb{Z};2, \mathbb{F}_2) \xrightarrow{R}
\operatorname{TR}_5^{n-1}(\mathbb{Z};2, \mathbb{F}_2) \to 0.$$
Inductively, the right-hand group is generated by the classes $\kappa$,
$\xi_{5,2}$, \dots, $\xi_{5,n-2}$, which are the images by the
restriction map of the classes $\kappa$, $\xi_{5,2}$, \dots,
$\xi_{5,n-2}$ in the middle group. Moreover,
Prop.~\ref{orbitsintegers} shows that the left-hand group is generated
by the classes $V^{n-2}(\kappa)$ and $\xi_{5,n-1}$. Hence, it will
suffice to show that, for $n \geqslant 3$, the image of the class
$V^{n-2}(\kappa)$ in
$\operatorname{TR}_5^n(\mathbb{Z};2,\mathbb{F}_2)$ is zero. This
follows from~\cite[Thm.~8.14]{rognes} as we now explain. We have the
commutative diagram with exact rows
$$\xymatrix{
{ \operatorname{TR}_{q+1}^{n-1}(\mathbb{Z};2,\mathbb{F}_2) }
\ar[r]^(.47){\partial} \ar[d]^{\hat{\Gamma}} &
{ \mathbb{H}_q(C_{2^{n-1}},T(\mathbb{Z});\mathbb{F}_2) } \ar[r]^(.56){N}
\ar@{=}[d] &
{ \operatorname{TR}_q^n(\mathbb{Z};2,\mathbb{F}_2) } \ar[d]^{\Gamma}
\cr
{ \hat{\mathbb{H}}^{-q-1}(C_{2^{n-1}},T(\mathbb{Z});\mathbb{F}_2) }
\ar[r]^(.54){\partial^h} &
{ \mathbb{H}_q(C_{2^{n-1}},T(\mathbb{Z});\mathbb{F}_2) } \ar[r]^{N^h} &
{ \mathbb{H}^{-q}(C_{2^{n-1}},T(\mathbb{Z});\mathbb{F}_2) } \cr
}$$
considered first in~\cite[(6.1)]{bokstedtmadsen}. It is follows
from~\cite[Thms.~0.2, 0.3]{rognes} that the left-hand vertical map
$\hat{\Gamma}$ is an isomorphism, for all integers $q + 1 \geqslant 0$
and $n \geqslant 1$. Hence, it suffices to show that the class
$V^{n-2}(\kappa)$ in the lower middle group is in the image of the
lower left-hand horizontal map $\partial^h$. The lower left-hand group
is the abutment of the strongly convergent, upper half-plan Tate
spectral sequence
$$\hat{E}_{s,t}^2 = \hat{H}^{-s}(C_{2^{n-1}},
\operatorname{TR}_t^1(\mathbb{Z};2,\mathbb{F}_2)) 
\Rightarrow \hat{\mathbb{H}}^{-s-t}(C_{2^{n-1}}, T(\mathbb{Z});
\mathbb{F}_2),$$
and the middle groups are the abutment of the strongly convergent,
first quadrant skeleton spectral sequence
$$E_{s,t}^2 = H_s(C_{2^{n-1}},
\operatorname{TR}_t^1(\mathbb{Z};2,\mathbb{F}_2))
\Rightarrow
\mathbb{H}_{s+t}(C_{2^{n-1}},T(\mathbb{Z});\mathbb{F}_2).$$
Moreover, the map $\partial^h$ induces a map of spectral sequences
$$\partial^{h,r} \colon \hat{E}_{s,t}^r \to E_{s-1,t}^r$$
which is an isomorphism, for $r = 2$ and $s \geqslant 1$. Suppose
that the homotopy class $\tilde{x} \in
\mathbb{H}_q(C_{2^{n-1}},T(\mathbb{Z});\mathbb{F}_2)$ is represented 
by the infinite cycle $x \in E_{s,t}^2$, and let $y \in
\hat{E}_{s+1,t}^2$ be the unique element with
$\smash{ \partial^{h,2}(y) = x }$. Then, if $y$ is an infinite
cycle, there exists a homotopy class $\tilde{y} \in
\hat{\mathbb{H}}^{-q-1}(C_{2^{n-1}},T(\mathbb{Z});\mathbb{F}_2)$
represented by $y$ such that $\partial^h(\tilde{y}) = \tilde{x}$;
compare~\cite[Thm.~2.5]{bokstedtmadsen}. We now return
to~\cite[Thm.~8.14]{rognes}. The homotopy class $V^{n-2}(\kappa)$ is
represented by the unique generator of $E_{2,3}^2$ which, in turn, is
the image by the map $\smash{ \partial^{h,2} }$ of the unique
generator of $\smash{ \hat{E}_{3,3}^{h,2} }$. In loc.~cit., the latter
generator is given the name $\smash{ u_{n-1}t^{-2}e_3 }$ and proved to
be an infinite cycle for $n \geqslant 3$. This shows that the image of
the class $V^{n-2}(\kappa)$ by the norm map
$$N \colon \mathbb{H}_5(C_{2^{n-1}},T(\mathbb{Z});\mathbb{F}_2) \to
\operatorname{TR}_5^n(\mathbb{Z};2,\mathbb{F}_2)$$
is zero as stated. We conclude that $\kappa,\xi_{5,2}, \dots,
\xi_{5,n-1}$ generate $\operatorname{TR}_5^n(\mathbb{Z};2)$.

We know from Thm.~\ref{TRsphere} that $\xi_{5,s}$ is annihilated by
$2^s$ and further claim that $\kappa$ is annihilated by $2^{n-1}$ and
that, for some unit $u \in \mathbb{Z}_2^*$, 
$$2 \cdot (\xi_{5,2} + \dots + \xi_{5,n-1} + 4u\kappa) = 0.$$
This implies the statement of the theorem for $q = 5$. Indeed, the
abelian group generated by $\kappa, \xi_{5,2}, \dots, \xi_{5,n-1}$ and
subject to the relations above is equal to
$$\mathbb{Z}/2^{n-1}\mathbb{Z} \cdot \kappa \oplus \bigoplus_{2
  \leqslant s < n} \mathbb{Z}/2^{s-1}\mathbb{Z} \cdot (\xi_{5,s} +
\dots + \xi_{5,n-1} + 4u\kappa)$$
and surjects onto $\operatorname{TR}_5^n(\mathbb{Z};2)$. Hence, it
suffices to show that the two groups have the same order. But this
follows by an induction argument based on the exact sequence
$$0 \to \mathbb{H}_5(C_{2^{n-1}},T(\mathbb{Z})) \to
\operatorname{TR}_5^n(\mathbb{Z};2) \to
\operatorname{TR}_5^{n-1}(\mathbb{Z};2) \to 0$$
and Prop.~\ref{orbitsintegers} above.

It remains to prove the claim. We first show that the class $2^{n-1}
\cdot \kappa$ is zero by induction on $n \geqslant 2$. The case $n =
2$ is true, so we assume the statment for $n-1$ and prove it for $n$. 
We again use the exact sequence
$$0 \to \mathbb{H}_5(C_{2^{n-1}},T(\mathbb{Z})) \to
\operatorname{TR}_5^n(\mathbb{Z};2) \to
\operatorname{TR}_5^{n-1}(\mathbb{Z};2) \to 0$$
and the calculation of the left-hand group in
Prop.~\ref{orbitsintegers}. The inductive hypothesis implies that the
image of the class $2^{n-2} \cdot \kappa$ by the right-hand map is
zero, and hence, this class is in the image of the left-hand map. It
follows that we can write
$$2^{n-2} \cdot \kappa = a \cdot \xi_{5,n-1} + b \cdot
V^{n-2}(\kappa)$$
with $a \in \mathbb{Z}/2^{n-2}\mathbb{Z}$ and $b \in
\mathbb{Z}/2\mathbb{Z}$. We apply the Frobenius map
$$F \colon \operatorname{TR}_5^n(\mathbb{Z};2) \to 
\operatorname{TR}_5^{n-1}(\mathbb{Z};2)$$
to this equation. The image of the left-hand side is zero, by
induction, and the image of the right-hand side is $\bar{a} \cdot
\xi_{5,n-2}$, where $\bar{a} \in \mathbb{Z}/2^{n-3}\mathbb{Z}$ is
reduction of $a$ modulo $2^{n-3}$. It follows that $\bar{a}$ is zero,
or equivalently, that $a \in
2^{n-3}\mathbb{Z}/2^{n-2}\mathbb{Z}$. This shows that $2^{n-1} \cdot
\kappa$ is zero as desired.

Finally, to prove the relation $2 \cdot (\xi_{5,2} + \dots \xi_{5,n-1}
+ 4u\kappa) = 0$, we consider the following long-exact sequence
$$\cdots \to \operatorname{TR}_6(\mathbb{S};2) \xrightarrow{1-F}
\operatorname{TR}_6(\mathbb{S};2) \xrightarrow{\partial}
K_5(\mathbb{S}_2;\mathbb{Z}_2) \xrightarrow{\operatorname{tr}}
\operatorname{TR}_5(\mathbb{S};2) \xrightarrow{1-F}
\operatorname{TR}_5(\mathbb{S};2) \to \cdots$$
We know from Lemma~\ref{bokstedtlemma} above that the group
$K_5(\mathbb{S}_2;\mathbb{Z}_2)$ is a free $\mathbb{Z}_2$-module of
rank one generated by the class $4\kappa + \tau$. Moreover, it follows
from Thm.~\ref{TRsphere} that the left-hand map $1-F$ is surjective
and that the kernel of the right-hand map $1-F$ is isomorphic to a free
$\mathbb{Z}_2$-module of rank one generated by the element $\Delta =
( \Delta^{(n)} )$ with $\Delta^{(n)} = \xi_{5,2} + \dots + \xi_{5,n-1}$.
It follows that there exists a unit $u \in \mathbb{Z}_2^*$ such that
$\Delta + u(4\kappa + \tau) = 0$ in
$\operatorname{TR}_5(\mathbb{S};2)$. But then $2(\Delta + 4u\kappa) =
0$, since $2(4\kappa + \tau) = 8\kappa$. This completes the proof.
\qed
\end{proof}

\begin{corollary}\label{cokernelkappa}The cokernel of the map induced
by the Hurewicz map
$$\ell \colon \operatorname{TR}_5^n(\mathbb{S};2) \to 
\operatorname{TR}_5^n(\mathbb{Z};2)$$
is equal to $\mathbb{Z}/2^v\mathbb{Z} \cdot \kappa$, where $v =
v(n-1)$ is the smaller of $3$ and $n-1$.
\end{corollary}

\begin{proof}It follows immediately from Thm.~\ref{TRintegers} that
the cokernel of the map $\ell$ is generated by the class of
$\kappa$. Moreover, since the class
$$\xi_{5,1} + \dots + \xi_{5,n-1} + 4u\kappa$$
has order $2$, it is also clear that the cokernel of the map $\ell$ is
annihilated by multiplication by $8$. Hence, it will suffice to 
show that, for $n = 4$, the cokernel of the map $\ell$ is not
annihilated by $4$, or equivalently, that the map $\ell$ takes
the class $\rho \in \operatorname{TR}_5^4(\mathbb{S};2)$ to zero. But
this follows immediately from the structure of the spectral sequences
that abuts $\mathbb{H}_5(C_8,T(\mathbb{S}))$ and
$\mathbb{H}_5(C_8,T(\mathbb{Z}))$. 
\qed
\end{proof}

\section{The groups $\operatorname{TR}_q^n(\mathbb{S},I;2)$}\label{relativesection}

We again implicitly consider homotopy groups with
$\mathbb{Z}_2$-coefficients. The Hurewicz map from the sphere spectrum
$\mathbb{S}$ to the Eilenberg MacLane spectrum $\mathbb{Z}$ for the
ring of integers induces a map of topological Hochschild
$\mathbb{T}$-spectra
$$\ell \colon T(\mathbb{S}) \to T(\mathbb{Z}).$$
In~\cite[Appendix]{bokstedtmadsen}, B\"{o}kstedt and Madsen constructs
a sequence of cyclotomic spectra
$$T(\mathbb{S},I) \xrightarrow{i}
T(\mathbb{S}) \xrightarrow{\ell}
T(\mathbb{Z}) \xrightarrow{\partial}
\Sigma T(\mathbb{S},I)$$
such that the underlying sequence of $\mathbb{T}$-spectra is a
cofibration sequence. As a consequence, the equivariant homotopy
groups
$$\operatorname{TR}_q^n(\mathbb{S},I;p) = [ S^q \wedge
(\mathbb{T}/C_{p^{n-1}})_+, T(\mathbb{S},I) ]_{\mathbb{T}}$$
come equipped with maps
$$\begin{aligned}
R \colon & \operatorname{TR}_q^n(\mathbb{S},I;p) \to
\operatorname{TR}_q^{n-1}(\mathbb{S},I;p)
\hskip7mm \text{(restriction)} \hfill\space \cr
F \colon & \operatorname{TR}_q^n(\mathbb{S},I;p) \to
\operatorname{TR}_q^{n-1}(\mathbb{S},I;p)
\hskip7mm \text{(Frobenius)} \hfill\space \cr
V \colon & \operatorname{TR}_q^{n-1}(\mathbb{S},I;p) \to
\operatorname{TR}_q^n(\mathbb{S},I;p) 
\hskip7mm \text{(Verschiebung)} \hfill\space \cr
d \colon & \operatorname{TR}_q^n(\mathbb{S},I;p) \to
\operatorname{TR}_{q+1}^n(\mathbb{S},I;p) 
\hskip7mm \text{(Connes' operator)} \hfill\space \cr
\end{aligned}$$
and all maps in the long-exact sequence of equivariant homotopy groups
induced by the cofibration sequence above,
$$\cdots \to \operatorname{TR}_q^n(\mathbb{S},I;2) \xrightarrow{i}
\operatorname{TR}_q^n(\mathbb{S};2) \xrightarrow{\ell}
\operatorname{TR}_q^n(\mathbb{Z};2) \xrightarrow{\partial}
\operatorname{TR}_{q-1}^n(\mathbb{S},I;2) \to \cdots,$$
are compatible with restriction maps, Frobenius maps, Verschiebung
maps, and Connes' operator. Moreover, this is a sequence of graded
modules over the graded ring $\operatorname{TR}_*^n(\mathbb{S};p)$.

\begin{lemma}\label{degreethreeexact}The following sequence is
exact, for all $n \geqslant 1$.
$$0 \to \mathbb{H}_3(C_{2^{n-1}}, T(\mathbb{S},I)) \xrightarrow{\,N\,}
\operatorname{TR}_3^n(\mathbb{S},I;2) \xrightarrow{\,R\,}
\operatorname{TR}_3^{n-1}(\mathbb{S},I;2) \to 0$$
\end{lemma}

\begin{proof}From the proof of Cor.~\ref{cokernellambda} we have a map
of short-exact sequences
$$\xymatrix{
{ 0 } \ar[r] &
{ \mathbb{H}_3(C_{2^{n-1}}, T(\mathbb{S})) } \ar[r] \ar[d] &
{ \operatorname{TR}_3^n(\mathbb{S};2) } \ar[r] \ar[d] &
{ \operatorname{TR}_3^{n-1}(\mathbb{S};2) } \ar[r] \ar[d] &
{ 0 } \cr
{ 0 } \ar[r] &
{ \mathbb{H}_3(C_{2^{n-1}}, T(\mathbb{Z})) } \ar[r] &
{ \operatorname{TR}_3^n(\mathbb{Z};2) } \ar[r] &
{ \operatorname{TR}_3^{n-1}(\mathbb{Z};2) } \ar[r] &
{ 0 } \cr
}$$
and that the left-hand vertical map is surjective. Moreover,
Lemma~\ref{evenzero} and Prop.~\ref{orbitsintegers} identify the
sequence of the statement with the sequence of kernels of the vertical
maps in this diagram. This completes the proof.
\qed
\end{proof}

\begin{corollary}\label{restrictionsurjective}The restriction map
$$R \colon \operatorname{TR}_q^n(\mathbb{S},I;2) \to
\operatorname{TR}_q^{n-1}(\mathbb{S},I;2)$$
is surjective, for all $q \leqslant 4$ and all $n \geqslant 1$.
\qed
\end{corollary}

We recall that for $n = 1$, the map $\ell$ is an isomorphism, if $q =
0$, and the zero map, if $q > 0$. It follows that the groups
$\operatorname{TR}_0^1(\mathbb{S},I;2)$,
$\operatorname{TR}_4^1(\mathbb{S},I;2)$, and
$\operatorname{TR}_5^1(\mathbb{S},I;2)$ are zero, that
$\operatorname{TR}_1^1(\mathbb{S},I;2)$ is isomorphic to
$\mathbb{Z}/2\mathbb{Z}$ generated by the unique class $\tilde{\eta}$
with $i(\tilde{\eta}) = \eta$, that $\operatorname{TR}_2^1(\mathbb{S},I;2)$
is isomorphic to $\mathbb{Z}/2\mathbb{Z} \oplus
\mathbb{Z}/2\mathbb{Z}$ generated by $\eta\tilde{\eta}$ and by the
class $\bar{\lambda} = \partial(\lambda)$, and that
$\operatorname{TR}_3^1(\mathbb{S},I;2)$ is isomorphic to
$\mathbb{Z}/8\mathbb{Z}$ generated by the unique class $\tilde{\nu}$
with $i(\tilde{\nu}) = \nu$. We note that $\eta \bar{\lambda} = 0$,
since $\operatorname{TR}_4^1(\mathbb{Z};2)$ is zero, while
$\eta^2\tilde{\eta} = 4\tilde{\nu}$. We consider the skeleton
spectral sequences
$$E_{s,t}^2 = H_s(C_{2^{n-1}}, \operatorname{TR}_t^1(\mathbb{S},I;2))
\Rightarrow \mathbb{H}_{s+t}(C_{2^{n-1}}, T(\mathbb{S},I)).$$
In the case $n = 2$, the $E^2$-term for $s + t \leqslant 5$ takes the
form
\begin{center}
\begin{tabular*}{0.85\textwidth}{@{\extracolsep{\fill}}cccccc}
0 &  &  &  &  &  \cr
0 & 0 &  &  &  &  \cr
$\mathbb{Z}/8\mathbb{Z}$ & $\mathbb{Z}/2\mathbb{Z}$ &
$4\mathbb{Z}/8\mathbb{Z}$ &  &  & \cr
$(\mathbb{Z}/2\mathbb{Z})^2$ &
$(\mathbb{Z}/2\mathbb{Z})^2$ &
$(\mathbb{Z}/2\mathbb{Z})^2$ &
$(\mathbb{Z}/2\mathbb{Z})^2$ &  &  \cr
$\mathbb{Z}/2\mathbb{Z}$ & $\mathbb{Z}/2\mathbb{Z}$ &
$\mathbb{Z}/2\mathbb{Z}$ & $\mathbb{Z}/2\mathbb{Z}$ &
$\mathbb{Z}/2\mathbb{Z}$ & \cr
0 & 0 & 0 & 0 & 0 & \hskip5mm 0 \cr
\end{tabular*}
\end{center}
The group $\smash{E_{s,1}^2}$ is generated by the class
$\tilde{\eta} z_s$, the group $\smash{E_{s,2}^2}$ by the classes
$\eta\tilde{\eta} z_s$ and $\bar{\lambda}z_s$, the group
$\smash{E_{s,3}^2}$ with $s = 0$ or $s$ an odd positive integer by the
class $\nu z_s$, and the group $\smash{E_{s,3}^2}$ with $s$ an even
positive integer by  the class $4\nu z_s$. We claim that
$d^2(\tilde{\eta}z_2) = \bar{\lambda}z_0$, or equivalently, that
Connes' operator maps
$$d\tilde{\eta} = \bar{\lambda}.$$
We show that the class $V(\bar{\lambda}) \in \mathbb{H}_2(C_2,
T(\mathbb{S},I))$ represented by $\bar{\lambda}z_0$ is zero. By
Lemma~\ref{degreethreeexact}, we may instead show that the
image $V(\bar{\lambda}) \in \operatorname{TR}_2^2(\mathbb{S},I;2)$ by
the norm map is zero. Now, $V(\bar{\lambda}) = V(\partial(\lambda))
= \partial(V(\lambda))$, and and by Prop.~\ref{TRintegers}, the class
$V(\lambda) \in \operatorname{TR}_3^2(\mathbb{Z};2)$ is either zero or
equal to $4\lambda$. But $\partial(4\lambda) = 4 \partial(\lambda) =
4\bar{\lambda}$ which is zero, by Cor.~\ref{cokernellambda}. This
proves the claim. The $d^2$-differential is now given by
Lemma~\ref{d^2-differential}. We find that the $\smash{ E^3 }$-term
for $s + t \leqslant 5$ takes the form
\begin{center}
\begin{tabular*}{0.85\textwidth}{@{\extracolsep{\fill}}cccccc}
0 &  &  &  &  &  \cr
0 & 0 &  &  &  &  \cr
$\mathbb{Z}/8\mathbb{Z}$ & $\mathbb{Z}/2\mathbb{Z}$ &
0 &  &  & \cr
$\mathbb{Z}/2\mathbb{Z}$ &
$\mathbb{Z}/2\mathbb{Z}$ &
$\mathbb{Z}/2\mathbb{Z}$ &
$\mathbb{Z}/2\mathbb{Z}$ &  &  \cr
$\mathbb{Z}/2\mathbb{Z}$ & $\mathbb{Z}/2\mathbb{Z}$ &
0 & 0 &
0 & \cr
0 & 0 & 0 & 0 & 0 & \hskip4mm 0 \cr
\end{tabular*}
\end{center}
and, for degree reasons, this is also the $E^{\infty}$-term. The group
$\smash{ E_{s,2}^3 }$ is generated by the class of
$\eta\tilde{\eta}z_s$. The class of $\bar{\lambda}z_s$ in
$\smash{ E_{s,2}^3 }$ is equal to zero, if $s$ is congruent to $0$ or
$1$ modulo $4$, and is equal to the class of $\eta \tilde{\eta}
z_s$, if $s$ is congruent to $2$ or $3$ modulo $4$.

The spectral sequences for $n \geqslant 3$ are similar with the only
difference being the groups $\smash{ E_{s,3}^r }$ with $s > 0$. In the
case $n = 3$, the $\smash{ E^{\infty} }$-term for $s + t \leqslant 5$
takes the form
\begin{center}
\begin{tabular*}{0.85\textwidth}{@{\extracolsep{\fill}}cccccc}
0 &  &  &  &  &  \cr
0 & 0 &  &  &  &  \cr
$\mathbb{Z}/8\mathbb{Z}$ & $\mathbb{Z}/4\mathbb{Z}$ &
$2\mathbb{Z}/4\mathbb{Z}$ &  &  & \cr
$\mathbb{Z}/2\mathbb{Z}$ &
$\mathbb{Z}/2\mathbb{Z}$ &
$\mathbb{Z}/2\mathbb{Z}$ &
$\mathbb{Z}/2\mathbb{Z}$ &  &  \cr
$\mathbb{Z}/2\mathbb{Z}$ & $\mathbb{Z}/2\mathbb{Z}$ &
0 & 0 &
0 & \cr
0 & 0 & 0 & 0 & 0 & \hskip4mm 0 \cr
\end{tabular*}
\end{center}
and in the case $n \geqslant 4$, it takes the form
\begin{center}
\begin{tabular*}{0.85\textwidth}{@{\extracolsep{\fill}}cccccc}
0 &  &  &  &  &  \cr
0 & 0 &  &  &  &  \cr
$\mathbb{Z}/8\mathbb{Z}$ & $\mathbb{Z}/8\mathbb{Z}$ &
$\mathbb{Z}/4\mathbb{Z}$ &  &  & \cr
$\mathbb{Z}/2\mathbb{Z}$ &
$\mathbb{Z}/2\mathbb{Z}$ &
$\mathbb{Z}/2\mathbb{Z}$ &
$\mathbb{Z}/2\mathbb{Z}$ &  &  \cr
$\mathbb{Z}/2\mathbb{Z}$ & $\mathbb{Z}/2\mathbb{Z}$ &
0 & 0 &
0 & \cr
0 & 0 & 0 & 0 & 0 & \hskip4mm 0 \cr
\end{tabular*}
\end{center}
We define $\epsilon \in \mathbb{H}_5(C_4,T(\mathbb{S},I))$ and
$\tilde{\rho} \in \mathbb{H}_5(C_8,T(\mathbb{S},I))$ to be the unique
homotopy classes that represent $2\tilde{\nu}z_2$ and
$\tilde{\nu}z_2$, respectively. We note that $V(\epsilon) =
2\tilde{\rho}$. We further define $\bar{\kappa} = \partial(\kappa) \in
\mathbb{H}_4(C_2, T(\mathbb{S},I))$.

\begin{proposition}\label{orbitsrelative}The groups
$\mathbb{H}_q(C_{2^{n-1}}, T(\mathbb{S},I))$ with $q \leqslant 5$
are given by
$$\begin{aligned}
\mathbb{H}_0(C_{2^{n-1}}, T(\mathbb{S},I)) & = 0 \cr
\mathbb{H}_1(C_{2^{n-1}}, T(\mathbb{S},I)) & =
\mathbb{Z}/2\mathbb{Z} \cdot V^{n-1}(\tilde{\eta}) \cr
\mathbb{H}_2(C_{2^{n-1}}, T(\mathbb{S},I)) & =
\mathbb{Z}/2\mathbb{Z} \cdot dV^{n-1}(\tilde{\eta}) \oplus
\mathbb{Z}/2\mathbb{Z} \cdot V^{n-1}(\eta\tilde{\eta}) \cr
\mathbb{H}_3(C_{2^{n-1}}, T(\mathbb{S},I)) & =
\mathbb{Z}/2\mathbb{Z} \cdot dV^{n-1}(\eta\tilde{\eta}) \oplus
\mathbb{Z}/8\mathbb{Z} \cdot V^{n-1}(\tilde{\nu}) \cr
\mathbb{H}_4(C_{2^{n-1}}, T(\mathbb{S},I)) & = \begin{cases}
0 & (n = 1) \cr
\mathbb{Z}/2^v\mathbb{Z} \cdot dV^{n-1}(\tilde{\nu}) \oplus
\mathbb{Z}/2\mathbb{Z} \cdot V^{n-2}(\bar{\kappa}) & (n \geqslant 2)
\cr
\end{cases} \cr
\mathbb{H}_5(C_{2^{n-1}}, T(\mathbb{S},I)) & = \begin{cases}
0 & \hskip2.1mm (n = 1) \cr
\mathbb{Z}/2\mathbb{Z} \cdot d\bar{\kappa} & \hskip2.1mm (n = 2) \cr
\mathbb{Z}/2\mathbb{Z} \cdot dV(\bar{\kappa}) \oplus
\mathbb{Z}/2\mathbb{Z} \cdot \epsilon & \hskip2.1mm (n = 3) \cr
\mathbb{Z}/2\mathbb{Z} \cdot dV^{n-2}(\bar{\kappa}) \oplus
\mathbb{Z}/4\mathbb{Z} \cdot V^{n-4}(\tilde{\rho}) & \hskip2.1mm
(n \geqslant 4) \cr 
\end{cases} \cr
\end{aligned}$$
where $v = v(n-1)$ is the smaller of $3$ and $n-1$.
\end{proposition}

\begin{proof}The statement for $q \leqslant 3$ follows immediately
from the spectral sequence above since the generators given in the
statement have the indicated orders. To prove the statement for $q =
4$, we first note that $dV^{n-1}(\tilde{\nu})$ has order
$v(n-1)$. Indeed, the class $\tilde{\nu}$ has order $8$ and
$d\tilde{\nu} = 0$. Moreover, the image of the map
$$\ell \colon \mathbb{H}_5(C_{2^{n-1}},T(\mathbb{S})) \to
\mathbb{H}_5(C_{2^{n-1}},T(\mathbb{Z}))$$
does not contain the class $V^{n-2}(\kappa)$. Indeed, this follows
immediately from the induced map of spectral sequences. It follows
that $V^{n-2}(\bar{\kappa})$ is a non-zero class of order $2$ which
is represented by the element $\eta\tilde{\eta}z_2 =
\bar{\lambda}z_2$ in the $\smash{ E^{\infty} }$-term of the spectral
sequence above. This proves the statement for $q = 4$. It remains to
prove the statement for $q = 5$. It follows
from~\cite[Prop.~4.4.1]{hm4} that the element $\eta\tilde{\eta}z_3 =
\bar{\lambda}z_3$ in the $\smash{ E^{\infty} }$-term of the spectral
sequence above represents the class $dV^{n-2}(\bar{\kappa})$. Hence,
this class is non-zero and has order $2$. Moreover, the spectral
sequence shows that the subgroup of $\mathbb{H}_5(C_{2^{n-1}},
T(\mathbb{S},I))$ generated by $dV^{n-2}(\bar{\kappa})$ is a direct
summand. This completes the proof.
\qed
\end{proof}

The following result was proved by Costeanu
in~\cite[Prop.~2.6]{costeanu}.

\begin{lemma}\label{imageofeta}The map
$$\ell \colon \operatorname{TR}_1^n(\mathbb{S};2) \to
\operatorname{TR}_1^n(\mathbb{Z};2)$$
takes the class $\eta = \eta \cdot [1]_n$ to the class $\xi_{1,1}$.
\end{lemma}

\begin{proof}We temporarily write $[1_{\mathbb{S}}]_n$ and
$[1_{\mathbb{Z}}]_n$ for the multiplicative unit elements of the
graded rings $\operatorname{TR}_*^n(\mathbb{S};2)$ and
$\operatorname{TR}_*^n(\mathbb{Z};2)$,
respectively. By~\cite[Prop.~2.7.1]{hm}, the map $\ell$ is a map of
graded algebras over the graded ring given by the stable homotopy
groups of spheres. Hence, it takes the class $\eta \cdot
[1_{\mathbb{S}}]_n$ to the class $\eta \cdot 
[1_{\mathbb{Z}}]_n$. Similarly, it is proved in~\cite[Cor.~6.4.1]{gh}
that the cyclotomic trace map
$$\operatorname{tr} \colon K_*(\mathbb{Z}) \to 
\operatorname{TR}_*^n(\mathbb{Z};2)$$
is a map of graded algebras over the graded ring given by the stable
homotopy groups of spheres. Hence, the class $\eta \cdot
[1_{\mathbb{Z}}]_n$ is equal to the image by the cyclotomic trace map
of the class $\eta \cdot 1_{\mathbb{Z}} \in K_1(\mathbb{Z})$. The
latter class is known to be equal to the generator $\{-1\} \in
K_1(\mathbb{Z})$. It is proved in~\cite[Lemma~2.3.3]{hm4} that the
image by the cyclotomic trace map of the generator $\{-1\}$ is equal to
the class
$$d \log [-1]_n \in \operatorname{TR}_1^n(\mathbb{Z};2).$$
To evaluate this class, we recall from~\cite[Thm.~F]{hm} that the ring
$\operatorname{TR}_0^n(\mathbb{Z};2)$ is canonically isomorphic to the
ring of Witt vectors $W_n(\mathbb{Z})$. One readily verifies that
$$[-1]_n = -[1]_n + V([1]_{n-1})$$
by evaluating the ghost coordinates. It follows that
$d[-1]_n = dV([1]_{n-1})$, and since the class $[-1]_n$ is a square
root of $1$, we find
$$\begin{aligned}
d\log [-1]_n & = [-1]_n d[-1]_n =
(-[1]_n + V([1]_{n-1})) \cdot dV([1]_{n-1}) \cr
{} & = dV([1]_{n-1}) + V(FdV([1]_{n-1})) =
dV([1]_{n-1}) + V(\eta \cdot [1]_{n-1}). \cr
\end{aligned}$$
But by Thm.~\ref{TRsphere}, this is the class $\xi_{1,1}$ as
stated.
\qed
\end{proof}

\begin{remark}It follows from Lemma~\ref{imageofeta} that
$\ell(V^s(\eta)) = \sum_{t \geqslant s} 2^{t-1}\xi_{1,t}$, if $s
\geqslant 2$.
\end{remark}

At present, we do not know the precise value of the map
$$\ell \colon \operatorname{TR}_q^n(\mathbb{S};2) \to
\operatorname{TR}_q^n(\mathbb{Z};2)$$
for $q \geqslant 3$. However, we have the following result. We define
$\tilde{\eta} \in \operatorname{TR}_1^n(\mathbb{S},I;2)$ to be the
unique class such that $i(\tilde{\eta}) = \eta - \xi_{1,1}$. The class
$\tilde{\nu}$ that appears in the statement will be defined in the
course of the proof. It would be desirable to better understand this
class. In particular, we do not know the values of 
$\eta^2\tilde{\eta}$ or $Fd\tilde{\nu}$.

\begin{theorem}\label{TRrelative}The groups
$\operatorname{TR}_q^n(\mathbb{S},I;2)$ with $q \leqslant 3$ are given
by
$$\begin{aligned}
\operatorname{TR}_0^n(\mathbb{S},I;2) & = 0 \cr
\operatorname{TR}_1^n(\mathbb{S},I;2) & =
\bigoplus_{0 \leqslant s < n}
\mathbb{Z}/2\mathbb{Z} \cdot V^s(\tilde{\eta}) \cr
\operatorname{TR}_2^n(\mathbb{S},I;2) & = 
\bigoplus_{0 \leqslant s < n} \big( 
\mathbb{Z}/2\mathbb{Z} \cdot V^s(\eta\tilde{\eta}) \oplus
\mathbb{Z}/2\mathbb{Z} \cdot dV^s(\tilde{\eta}) \big) \cr
\operatorname{TR}_3^n(\mathbb{S},I;2) & = 
\bigoplus_{0 \leqslant s < n}
\mathbb{Z}/8\mathbb{Z} \cdot V^s(\tilde{\nu}) \oplus
\bigoplus_{1 \leqslant s < n}
\mathbb{Z}/2\mathbb{Z} \cdot dV^s(\eta\tilde{\eta}) \cr
\end{aligned}$$
and the group $\operatorname{TR}_4^n(\mathbb{S},I;2)$ is generated by
$dV^s(\tilde{\nu})$ with $0 \leqslant s < n$. Moreover, the
restriction map takes $\tilde{\eta}$ to $\tilde{\eta}$ and
$\tilde{\nu}$ to $\tilde{\nu}$, and the Frobenius map takes both
$\tilde{\eta}$ and $\tilde{\nu}$ to zero and takes $d\tilde{\eta}$ to 
$d\tilde{\eta}$. The class $d(\eta\tilde{\eta}) = \eta
d(\tilde{\eta})$ is zero.
\end{theorem}

\begin{proof}The statement for $q = 0$ follows immediately from
Thm.~\ref{TRsphere} and Prop.~\ref{TRintegers}. In the case $q = 1$,
Lemma~\ref{evenzero} shows that the map
$i \colon \operatorname{TR}_1^n(\mathbb{S},I;2) \to
\operatorname{TR}_1^n(\mathbb{S};2)$ is injective, and
Lemma~\ref{imageofeta} shows that the class $\eta - \xi_{1,1}$ is in
the image. As said above, we define $\tilde{\eta} \in
\operatorname{TR}_1^n(\mathbb{S},I;2)$ to be the unique class with
$i(\tilde{\eta}) = \eta - \xi_{1,1}$. The statement for $q = 1$ now
follows immediately from Thm..~\ref{TRsphere} and
Prop.~\ref{TRintegers}. For $q = 2$, a similar argument shows that the
group $\operatorname{TR}_2^n(\mathbb{S},I;2)$ contains the subgroup  
$$\operatorname{TR}_2^n(\mathbb{S},I;2)' = 
\bigoplus_{0 \leqslant s < n}
\mathbb{Z}/2\mathbb{Z} \cdot V^s(\eta\tilde{\eta}) \oplus
\bigoplus_{1 \leqslant s < n}
\mathbb{Z}/2\mathbb{Z} \cdot dV^s(\tilde{\eta})$$
which maps isomorphically onto the image of $i \colon
\operatorname{TR}_2^n(\mathbb{S},I;2) \to
\operatorname{TR}_2^n(\mathbb{S};2)$, and Lemma~\ref{cokernellambda}
shows that the kernel of the latter map is $\mathbb{Z}/2\mathbb{Z}
\cdot \bar{\lambda}$. Therefore, to prove the statement for $q = 2$,
it remains to prove that $d\tilde{\eta} = \bar{\lambda}$. We have
already proved this equality, for $n = 1$, in the discussion
preceeding Prop.~\ref{orbitsrelative}. It follows that the iterated
restriction map $R^{n-1} \colon \operatorname{TR}_2^n(\mathbb{S},I;2)
\to \operatorname{TR}_2^1(\mathbb{S},I;2)$ takes the class
$d\tilde{\eta}$ to the class $\bar{\lambda}$. Since the kernel of this
map is equal to the subgroup $\operatorname{TR}_2^n(\mathbb{S},I;2)'$,
it suffices to show that the class $i(d\tilde{\eta} - \bar{\lambda})
\in \operatorname{TR}_2^n(\mathbb{S};2)$ is zero. We have
$$i(d\tilde{\eta} - \bar{\lambda}) = i(d\tilde{\eta}) =
d(i(\tilde{\eta})) = d\eta - d\xi_{1,1}.$$
The class $d\eta$ is zero, since $\eta$ is in the image of the
cyclotomic trace map, and we proved in Thm.~\ref{TRsphere} that
$d\xi_{1,1}$ is zero. Th statement for $q = 2$ follows. It also
follows that $F(d\tilde{\eta}) = d\tilde{\eta}$, since $d\tilde{\eta}
= \partial(\lambda)$ and $\lambda$ is in the image of the cyclotomic
trace map.

We next prove the statement for $q = 3$. By
Lemma~\ref{degreethreeexact}, the sequences
$$0 \to \mathbb{H}_3(C_{2^{n-1}}, T(\mathbb{S},I)) \xrightarrow{N}
\operatorname{TR}_3^n(\mathbb{S},I;2) \xrightarrow{R}
\operatorname{TR}_3^{n-1}(\mathbb{S},I;2) \to 0$$
are exact. The left-hand group was evaluated in
Prop.~\ref{orbitsrelative} above. To complete the proof, we
inductively construct classes
$$\tilde{\nu} = \tilde{\nu}_n \in
\operatorname{TR}_3^n(\mathbb{S},I;2) \hskip5mm (n \geqslant 1)$$
such that $R(\tilde{\nu}_n) = \tilde{\nu}_{n-1}$ and $F(\tilde{\nu}_n)
= 0$, and such that $\tilde{\nu}_1$ is the class $\tilde{\nu}$ already
defined. By Prop.~\ref{evenzero} and Cor.~\ref{cokernellambda}, we
have a short-exact sequence
$$0 \to \operatorname{TR}_3^n(\mathbb{S},I;2) \xrightarrow{i}
\operatorname{TR}_3^n(\mathbb{S};2) \xrightarrow{\ell}
\operatorname{TR}_3^n(\mathbb{Z};2)' \to 0,$$
where the right-hand group is the index two subgroup of
$\operatorname{TR}_3^n(\mathbb{Z};2)$ defined by
$$\operatorname{TR}_3^n(\mathbb{Z};2)' = \bigoplus_{1 \leqslant s < n} 
\mathbb{Z}/2^{s+1}\mathbb{Z} \cdot \xi_{3,s}.$$
To define the class $\tilde{\nu}_2$, we first note that $\ell(\nu) =
a_1 \xi_{3,1}$, where $a_1 \in (\mathbb{Z}/4\mathbb{Z})^*$ is a unit,
and choose a unit $\tilde{a}_1 \in (\mathbb{Z}/8\mathbb{Z})^*$ whose
reduction modulo $4$ is $a_1$. Then, we have $\ell(\nu -
\tilde{a}_1\xi_{3,1}) = 0$ and $F(\nu - \tilde{a}_1\xi_{3,1}) =
(1-\tilde{a}_1)\nu$. We choose $b_1 \in \mathbb{Z}/8\mathbb{Z}$ such
that $2b_1 = \tilde{a}_1 - 1$ and define $\tilde{\nu}_2$ to be the
unique class such that
$$i(\tilde{\nu}_2) = \nu - \tilde{a}_1\xi_{3,1} + b_1 V(\nu).$$
Then $R(\tilde{\nu}_2) = \tilde{\nu}_1$ and $F(\tilde{\nu}_2) = 0$ as
desired.

We next define the class $\tilde{\nu}_3$. The image of $\tilde{\nu}_2$
by the composition
$$\operatorname{TR}_3^2(\mathbb{S},I;2) \xrightarrow{i}
\operatorname{TR}_3^2(\mathbb{S};2) \xrightarrow{S}
\operatorname{TR}_3^3(\mathbb{S};2) \xrightarrow{\ell}
\operatorname{TR}_3^3(\mathbb{Z};2)'$$
is equal to $a_2\xi_{3,2}$, for some $a_2 \in
\mathbb{Z}/8\mathbb{Z}$. We claim that, in fact, $a_2 \in
4\mathbb{Z}/8\mathbb{Z}$. Indeed, since $F(\tilde{\nu}_2) = 0$, we
have $F(a_2\xi_{3,2}) = 0$. But $F(\xi_{3,2}) = \xi_{3,1}$ which shows
that the modulo $4$ reduction of $a_2$ is zero as claimed. We let $b_2
\in \mathbb{Z}/2\mathbb{Z}$ be the unique element such that $4b_2 =
a_2$ and define $\tilde{\nu}_3$ to be the unique class such that
$$i(\tilde{\nu}_3) = \begin{cases}
S(i(\tilde{\nu}_2)) + b_2(4\xi_{3,2} +
dV^2(\eta^2) + 2V^2(\nu)) & \text{if $4\xi_{3,1} = dV(\eta^2)$} \cr
S(i(\tilde{\nu}_2)) + b_2(4\xi_{3,2} +
dV^2(\eta^2)) & \text{if $4\xi_{3,1} = \eta^2\xi_{1,1}$.} \cr
\end{cases}$$
The sum on the right-hand side is in the kernel of
$\ell$, since both $\ell(\eta^2) \in
\operatorname{TR}_2^1(\mathbb{Z};2)$ and $\ell(\nu) \in
\operatorname{TR}_3^1(\mathbb{Z};2)$ are zero. We also have
$R(\tilde{\nu}_3) = \tilde{\nu}_2$ and $F(\tilde{\nu}_3) = 0$ as
desired. Indeed, if $4\xi_{3,1} = dV(\eta^2)$, then
$$\begin{aligned}
i(F(\tilde{\nu}_3)) & = F(S(i(\tilde{\nu}_2)) + b_2(4\xi_{3,2} +
dV^2(\eta^2) + 2V^2(\nu))) \cr
{} & = S(i(F(\tilde{\nu}_2))) + b_2(4\xi_{3,1} +
dV(\eta^2) + V(\eta^3) + 4V(\nu)), \cr
\end{aligned}$$
and if $4\xi_{3,1} = \eta^2\xi_{3,1}$, then
$$\begin{aligned}
i(F(\tilde{\nu}_3)) & = F(S(i(\tilde{\nu}_2)) + b_2(4\xi_{3,2} +
dV^2(\eta^2))) \cr
{} & = S(i(F(\tilde{\nu}_2))) + b_2(4\xi_{3,1} +
dV(\eta^2) + V(\eta^3)), \cr
\end{aligned}$$
and in either case, the sum is zero.

Finally, we let $n \geqslant 4$ and assume that the class
$\tilde{\nu}_{n-1}$ has been defined. We find as before that the image
of the class $\tilde{\nu}_{n-1}$ by the composition
$$\operatorname{TR}_3^{n-1}(\mathbb{S},I;2) \xrightarrow{i}
\operatorname{TR}_3^{n-1}(\mathbb{S};2) \xrightarrow{S}
\operatorname{TR}_3^n(\mathbb{S};2) \xrightarrow{\ell}
\operatorname{TR}_3^n(\mathbb{Z};2)'$$
is equal to $a_{n-1}\xi_{3,n-1}$ with $a_{n-1} \in
2^{n-1}\mathbb{Z}/2^n\mathbb{Z}$ and define $\tilde{\nu}_n$ to be the
unique class whose image by the map $i$ is equal to
$$i(\tilde{\nu}_n) = S(i(\tilde{\nu}_{n-1})) - a_{n-1}\xi_{3,n-1}.$$
Then $R(\tilde{\nu}_n) = \tilde{\nu}_{n-1}$ and $F(\tilde{\nu}_n) =
0$, since $2^{n-1}\xi_{3,n-2} = 0$, for $n \geqslant 4$.

It remains to prove that the group
$\operatorname{TR}_4^n(\mathbb{S},I;2)$ is generated by the homotopy
classes $dV^s(\tilde{\nu})$ with $0 \leqslant s < n$. The sequence
$$\mathbb{H}_4(C_{2^{n-1}},T(\mathbb{S},I)) \to
\operatorname{TR}_4^n(\mathbb{S},I;2) \to
\operatorname{TR}_4^{n-1}(\mathbb{S},I;2) \to 0,$$
which is exact by Cor.~\ref{restrictionsurjective}, together with
Prop.~\ref{orbitsrelative} show that
$\operatorname{TR}_4^n(\mathbb{S},I;2)$ is generated by the classes
$dV(\tilde{\nu})$, $1 \leqslant s < n$, and $\bar{\kappa}$. Indeed,
since the boundary map
$$\partial \colon \operatorname{TR}_5^n(\mathbb{Z};2) \to
\operatorname{TR}_4^n(\mathbb{S},I;2)$$
commutes with the Verschiebung, it follows that $V^{n-1}(\bar{\kappa})
= c\bar{\kappa}$, for some $c \in 2^{n-1}\mathbb{Z}$. Hence, it
suffices to show that there exists a class $x_n \in
\operatorname{TR}_3^n(\mathbb{S},I;2)$ with $dx_n = \bar{\kappa}$. The
statement for $n = 1$ is trivial, since the group
$\operatorname{TR}_4^1(\mathbb{S},I;2)$ is zero. We postpone the proof
of the statement for $n = 2$ to Lemma~\ref{connesoperatorsurjective}
below and here prove the induction step. So we let $n \geqslant 3$ and
assume that there exists a class $x_{n-1} \in
\operatorname{TR}_3^{n-1}(\mathbb{S},I;2)$ with $dx_{n-1} =
\bar{\kappa}$. We use Cor.~\ref{restrictionsurjective} to choose a
class $x_n' \in \operatorname{TR}_3^n(\mathbb{S},I;2)$ with 
$R(x_n') = x_{n-1}$. Then the exact sequence above and
Prop.~\ref{orbitsrelative} show that
$$dx_n' = \bar{\kappa} + adV^{n-1}(\tilde{\nu}) + bV^{n-1}(\kappa)
= adV^{n-1}(\tilde{\nu}) + (1+bc)\bar{\kappa},$$
for some integers $a$ and $b$. Since $1+bc$ is a $2$-adic unit, the class
$$x_n = (1+bc)^{-1}(x_n' - aV^{n-1}(\tilde{\nu}))$$
is well-defined and satisfies $dx_n = \bar{\kappa}$ as desired.
\qed
\end{proof}

One wonders whether the class $\tilde{\nu}$, which was defined in the
proof above, satisfies that $d\tilde{\nu} = \bar{\kappa}$. This would
imply that $Fd\tilde{\nu} = d\tilde{\nu}$, since $\kappa$ is in the
image of the cyclotomic trace map.

The following result was used in the proof of Thm.~\ref{TRrelative}
above.

\begin{lemma}\label{connesoperatorsurjective}Connes' operator
$$d \colon \operatorname{TR}_3^2(\mathbb{S},I;2) \to
\operatorname{TR}_4^2(\mathbb{S},I;2)$$
is surjective.
\end{lemma}

\begin{proof}The groups $\operatorname{TR}_q^2(\mathbb{S},I;2)$ for $q
\leqslant 5$ are given by
$$\begin{aligned}
\operatorname{TR}_0^2(\mathbb{S},I;2) & = 0 \cr
\operatorname{TR}_1^2(\mathbb{S},I;2) & = \mathbb{Z}/2\mathbb{Z} \cdot
\tilde{\eta} \oplus \mathbb{Z}/2\mathbb{Z} \cdot V(\tilde{\eta}) \cr
\operatorname{TR}_2^2(\mathbb{S},I;2) & =  \mathbb{Z}/2\mathbb{Z} \cdot
d\tilde{\eta} \oplus \mathbb{Z}/2\mathbb{Z} \cdot dV(\tilde{\eta})
\oplus \mathbb{Z}/2\mathbb{Z} \cdot \eta\tilde{\eta} \oplus
\mathbb{Z}/2\mathbb{Z} \cdot V(\eta\tilde{\eta}) \cr
\operatorname{TR}_3^2(\mathbb{S},I;2) & = \mathbb{Z}/2\mathbb{Z} \cdot
dV(\eta\tilde{\eta}) \oplus \mathbb{Z}/8\mathbb{Z} \cdot \tilde{\nu}
\oplus \mathbb{Z}/8\mathbb{Z} \cdot V(\tilde{\nu}) \cr
\operatorname{TR}_4^2(\mathbb{S},I;2) & = \mathbb{Z}/2\mathbb{Z} \cdot
\bar{\kappa} \oplus \mathbb{Z}/2\mathbb{Z} \cdot dV(\tilde{\nu}) \cr
\operatorname{TR}_5^2(\mathbb{S},I;2) & = 0 \cr
\end{aligned}$$
Hence, the lemma is equivalent to the statement that in the spectral
sequence
$$E_{s,t}^2 = H_s(C_2,\operatorname{TR}_t^2(\mathbb{S},I;2)) \Rightarrow
\mathbb{H}_{s+t}(C_2,\mathit{TR}^{\,2}(\mathbb{S},I)),$$ 
the $d^2$-differential $d^2 \colon E_{3,3}^2 \to E_{1,4}^2$ is
surjective. We first argue that this is equivalent to the statement
that $\smash{ \mathbb{H}_5(C_2,\mathit{TR}^{\,2}(\mathbb{S},I)) }$
has order $4$. The elements $\bar{\kappa}z_0$ and $dV(\tilde{\nu})z_0$
in $E_{0,4}^2$ are infinite cycles and represent the homotopy classes
$V^2(\bar{\kappa})$ and $\smash{ 2dV^2(\tilde{\nu}) }$ of
$\smash{ \mathbb{H}_4(C_2,\mathit{TR}^{\,2}(\mathbb{S},I)) }$. We
claim that these classes are non-zero and generate a subgroup of order
$4$. To see this, we consider the norm maps from
Prop.~\ref{fundamentalcofibrationsequence},
$$\mathbb{H}_4(C_2,\mathit{TR}^{\,2}(\mathbb{S},I)) \xrightarrow{N_2}
\operatorname{TR}_4^3(\mathbb{S},I;2) \xleftarrow{N_1}
\mathbb{H}_4(C_4, T(\mathbb{S},I)).$$
It will suffice to show that the subgroup of the middle group
generated by the images of the classes $V^2(\bar{\kappa})$ and
$dV^2(\tilde{\nu})$ has order $4$. This subgroup is equal to the
subgroup generated by the images of the classes $V^2(\bar{\kappa})$
and $dV^2(\tilde{\nu})$ of the right-hand group. The right-hand map is
injective, since $\operatorname{TR}_5^2(\mathbb{S},I;2)$ is zero. (The
left-hand map is also injective, since
$\operatorname{TR}_5^1(\mathbb{S},I;2)$ is zero.) Hence, it suffices
to show that the subgroup of the right-hand group generated by the
classes $V^2(\bar{\kappa})$ and $dV^2(\tilde{\nu})$ has order
$4$. But this is proved in Prop.~\ref{orbitsrelative}. The claim
follows. We conclude that in the spectral sequence under
consideration, the differentials
$$d^r \colon E_{r,5-r}^r \to E^r_{0,4}$$
are zero, for all $r \geqslant 2$. It follows that the groups
$E_{0,5}^{\infty}$, $E_{2,3}^{\infty}$, $E_{3,2}^{\infty}$,
$E_{4,1}^{\infty}$, and $E_{5,0}^{\infty}$ have orders $0$, $2$, $2$,
$0$, and $0$, respectively, and that for all $r \geqslant 3$, the
differentials
$$d^r \colon E_{r+1,4-r}^r \to E_{1,4}^r$$
are zero. We conclude that the differential $d^2 \colon E_{3,3}^2 \to
E_{1,4}^2$ is surjective if and only if the group
$\smash{ \mathbb{H}_5(C_2,\mathit{TR}^{\,2}(\mathbb{S},I)) }$ has
order $4$.

The order of the group
$\mathbb{H}_5(C_2,\mathit{TR}^{\,2}(\mathbb{S},I)$ is divisible by
$4$ and to show that it is equal to $4$, we consider the following
diagram with exact rows and columns:
$$\xymatrix{
{ \operatorname{TR}_7^1(\mathbb{S};2) } \ar[r]^(.42){0} \ar[d] &
{ \mathbb{H}_6(C_2,\mathit{TR}^{\,2}(\mathbb{S};2)) } \ar[r]
\ar[d]^{\ell} &
{ \operatorname{TR}_6^3(\mathbb{S};2) } \ar[d] \cr
{ \operatorname{TR}_7^1(\mathbb{Z};2) } \ar[r]^(.42){\delta''} \ar[d] &
{ \mathbb{H}_6(C_2,\mathit{TR}^{\,2}(\mathbb{Z};2)) } \ar[r]
\ar[d]^{\partial} &
{ \operatorname{TR}_6^3(\mathbb{Z};2) } \ar[d] \cr
{ \operatorname{TR}_6^1(\mathbb{S},I;2) } \ar[r]^(.42){\delta'} \ar[d] &
{ \mathbb{H}_5(C_2,\mathit{TR}^{\,2}(\mathbb{S},I)) } \ar[r]
\ar[d] &
{ \operatorname{TR}_5^3(\mathbb{S},I;2) } \ar[d] \cr
{ \operatorname{TR}_6^1(\mathbb{S};2) } \ar[r]^(.42){0} &
{ \mathbb{H}_5(C_2,\mathit{TR}^{\,2}(\mathbb{S};2)) } \ar[r] &
{ \operatorname{TR}_5^3(\mathbb{S};2) } \cr
}$$
It follows from Thms.~\ref{TRsphere} and~\ref{TRintegers} that the
group $\operatorname{TR}_5^3(\mathbb{S},I;2)$ is equal to
$\mathbb{Z}/2\mathbb{Z} \cdot 2\xi_{5,2}$. Hence, it will suffice to
show that the image of the map $\delta'$ has order at most $2$. Since
the lower left-hand horizontal map in the diagram above is zero, we
conclude that the image of the map $\delta'$ is contained in the image
of the map $\partial$. Therefore, it suffices to show that the image
of the map $\partial$ has order at most $2$. 

The group $\operatorname{TR}_6^3(\mathbb{Z};2)$ is zero by
Prop.~\ref{evenzero} and the group
$\operatorname{TR}_7^1(\mathbb{Z};2)$ is cyclic of order $4$. It
follows that the group
$\smash{ \mathbb{H}_6(C_2,\mathit{TR}^{\,2}(\mathbb{Z};2)) }$ is
cyclic and has order either $0$, $2$, or $4$. If the order is either 
$0$ or $2$, we are done, so assume that the order is $4$. We must show
that $2$ times a generator is contained in the image of the map $\ell$
in the diagram above. To this end, we consider the diagram
$$\xymatrix{
{ \mathbb{H}_6(C_4,\mathit{TR}^{\,2}(\mathbb{S};2)) } \ar[r]^{F}
\ar[d]^{\ell} &
{ \mathbb{H}_6(C_2,\mathit{TR}^{\,2}(\mathbb{S};2) } \ar[d]^{\ell} \cr
{ \mathbb{H}_6(C_4,\mathit{TR}^{\,2}(\mathbb{Z};2)) } \ar[r]^{F} &
{ \mathbb{H}_6(C_2,\mathit{TR}^{\,2}(\mathbb{Z};2) } \cr
}$$
We first show that the lower horizontal map $F$ is surjective. The
assumption that the lower right-hand group has order $4$ implies that
a generator of this group is represented in the spectral sequence
$$E_{s,t}^2 = H_s(C_2,\operatorname{TR}_t^2(\mathbb{Z};2))
\Rightarrow \mathbb{H}_{s+t}(C_2,\mathit{TR}^{\,2}(\mathbb{Z};2))$$
by the element $\smash{ \lambda z_3 \in E_{3,3}^2 }$. This element is
the image by the map of spectral sequences induced by the map $F$ of
the element
$\lambda z_3 \in E_{3,3}^2$ in the spectral sequence
$$E_{s,t}^2 = H_s(C_4,\operatorname{TR}_t^2(\mathbb{Z};2))
\Rightarrow \mathbb{H}_{s+t}(C_4,\mathit{TR}^{\,2}(\mathbb{Z};2)).$$
We must show that the latter element $\lambda z_3$ is an infinite
cycle. For degree reasons, the only possible non-zero differential is
$\smash{ d^3 \colon E_{3,3}^3 \to E_{0,5}^3 }$. The target group is
equal to $\mathbb{Z}/2\mathbb{Z} \cdot \kappa z_0$, and the generator
$\kappa z_0$ represents the homotopy class $V^2(\kappa)$ in 
$\mathbb{H}_5(C_4,\mathit{TR}^{\,2}(\mathbb{Z};2))$. To see that
this class is non-zero, we consider the norm maps 
$$\mathbb{H}_4(C_4,\mathit{TR}^{\,2}(\mathbb{Z};2)) \xrightarrow{N_2}
\operatorname{TR}_5^4(\mathbb{Z};2) \xleftarrow{N_1}
\mathbb{H}_4(C_8,T(\mathbb{Z})).$$
We may instead prove that the image of the class $V^2(\kappa)$ by the
left-hand map is non-zero. This image class, in turn, is equal to the
image of the class $V^2(\kappa)$ by the right-hand map which is
injective since $\operatorname{TR}_6^3(\mathbb{Z};2)$ is zero. Now,
Prop.~\ref{orbitsintegers} shows that the class $V^2(\kappa)$ in the
right-hand group is non-zero. We conclude that the lower horizontal
map $F$ in the square diagram above is surjective as stated.

Finally, we show that the image of the left-hand vertical map $\ell$
in the square diagram above contains $2$ times the homotopy class
represented by the element $\lambda z_3$. In fact, the image of the
composition
$$\mathbb{H}_5(C_4, T(\mathbb{S})) \xrightarrow{S}
\mathbb{H}_5(C_4, \mathit{TR}^{\,2}(\mathbb{S};2))
\xrightarrow{\ell} \mathbb{H}_5(C_4,
\mathit{TR}^{\,2}(\mathbb{Z};2))$$
of the Segal-tom Dieck splitting and the map $\ell$ contains $2$
times the class represented by $\lambda z_3$. Indeed, by
Prop.~\ref{orbitssphere}, the element $\nu z_3 \in E_{3,3}^2$ of the
spectral sequence
$$E_{s,t}^2 = H_s(C_4,\operatorname{TR}_t^1(\mathbb{S};2))
\Rightarrow \mathbb{H}_{s+t}(C_4,T(\mathbb{S}))$$
is an infinite cycle whose image by the map of spectral sequence
induced by the composition of the maps $S$ and $\ell$ is equal
$2\lambda z_3 \in E_{3,3}^2 = \mathbb{Z}/4\mathbb{Z} \cdot \lambda
z_3$. This completes the proof.
\qed
\end{proof}

\section{The groups $\operatorname{Wh}_q^{\operatorname{Top}}(S^1)$
 for $q \leqslant 3$}\label{algebrasection}

In this section, we complete the proof of Thm.~\ref{main} of the
Introduction. It follows from~\cite[Thm.~1.2]{grunewaldkleinmacko}
that the odd-primary torsion subgroup of
$\operatorname{Wh}_q^{\operatorname{Top}}(S^1)$ is zero, for $q
\leqslant 3$. Hence, it suffices to consider the homotopy groups with
$\mathbb{Z}_2$-coefficients. We implicitly consider homotopy groups
with $\mathbb{Z}_2$-coefficients.

As we explained in the introduction, there is a long-exact sequence
$$\cdots \to \operatorname{Wh}_q^{\text{Top}}(S^1) \to
\widetilde{\operatorname{TR}}_q(\mathbb{S}[x^{\pm 1}], I[x^{\pm 1}]; 2)
\xrightarrow{1-F} 
\widetilde{\operatorname{TR}}_q(\mathbb{S}[x^{\pm 1}], I[x^{\pm 1}]; 2)
\to \cdots$$
where the middle and on the right-hand terms are the
cokernel of the assembly map
$$\alpha \colon \operatorname{TR}_q(\mathbb{S}; I; 2) \oplus 
\operatorname{TR}_{q-1}(\mathbb{S}; I; 2) \to
\operatorname{TR}_q(\mathbb{S}[x^{\pm 1}], I[x^{\pm 1}]; 2).$$
Moreover, since the groups $\operatorname{TR}_q^n(\mathbb{S},I;2)$ are
finite, for all integers $q$ and $n \geqslant 1$, the limit system
$\smash{ \{ \operatorname{TR}_q^n(\mathbb{S},I;2) \} }$ satisfies the
Mittag-Leffler condition, and Cor.~\ref{kernelR-F} then shows that
the same holds for the limit system $\smash{ \{
\widetilde{\operatorname{TR}}{}_q^n(\mathbb{S}[x^{\pm1}],I[x^{\pm1}];2) 
\} }$. It follows that, for all integers $q$, the canonical map
$$\widetilde{\operatorname{TR}}_q(\mathbb{S}[x^{\pm1}],I[x^{\pm1}];2)
\to \operatornamewithlimits{lim}_n
\widetilde{\operatorname{TR}}{}_q^n(\mathbb{S}[x^{\pm1}],I[x^{\pm1}];2)$$
is an isomorphism. Finally, Thm.~\ref{fundamentaltheorem} expresses the
right-hand side in terms of the groups
$\operatorname{TR}_q^m(\mathbb{S},I;2)$ which we evaluated in
Thm.~\ref{TRrelative} above, for $q \leqslant 3$.

\begin{theorem}\label{whiteheadzeroandone}The groups
$\operatorname{Wh}_0^{\operatorname{Top}}(S^1)$ and
$\operatorname{Wh}_1^{\operatorname{Top}}(S^1)$ are zero.
\end{theorem}

\begin{proof}We first note that, as an immediate consequence of
Thms.~\ref{fundamentaltheorem} and~\ref{TRrelative}, the group
$\smash{\widetilde{\operatorname{TR}}_0(\mathbb{S}[x^{\pm 1}], 
I[x^{\pm1}];2)}$ is zero. Moreover, we showed in Thm.~\ref{TRrelative} 
that the Frobenius map $F \colon \operatorname{TR}_1^m(\mathbb{S},I;2)
\to \operatorname{TR}_1^{m-1}(\mathbb{S},I;2)$ is zero, and hence,
$$1 - F \colon
\widetilde{\operatorname{TR}}_1(\mathbb{S}[x^{\pm 1}], I[x^{\pm 1}];2) \to
\widetilde{\operatorname{TR}}_1(\mathbb{S}[x^{\pm 1}], I[x^{\pm
  1}];2),$$
is the identity map. This shows that the group
$\operatorname{Wh}_0^{\text{Top}}(S^1)$ is zero as stated. To prove
that $\smash{\operatorname{Wh}_1^{\text{Top}}(S^1)}$ is zero, it
remains to prove that the map
$$1 - F \colon
\widetilde{\operatorname{TR}}_2(\mathbb{S}[x^{\pm 1}], I[x^{\pm 1}];2) \to
\widetilde{\operatorname{TR}}_2(\mathbb{S}[x^{\pm 1}], I[x^{\pm 1}];2),$$
is surjective. So let $\omega = (\omega^{(n)})$ be an element on
the right-hand side. We find an element $\omega' = (\omega'{}^{(n)})$
such that $(R-F)(\omega') = \omega$. By Thm.~\ref{fundamentaltheorem},
we can write $\omega^{(n)}$ uniquely as a sum
$$\sum_{j \in \mathbb{Z} \smallsetminus \{0\}} \!\!
\big( a_{0,j}^{(n)} [x]_n^j + b_{0,j}^{(n)} [x]_n^j d\log [x]_n \big) + 
\sum_{ \substack{1 \leqslant s < n \\ j \in \mathbb{Z} \smallsetminus
  2\mathbb{Z}}} \!\! \big( V^s(a_{s,j}^{(n)} [x]_{n-s}^j) +
dV^s(b_{s,j}^{(n)} [x]_{n-s}^j) \big)$$
with $a_{s,j}^{(n)} \in \operatorname{TR}_2^{n-s}(\mathbb{S},I;2)$ and 
$b_{s,j}^{(n)} \in \operatorname{TR}_1^{n-s}(\mathbb{S},I;2)$. We first
consider the four types of summands separately.

First, if $\omega^{(n)} = V^s(a^{(n)}[x]^j)$ with $s \geqslant 1$, we
let $\omega' = \omega$. Then 
$$(R-F)(\omega'{}^{(n+1)}) = (R-F)(V^s(a^{(n+1)}[x]^j)) = V^s(a^{(n)}[x]^j),$$
since $FV = 2$ and $2a^{(n)} = 0$. We note that here $j$ may be any
integer.

Second, if $\omega^{(n)} = dV^s(b^{(n)}[x]^j)$, where $j$ and $s 
\geqslant 1$ are integers, we define
$$\omega'{}^{(n)} = 
- \! \sum_{s \leqslant r < n-1} dV^{r+1}(b^{(n-1-r+s)}[x]^j) 
- \! \sum_{s \leqslant r < n} V^r(\eta b^{(n-r+s)}[x]^j).$$
Then we have $R(\omega'{}^{(n+1)}) = \omega'{}^{(n)}$ and 
$$\begin{aligned}
(R-F)(\omega'{}^{(n+1)}) 
{} & = - \! \sum_{s \leqslant r < n-1} dV^{r+1}(b^{(n-1-r+s)}[x]^j)
- \! \sum_{s \leqslant r < n} V^r(\eta b^{(n-r+s)}[x]^j) \cr
{} & \hskip4mm + \! \sum_{s \leqslant r < n} dV^r(b^{(n-r+s)}[x]^j)
+ \! \sum_{s \leqslant r < n} V^r(\eta b^{(n-r+s)}[x]^j) \cr
{} & = dV^s(b^{(n)}[x]^j) \cr
\end{aligned}$$
as desired. 

Third, if $\omega^{(n)} = b^{(n)}[x]^jd\log[x]$, we let $\omega' =
\omega$. Then $(R-F)(\omega'{}^{(n+1)}) = \omega^{(n)}$, since
$F(b^{(n)}) = 0$.

Fourth, we consider the case $\omega^{(n)} = a^{(n)}[x]^j$. Then
$a^{(n)} \in \operatorname{TR}_2^n(\mathbb{S},I;2)$ and we showed in 
Thm.~\ref{TRrelative} that this group is an $\mathbb{F}_2$-vector
space with a basis given by the classes $V^s(\eta\tilde{\eta})$ and
$dV^s(\eta\tilde{\eta})$, where $0 \leqslant s < n$. If $a^{(n)} =
V^s(\eta\tilde{\eta})$ with $0 \leqslant s < n$, then we let $\omega'
= \omega$. Then $(R-F)(\omega'{}^{(n+1)}) = \omega^{(n)}$, since
$F(\tilde{\eta}) = 0$. Next, suppose that $a^{(n)} =
dV^s(\tilde{\eta})$ with $1 \leqslant s < n$. Then
$$dV^s(\tilde{\eta})[x]^j = dV^s(\tilde{\eta}[x]^{2^sj}) -
jV^s(\tilde{\eta})[x]^jd\log[x].$$
and we have already considered the two terms on the right-hand
side. Hence, also in this case, there exists $\omega'$ such that
$(R-F)(\omega'{}^{(n+1)}) = \omega^{(n)}$. Similarly, in the remaining
case $\omega^{(n)} = (d\tilde{\eta})[x]^j$, the calculation
$$\begin{aligned}
(R-F)(dV(\tilde{\eta}[x]^j)) & = dV(\tilde{\eta}[x]^j) -
d(\tilde{\eta}[x]^j) - \eta\tilde{\eta}[x]^j \cr
{} & = dV(\tilde{\eta}[x]^j) - (d\tilde{\eta})[x]^j +
j\tilde{\eta}[x]^j d\log[x] - \eta\tilde{\eta}[x]^j \cr
\end{aligned}$$
shows that there exists $\omega'$ such that $(R-F)(\omega'{}^{(n+1)})
= \omega^{(n)}$. Indeed, we have already considered
$dV(\tilde{\eta}[x]^j)$, $\tilde{\eta}[x]^jd\log[x]$, and
$\eta\tilde{\eta}[x]^j$.

Finally, we can write every element $\omega = (\omega^{(n)})$ of
$\widetilde{\operatorname{TR}}_2(\mathbb{S}[x]^{\pm1},I[x^{\pm1}];2)$
as a series $\omega = \sum_{i \in I} \omega_i$,
where each $\omega_i$ is an element of the one of the four
types considered above, and where, for every $n \geqslant 1$, all but
finitely many of the $\smash{ \omega_i^{(n)} }$ are zero. Now, for
every $i \in I$, we have constructed an element $\smash{ \omega_i' =
(\omega_i'{}^{(n)}) }$ such that $(R-F)(\omega_i') = \omega_i$. Moreover,
the element $\omega_i'$ has the property that, if $\smash{
  \omega_i^{(n)} = 0 }$, then also $\smash{ \omega_i'{}^{(n)} = 0}$.
It follows that, for all $n \geqslant 1$, all but finitely many of the
$\omega_i'{}^{(n)}$. Hence, the series $\omega' = \sum_{i \in
  I}\omega_i'$ defines an element with $(R-F)(\omega') = \omega$ as
desired.
\qed
\end{proof}

\begin{theorem}\label{whiteheadtwo}There is a canonical isomorphism
$$\operatorname{Wh}_2^{\operatorname{Top}}(S^1) \xrightarrow{\sim}
\bigoplus_{r \geqslant 1} \bigoplus_{j \in \mathbb{Z} \smallsetminus
  2\mathbb{Z}} \mathbb{Z}/2\mathbb{Z}.$$
\end{theorem}

\begin{proof}We first evaluate the kernel of the map $1-F$ in the
long-exact sequence at the beginning of the section. Let $\omega =
(\omega^{(n)})$ be an element of
$\widetilde{\operatorname{TR}}_2(\mathbb{S}[x^{\pm1}],I[x^{\pm1}];2)$. 
Then $\omega$ lies in the kernel of $1-F$ if and only if the
coefficients
$$\begin{aligned}
a_{s,j}^{(n)} & = a_{s,j}(\omega^{(n)}) \in
\operatorname{TR}_2^{n-s}(\mathbb{S},I;2) \cr
b_{s,j}^{(n)} & = b_{s,j}(\omega^{(n)}) \in
\operatorname{TR}_1^{n-s}(\mathbb{S},I;2) \cr
\end{aligned}$$
satisfy the equations of Cor.~\ref{kernelR-F}. In the case at hand, the
equations imply that the coefficients above are determined by the
coefficients $\smash{ b_{1,j}^{(n)} }$. Indeed, if we write $j$ as
$2^u j'$ with $j'$ odd, then we have
$$\begin{aligned}
a_{s,j}^{(n)} & = \begin{cases}
F^u(db_{1,j'}^{(n+1+u)} + \eta b_{1,j'}^{(n+1+u)}) & (s = 0) \cr
\eta b_{1,j}^{(n+1-s)} & (1 \leqslant s < n) \cr
\end{cases} \cr
\end{aligned}$$
$$\begin{aligned}
b_{s,j}^{(n)} & = \begin{cases}
0 & \hskip25.3mm (\text{$s = 0$ and $j$ even}) \cr
-jb_{1,j}^{(n+1)} & \hskip25.3mm (\text{$s = 0$ and $j$ odd}) \cr
b_{1,j}^{(n+1-s)} & \hskip25.3mm (1 \leqslant s < n). \cr
\end{cases} \cr
\end{aligned}$$
The coefficients $b_{1,j}^{(n)}$, however, are not unrestricted, since
for every $n \geqslant 1$, all but finitely many of the coefficients
$\smash{a_{s,j}^{(n)}}$ and $\smash{b_{s,j}^{(n)}}$ are
zero. We write 
$$b_{1,j}^{(n)} = \sum_{0 \leqslant r < n-1} c_{r,j}
V^r(\tilde{\eta})$$
and consider the coefficients
$$c_{r,j} = c_{r,j}(\omega) \in \mathbb{Z}/2\mathbb{Z}.$$
Since $R(b_{1,j}^{(n+1)}) = b_{1,j}^{(n)}$ and $R(\tilde{\eta}) =
\tilde{\eta}$, the coefficients $c_{r,j}$ depend only on the integers
$r \geqslant 0$ and $j \in \mathbb{Z} \smallsetminus \mathbb{Z}$ and
not on $n$. They determine and are determined by the coefficients
$\smash{ a_{s,j}^{(n)} }$ and $\smash{ b_{s,j}^{(n)} }$. 

The requirement that for all $n \geqslant 1$, all but finitely many of
the $\smash{b_{s,j}^{(n)}}$ be zero implies that there exists a finite
subset $I=I(\omega) \subset \mathbb{Z}/2\mathbb{Z}$ such that
$c_{r,j}$ is zero, unless $j \in I$. We fix $j \in I$ and consider
$\smash{ a_{0,2^uj}^{(n)} }$, with $u \geqslant 0$. We calculate
$$\begin{aligned}
a_{0,2^uj}^{(n)} & = F^u(db_j^{(n+1+u)} + \eta b_j^{(n+1+u)}) \cr
{} & = \sum_{0 \leqslant r < n + u} 
c_{r,j} F^u(dV^r(\tilde{\eta}) + V^r(\eta\tilde{\eta})) \cr
{} & = \sum_{0 \leqslant r < u} c_{r,j}
F^{u-r}(d\tilde{\eta} + \eta\tilde{\eta}) \hskip2mm +
\sum_{u \leqslant r < u + n} c_{r,j}(dV^{r-u}(\tilde{\eta}) +
V^{r-u}(\eta\tilde{\eta})) \cr
{} & = \sum_{0 \leqslant r < u}
c_{r,j} d\tilde{\eta} \hskip2mm + \sum_{u \leqslant r < u + n}
c_{r,j}(dV^{r-u}(\tilde{\eta}) + V^{r-u}(\eta\tilde{\eta})). \cr
\end{aligned}$$
Now, for all $n \geqslant 1$, there exists $N^{(n)} = N^{(n)}(\omega)$
such that for all $j \in I$ and all $\smash{ u \geqslant N^{(n)}}$,
the coefficient $\smash{a_{0,2^uj}^{(n)}}$ is zero. We assume that
$\smash{N^{(n)}}$ is chosen minimal. Since
$$R \colon \operatorname{TR}_2^n(\mathbb{S},I;2) \to
\operatorname{TR}_2^{n-1}(\mathbb{S},I;2)$$
is surjective and takes $a_{0,2^uj}^{(n)}$ to $a_{0,2^uj}^{(n-1)}$, we
have $N^{(n)} \geqslant N^{(n-1)}$. Considering the coefficients of
$d\tilde{\eta}$ and $\eta\tilde{\eta}$ in the sum above, we find that
for all $\smash{u \geqslant N^{(n)}}$,
$$\begin{aligned}
\sum_{0 \leqslant r < u + 1} c_{r,j} \hskip2mm & = \hskip2mm 0 \hskip8mm
(\text{coefficient of $d\tilde{\eta}$}) \cr
c_{u,j} \hskip2mm  & = \hskip2mm 0 \hskip8mm 
(\text{coefficients of $\eta\tilde{\eta}$}) \cr
\end{aligned}$$
But these equations are satisfied also for $\smash{ u \geqslant
N^{(n-1)}}$ which implies that we also have $N^{(n)} \leqslant
N^{(n-1)}$. We conclude that there exists an integer $N = N(\omega)
\geqslant 0$ independent of $n$ such that $c_{u,j} = 0$, for $\smash{u
\geqslant N}$, and that the coefficient $c_{0,j}$ is equal to the sum
of the coefficients $c_{r,j}$ with $r \geqslant 1$. Conversely,
suppose we are given coefficients $c_{r,j}$ all but finitely many of
which are zero. Then, for every $n \geqslant 1$, all but finitely many
of the corresponding coefficients $\smash{ a_{s,j}^{(n)} }$ and 
$\smash{ b_{s,j}^{(n)} }$ are zero. This shows that the map
$$\operatorname{ker}(
1 - F \colon
\widetilde{\operatorname{TR}}_2(\mathbb{S}[x^{\pm 1}], I[x^{\pm 1}];2) \to
\widetilde{\operatorname{TR}}_2(\mathbb{S}[x^{\pm 1}], I[x^{\pm 1}];2)) \to
\bigoplus_{r \geqslant 1} 
\bigoplus_{j \in \mathbb{Z}\smallsetminus\mathbb{Z}} 
\mathbb{Z}/2\mathbb{Z}$$
that to $\omega$ assigns $(c_{r,j}(\omega))$ is an isomorphism.

It remains to show that the map
$$1 - F \colon 
\widetilde{\operatorname{TR}}_3(\mathbb{S}[x^{\pm 1}], I[x^{\pm 1}];2) \to
\widetilde{\operatorname{TR}}_3(\mathbb{S}[x^{\pm 1}], I[x^{\pm 1}];2)$$
is surjective. Given the element $\omega = (\omega^{(n)})$ on the
right-hand side, we find an element 
$\smash{ \omega' = (\omega'{}^{(n)}) }$ on the left-hand side such
that $\smash{ (R-F)(\omega') = \omega }$. As in the proof of
Thm.~\ref{whiteheadzeroandone}, we first consider several cases
seperately.

First, if $\omega^{(n)} = dV^s(b^{(n)}[x]^j)$, where $j$ and $s
\geqslant 1$ are integers, we define
$$\omega'{}^{(n)} = 
- \!\! \sum_{s \leqslant r < n-1} dV^{r+1}(b^{(n-1-r+s)}[x]^j) 
- \!\! \sum_{s \leqslant r < n} V^r(\eta b^{(n-r+s)}[x]^j).$$
Then we find that $R(\omega'{}^{(n+1)}) = \omega'{}^{(n)}$ and
$(R-F)(\omega'{}^{(n+1)}) = \omega^{(n)}$ by calculations entirely
similar to the ones in the proof of Thm.~\ref{whiteheadzeroandone}.

Second, if $\omega^{(n)} = b^{(n)}[x]^jd\log[x]$, we consider three
cases separately. In the case $\omega^{(n)} =
V^s(\eta\tilde{\eta})[x]^jd\log[x]$ with $0 \leqslant s < n$, we let
$\omega' = \omega$. Then $(R-F)(\omega') = \omega$ since
$F(\tilde{\eta}) = 0$. In the case
$\omega^{(n)} = dV^s(\tilde{\eta})[x]^jd\log[x]$, where $1 \leqslant s
< n$, we note that
$\omega^{(n)} = dV^s(\tilde{\eta}[x]^{2^sj}d\log[x])$ and define
$$\omega'{}^{(n)} = 
- \!\! \sum_{s \leqslant r < n-1} dV^{r+1}(\tilde{\eta}[x]^{2^sj}d\log[x]) 
- \!\! \sum_{s \leqslant r < n} V^r(\eta\tilde{\eta}[x]^{2^sj}d\log[x]).$$
Then $R(\omega'{}^{(n+1)}) = \omega'{}^{(n)}$ and
$(R-F)(\omega'{}^{(n+1)}) = \omega^{(n)}$ as before. In the remaining
case $\omega^{(n)} = (d\tilde{\eta})[x]^jd\log[x]$, the calculation
$$(R-F)(dV(\tilde{\eta}[x]^jd\log[x])) = 
dV(\tilde{\eta}[x]^jd\log[x]) - (d\tilde{\eta})[x]^jd\log[x] -
\eta\tilde{\eta}[x]^jd\log[x]$$
shows that there exists $\omega'$ with $(R-F)(\omega') =
\omega$. Indeed, we have already considered
$dV(\tilde{\eta}[x]^jd\log[x])$ and $\eta\tilde{\eta}[x]^jd\log[x]$.

Third, if $\omega^{(n)} = a^{(n)}[x]^j$, we consider two cases
separately. In the first case, we have $\omega^{(n)} =
V^s(\tilde{\nu})[x]^j$ with $0  \leqslant s < n$ and define
$$\omega'{}^{(n)} =  V^s(\tilde{\nu})[x]^j + F(V^s(\tilde{\nu})[x]^j)
+ F^2(V^s(\tilde{\nu})[x]^j).$$
Then $R(\omega'{}^{(n+1)}) = \omega'{}^{(n)}$ and
$(R-F)(\omega'{}^{(n+1)}) = \omega^{(n)}$ because $F^3V^s(\tilde{\nu})
= 0$. In the second case, $\omega^{(n)} =
dV^s(\eta\tilde{\eta})[x]^j$, we calculate
$$dV^s(\eta\tilde{\eta})[x]^j = dV^s(\eta\tilde{\eta}[x]^{2^sj}) -
jV^s(\eta\tilde{\eta})[x]^jd\log[x].$$
Since we have already considered the two terms on the right-hand side,
it follows that there exists $\omega'$ with $(R-F)(\omega') = \omega$.

Finally, we consider $\omega^{(n)} = V^s(a^{(n)}[x]^j)$ with $1
\leqslant s < n$. For $s \geq 3$, we define
$$\omega'{}^{(n)} =  V^s(a^{(n)}[x]^j) + FV^s(a^{(n+1)}[x]^j)
+ F^2V^s(a^{(n+2)}[x]^j).$$
Then $R(\omega'^{(n+1)}) = \omega'{}^{(n)}$ and
$(R-F)(\omega'{}^{(n+1)}) = \omega^{(n)}$ since $8a^{(n+3)} = 0$. For
$s=0$ and $s=1$, the calculation
$$\begin{aligned}
(R-F)(V(a^{(n+1)}[x]^j)) & = V(a^{(n)}[x]^j) - 2a^{(n+1)}[x]^j \cr
(R-F)(V^2(a^{(n+1)}[x]^j) + FV^2(a^{(n+2)}[x]^j)) & = V^2(a^{(n)}[x]^j) -
4a^{(n+2)}[x]^j, \cr
\end{aligned}$$
shows that there exists $\omega'$ with $(R-F)(\omega') =
\omega$. Indeed, we have already considered $2a^{(n+1)}[x]^j$ and
$4a^{(n+2)}[x]^j$ above.

The elements $\omega'$ with $(R-F)(\omega') = \omega$ which we
constructed above have the property that, if $\smash{ \omega^{(n)} }$
is zero, then $\smash{ \omega'{}^{(n)} }$ is zero. It follows as in
the proof of Thm.~\ref{whiteheadzeroandone} that the map $1-F$ in
question is surjective. 
\qed
\end{proof}

\begin{theorem}\label{whiteheadthree}There is a canonical isomorphism
$$\operatorname{Wh}_3^{\operatorname{Top}}(S^1) \xrightarrow{\sim}
\bigoplus_{r \geqslant 0} \bigoplus_{j \in \mathbb{Z} \smallsetminus
  2\mathbb{Z}} \mathbb{Z}/2\mathbb{Z} \; \oplus \;
\bigoplus_{r \geqslant 1} \bigoplus_{j \in \mathbb{Z} \smallsetminus
  2\mathbb{Z}} \mathbb{Z}/2\mathbb{Z}.$$
\end{theorem}

\begin{proof}We first show that the kernel of the map $1-F$ in the
long-exact sequence at the beginning of the section is canonically
isomorphic to the group that appears on the right-hand side in the
statement. So we let $\smash{ \omega = (\omega^{(n)}) }$ be an element
of $\smash{
  \widetilde{\operatorname{TR}}_3(\mathbb{S}[x^{\pm1}],I[x^{\pm1}];2)
}$ that lies in the kernel of $1-F$. The equations of
Cor.~\ref{kernelR-F} again show that the coefficients 
$$\begin{aligned}
a_{s,j}^{(n)} & = a_{s,j}(\omega^{(n)}) \in
\operatorname{TR}_3^{n-s}(\mathbb{S},I;p) \cr
b_{s,j}^{(n)} & = b_{s,j}(\omega^{(n)}) \in
\operatorname{TR}_2^{n-s}(\mathbb{S},I;p) \cr
\end{aligned}$$
are completely determined by the coefficients $b_{1,j}^{(n)}$. Indeed,
we find
$$\begin{aligned}
a_{s,j}^{(n)} & = \begin{cases}
F^u(db_{1,j'}^{(n+1+u)} + \eta b_{1,j'}^{(n+1+u)}) & (s = 0) \cr
\eta b_{1,j}^{(n+1-s)} & (1 \leqslant s < n) \cr
\end{cases} \cr
b_{s,j}^{(n)} & = \begin{cases}
jF^u(b_{1,j'}^{(n+1+u)}) & \hskip18.9mm (s = 0) \cr
b_{1,j}^{(n+1-s)} & \hskip18.9mm (1 \leqslant s < n). \cr
\end{cases} \cr
\end{aligned}$$
where $j = 2^u j'$ with $j'$ odd. For example, if $1 \leqslant s < n$, then 
$$\begin{aligned}
a_{s,j}^{(n)} & = 2a_{s+1,j}^{(n+1)} + \eta b_{s+1,j}^{(n+1)} 
= 2(2a_{s+2,j}^{(n+2)} + \eta  b_{s+2,j}^{(n+2)}) + \eta
b_{s+1,j}^{(n+1)} \cr
{} & = 2(2(2a_{s+3,j}^{(n+3)} + \eta b_{s+3,j}^{(n+3)}) + \eta
b_{s+2,j}^{(n+2)}) + \eta b_{s+1,j}^{(n+1)} \cr
{} & = \eta b_{s+1,j}^{(n+1)} = \eta b_{1,j}^{(n+1-s)} \cr
\end{aligned}$$
since $\operatorname{TR}_3^{n-s}(\mathbb{S},I;2)$ is annihilated by
$8$. We now write
$$b_{1,j}^{(n)}
= \sum_{0 \leqslant r < n-1} c_{r,j} V^r(\eta\tilde{\eta})
+ \sum_{0 \leqslant r < n-1} c_{r,j}' dV^r(\tilde{\eta}),$$
where the coefficients $c_{r,j} = c_{r,j}(\omega)$ and $c_{r,j}' =
c_{r,j}'(\omega)$ 
are independent on $n$. It is clear that the $c_{r,j}$ and $c_{r,j}'$
are non-zero for only finitely many values of the odd integer $j$. We
fix such a $j$ and evaluate the coefficients $\smash{ a_{0,2^uj}^{(n)} }$
and $\smash{ b_{0,2^uj}^{(n)} }$ for $u \geqslant 1$ as functions of
the coefficients $c_{r,j}$ and $c_{r,j}'$.
$$\begin{aligned}
a_{0,2^uj}^{(n)} & = F^u(db_j^{(n+1+u)} + \eta b_j^{(n+1+u)}) \cr
{} & = \hskip2mm \sum_{0 \leqslant r < n+u}
c_{r,j}(F^udV^r(\eta\tilde{\eta}) + \eta F^uV^r(\eta\tilde{\eta})) \cr
{} & \hskip6mm + \sum_{0 \leqslant r < n+u}
c_{r,j}'(F^uddV^r(\tilde{\eta}) + \eta F^udV^r(\tilde{\eta})) \cr
{} & = \hskip2mm \sum_{u \leqslant r < n+u}
c_{r,j}(dV^{r-u}(\eta\tilde{\eta}) + V^{r-u}(\eta^2\tilde{\eta})) \cr
b_{0,2^uj}^{(n)} & = jF^u(b_j^{(n+1+u)}) \cr
{} & = \hskip2mm \sum_{0 \leqslant r < n+u}
jc_{r,j}F^uV^r(\eta\tilde{\eta}) \hskip2mm +
\sum_{0 \leqslant r < n+u}
jc_{r,j}'F^udV^r(\tilde{\eta}) \cr
{} & = \hskip2mm \sum_{0 \leqslant r < u} jc_{r,j}'d\tilde{\eta}
\hskip2mm + \sum_{u \leqslant r < n+u}
jc_{r,j}'(dV^{r-u}(\tilde{\eta}) + V^{r-u}(\eta\tilde{\eta})), \cr
\end{aligned}$$
We claim that the elements $dV^{r-u}(\eta\tilde{\eta})$ and
$V^{r-u}(\eta^2\tilde{\eta})$ with $u < r < n+u$ form a linearly 
independent set. Indeed, the map
$$i_* \colon \operatorname{TR}_3^{n+u}(\mathbb{S},I;2) \to
\operatorname{TR}_3^{n+u}(\mathbb{S};2)$$
is injective by Prop.~\ref{evenzero}, and Lemma~\ref{imageofeta} shows
that
$$\begin{aligned}
i_*(dV^{r-u}(\eta\tilde{\eta}) & = dV^{r-u}(\eta^2) +
dV^{r-u+1}(\eta^2) \cr
i_*(V^{r-u}(\eta^2\tilde{\eta}) & = V^{r-u}(\eta^3) +
V^{r-u+1}(\eta^3) = 4V^{r-u}(\nu) + 4V^{r-u+1}(\nu). \cr
\end{aligned}$$
The claim then follows from Thm.~\ref{TRsphere}. We now conclude as in
the proof of Thm.~\ref{whiteheadtwo} that the map that to $\omega$
assigns $((c_{r,j}(\omega)),(c_{r,j}'(\omega)))$ defines an isomorphism
$$\begin{aligned}
\operatorname{ker} & (
1 - F \colon
\widetilde{\operatorname{TR}}_3(\mathbb{S}[x^{\pm 1}], I[x^{\pm1}];2) \to
\widetilde{\operatorname{TR}}_3(\mathbb{S}[x^{\pm 1}], I[x^{\pm1}];2)) \cr
{} & \xrightarrow{\sim} \bigoplus_{r \geqslant 0}
\bigoplus_{j \in \mathbb{Z}\smallsetminus\mathbb{Z}} 
\mathbb{Z}/2\mathbb{Z} \; \oplus \;
\bigoplus_{r \geqslant 1} 
\bigoplus_{j \in \mathbb{Z}\smallsetminus\mathbb{Z}} 
\mathbb{Z}/2\mathbb{Z} \cr
\end{aligned}$$

Finally, we argue as in the proof of Thm.~\ref{whiteheadzeroandone}
that the map
$$1 - F \colon
\widetilde{\operatorname{TR}}_4(\mathbb{S}[x^{\pm 1}], I[x^{\pm 1}];2) \to
\widetilde{\operatorname{TR}}_4(\mathbb{S}[x^{\pm 1}], I[x^{\pm 1}];2)$$
is surjective. Given $\omega = (\omega^{(n)})$ on the right-hand side,
we find $\omega' = (\omega'{}^{(n)})$ on the left-hand side with
$(R-F)(\omega') = \omega$. 

First, if $\omega^{(n)} = dV^s(b^{(n)}[x]^j)$, where $1 \leqslant s <
n$ and $j$ are integers, we define
$$\omega'{}^{(n)} = 
- \!\! \sum_{s \leqslant r < n-1} dV^{r+1}(b^{(n-1-r+s)}[x]^j) 
- \!\! \sum_{s \leqslant r < n} V^r(\eta b^{(n-r+s)}[x]^j).$$
Then we have $R(\omega'{}^{(n+1)}) = \omega'{}^{(n)}$ and
$(R-F)(\omega'{}^{(n+1)}) = \omega^{(n)}$ as desired.

Second, if $\omega^{(n)} = b^{(n)}[x]^jd\log[x]$, we consider two
cases separately. In the case $\omega^{(n)} =
dV^s(\eta\tilde{\eta})[x]^jd\log[x]$ with $1 \leqslant s < n$, we write
$\omega^{(n)} = dV^s(\eta\tilde{\eta}[x]^{2^sj})d\log[x]$ and
define 
$$\omega'{}^{(n)} = 
- \!\! \sum_{s \leqslant r < n-1} dV^{r+1}(\eta\tilde{\eta}[x]^jd\log[x]) 
- \!\! \sum_{s \leqslant r < n} V^r(\eta^2\tilde{\eta}[x]^jd\log[x]).$$
Then $R(\omega'{}^{(n+1)}) = \omega'{}^{(n)}$ and
$(R-F)(\omega'{}^{(n+1)}) = \omega^{(n)}$ as before. In the case where
$\omega^{(n)} = V^s(\tilde{\nu})[x]^jd\log[x]$ with $0 \leqslant s <
n$, we define
$$\omega'{}^{(n)} = 
V^s(\tilde{\nu})[x]^jd\log[x] +
F(V^s(\tilde{\nu})[x]^jd\log[x]) +
F^2(V^s(\tilde{\nu})[x]^jd\log[x]).$$
Then $R(\omega'{}^{(n+1)}) = \omega'{}^{(n)}$ and
$(R-F)(\omega'{}^{(n+1)}) = \omega^{(n)}$ since $8\tilde{\nu}$ and
$F\tilde{\nu}$ are zero.

Finally, we consider $\omega^{(n)} = dV^s(\tilde{\nu})[x]^j$ with
$0 \leqslant s < n$. For $s \geqslant 1$,
$$dV^s(\tilde{\nu})[x]^j = dV^s(\tilde{\nu}[x]^{2^sj}) -
jV^s(\tilde{\nu})[x]^jd\log[x]$$
and the two terms on the right-hand side were considered above. 
It follows that there exists $\omega'$ with $(R-F)(\omega') =
\omega$. For $s = 0$, we calculate 
$$\begin{aligned}
(R-F)(dV(\tilde{\nu}[x]^j)) & = dV(\tilde{\nu}[x]^j) -
(d\tilde{\nu})[x]^j + j\tilde{\nu}[x]^jd\log[x] -
\eta\tilde{\nu}[x]^j \cr
(R-F)(\eta\tilde{\nu}[x]^j) & = \eta\tilde{\nu}[x]^j. \cr
\end{aligned}$$
This shows that also for $\omega^{(n)} = (d\tilde{\nu})[x]^j$, there
exists $\omega'$ such that $(R-F)(\omega') = \omega$. Indeed, we have
already considered the remaining classes on the right-hand side. 

The elements $\omega'$ with $(R-F)(\omega') = \omega$ which we
constructed above have the property that, if $\smash{ \omega^{(n)} }$
is zero, then $\smash{ \omega'{}^{(n)} }$ is zero. It follows as in
the proof of Thm.~\ref{whiteheadzeroandone} that the map $1-F$ in
question is surjective. This completes the proof.
\qed
\end{proof}

\providecommand{\bysame}{\leavevmode\hbox to3em{\hrulefill}\thinspace}
\providecommand{\MR}{\relax\ifhmode\unskip\space\fi MR }
\providecommand{\MRhref}[2]{%
  \href{http://www.ams.org/mathscinet-getitem?mr=#1}{#2}
}
\providecommand{\href}[2]{#2}

\end{document}